\documentclass[11pt]{amsart}
\usepackage{graphicx}
\usepackage[pagebackref, colorlinks=true, linkcolor=black, citecolor=black, urlcolor=black]{hyperref}

\usepackage{xargs}                      % Use more than one optional parameter in a new commands
\usepackage[colorinlistoftodos,prependcaption %,textsize=tiny
]{todonotes}                                                                                                                 %todo notes!

\newcommandx{\todoN}[2][1=]{\todo[linecolor=red,backgroundcolor=white,bordercolor=red,#1]{#2}}			%Nicola's todo notes in Red

\usepackage{comment}
\usepackage{calc}

\usepackage{amsmath, amsthm, amsfonts, amssymb}
\usepackage{enumerate}
\usepackage{bbm}
\usepackage{mathrsfs}
\usepackage{amssymb}
\usepackage{mathtools}
\usepackage[all]{xy}
\usepackage{tikz}
\usepackage{tikz-cd}

\usetikzlibrary{matrix,arrows,decorations.pathmorphing}
\usepackage{xcolor}
\usepackage{bm}
	\usetikzlibrary{calc}
	\usetikzlibrary{trees}
	\tikzstyle{every picture}=[scale=.35,inner sep=0]
	\usepackage{color}

\usepackage{stmaryrd} %%for \llbracket == [[, \rrbracket == ]], \llparenthesis == ((, \rrparenthesis == ))

\usepackage[normalem]{ulem} % for strike through with command \sout{...}

\newtheorem{thm}{Theorem}

\theoremstyle{definition}

\theoremstyle{theorem}
\newtheorem{theorem}{Theorem}[section]
\newtheorem{lemma}[theorem]{Lemma}
\newtheorem{proposition}[theorem]{Proposition}

\theoremstyle{definition}
\newtheorem{definition}[theorem]{Definition}

\theoremstyle{remark}
\newtheorem{remark}[theorem]{Remark}

\newtheorem{example}[theorem]{Example}

\numberwithin{equation}{section}

\makeatletter
\@namedef{subjclassname@2020}{\textup{MSC2020}}
\makeatother

\begin{document}

\title[Coinvariants of metaplectic representations]{Coinvariants of metaplectic representations\\on moduli of abelian varieties}

\author[N.~Tarasca]{Nicola Tarasca}
\address{Nicola Tarasca 
\newline \indent Department of Mathematics \& Statistics
\newline \indent Virginia Commonwealth University, Richmond, VA 23284}
\email{tarascan@vcu.edu}


\begin{abstract}
We construct spaces of coinvariants at principally polarized abelian varieties with respect to the action of 
an infinite-dimensional Lie algebra.
We show how these spaces  globalize to twisted $\mathscr{D}$-modules  on moduli of principally polarized abelian varieties, and we determine  
the Atiyah algebra of a line bundle acting on them. We prove  analogous results  on the universal abelian variety.
An essential aspect of our study involves analyzing the Atiyah algebra of the Hodge and canonical line bundles on moduli  of abelian varieties and the universal abelian variety.
\end{abstract}

\vspace*{-1.5pc}

\maketitle

\vspace{-1.5pc}

\section*{Introduction}

%%%%%%%%%%%%%%%%%%%%%%%%%%%%
%%%%%%%%%%%%%%%%%%%%%%%%%%%%
%%%%%%%%%%%%%%%%%%%%%%%%%%%%

Spaces of coinvariants have classically been constructed by assigning representations of affine Lie algebras, and more generally,  vertex operator algebras, to pointed algebraic  curves \cite{tuy, bzf}. These spaces  globalize to quasi-coherent sheaves 
equipped with a twisted $\mathscr{D}$-module structure on moduli of  curves. Said structure is induced by an equivariant action of the \mbox{Virasoro} algebra on vertex operator algebra modules and on moduli  of curves with marked points and formal coordinates.
The construction generalizes to produce  twisted $\mathscr{D}$-modules on moduli spaces parametrizing pointed curves with extra features: line bundles, principal bundles, and $\mathscr{D}$-bundles \cite{frenkel2004twisted, bzf, ben2010mathcal}.

Removing curves out of the picture, 
Arbarello--De Concini \cite{adc91} 
constructed a moduli space $\widehat{\mathcal{A}}_g$ of suitable extensions of principally polarized abelian varieties which includes, via an extended Torelli map,
the moduli space $\widehat{\mathcal{M}}_g$ of algebraic curves with a marked point and a formal coordinate. They showed that an infinite-dimensional symplectic algebra $\mathfrak{sp}\left(H' \right)$ acts transitively on $\widehat{\mathcal{A}}_g$, extending the transitive action of the Witt algebra on $\widehat{\mathcal{M}}_g$.
Here, $H'$ is the vector space of Laurent series in a formal variable with zero constant term and carries a symplectic form  induced from the  infinite-dimensional Heisenberg algebra.
The oscillator representation which injects the Witt algebra into $\mathfrak{sp}\left(H' \right)$ thus appears as the differential of the extended Torelli map. 

The present project started from the realization that the results of \cite{adc91} offer an opening for an extension of the classical conformal field theory from  moduli spaces of curves to  moduli spaces of abelian varieties.
Removing curves out of the construction of coinvariants, we thus build and study twisted $\mathscr{D}$-modules on moduli spaces of abelian varieties.

Our construction proceeds as follows.
The symplectic algebra $\mathfrak{sp}\left(H' \right)$ admits a unique (up to isomorphism) non-trivial central extension, called the \textit{metaplectic algebra} $\mathfrak{mp}\left(H' \right)$ (this is reviewed in \S\ref{sec:symp&meta} and the uniqueness follows from Proposition \ref{prop:cohLiespH}).
For a representation $V$ of $\mathfrak{mp}\left(H' \right)$, we define the space of coinvariants of $V$ at a point $a\in \widehat{\mathcal{A}}_g$ as
\begin{equation}
\label{eq:coinvonAtilde}
\widehat{\mathbb{V}}(V)_{a} := V \, / \,\left(\mathfrak{sp}_F\left(H' \right) \cdot V\right),
\end{equation}
where $\mathfrak{sp}_F\left(H' \right)$ is the stabilizer of $\mathfrak{sp}\left(H' \right)$ at $a$---see \eqref{eq:spF}---and the action of $\mathfrak{sp}_F\left(H' \right)$ on $V$ is given by a splitting $\mathfrak{sp}_F\left(H' \right)\subset \mathfrak{mp}\left(H' \right)$---see Proposition \ref{prop:splitting2}(i).
The space \eqref{eq:coinvonAtilde} is the largest quotient of $V$ on which $\mathfrak{sp}_F\left(H' \right)$ acts trivially.
We show that the spaces \eqref{eq:coinvonAtilde} globalize to a sheaf $\widehat{\mathbb{V}}(V)$ on $\widehat{\mathcal{A}}_g$ equipped with a twisted $\mathscr{D}$-module structure (Theorem \ref{thm:AtiyahactingonVhatA}). Furthermore, in case $V$ is an admissible $\mathfrak{mp}\left(H' \right)$--representation, as per  Definition \ref{def:Heisenberg--admissible}, we show that the sheaf $\widehat{\mathbb{V}}(V)$ descends to the moduli space $\mathcal{A}_g$ of principally polarized, $g$-dimensional abelian varieties:

\begin{thm}
\label{thm:maininitVA}
For an admissible $\mathfrak{mp}\left(H' \right)$--representation $V$, the spaces of coinvariants \eqref{eq:coinvonAtilde} give rise to a twisted $\mathscr{D}$-module $\mathbb{V}(V)$ on $\mathcal{A}_g$.
\end{thm}

For instance, the theorem applies to Heisenberg vertex algebras of arbitrary rank and lattice vertex (super)-algebras for nondegenerate lattices (Lemma \ref{lemma:exadm}).
The metaplectic algebra $\mathfrak{mp}\left(H' \right)$ acts on $\widehat{\mathcal{A}}_g$ via the projection $\mathfrak{mp}\left(H' \right)\rightarrow \mathfrak{sp}\left(H' \right)$ 
and replaces the Virasoro algebra of the classical conformal field theory. The action of $\mathfrak{mp}\left(H' \right)$ on $V$ induces the twisted $\mathscr{D}$-module structure on $\mathbb{V}(V)$. To explicitly determine it, we use the formalism of Atiyah algebras from Beilinson--Schechtman \cite{besh}.  Let $\Lambda$ be the Hodge line bundle on $\mathcal{A}_g$, and $\mathscr{F}_\Lambda$  the Atiyah algebra of $\Lambda$, i.e., the sheaf of first-order differential operators acting on $\Lambda$. We say that an admissible $\mathfrak{mp}\left(H' \right)$--representation $V$ is of central charge $c$ if the action of $\mathfrak{mp}\left(H' \right)$ induces an action of the Virasoro algebra on $V$ of central charge~$c$---see Definition \ref{def:centralcharge}.

\begin{thm}
\label{thm:mainVA}
For an admissible $\mathfrak{mp}\left(H' \right)$--representation  $V$ of central charge $c$, the sheaf $\mathbb{V}(V)$ on $\mathcal{A}_g$ carries an action of  the Atiyah algebra $\frac{c}{2}\,\mathscr{F}_\Lambda$. This action induces the twisted $\mathscr{D}$-module structure on $\mathbb{V}(V)$.
\end{thm}

Moreover, we extend these results over the universal abelian variety.
It is shown in \cite{adc91} that $\widehat{\mathcal{A}}_g$ admits a universal family $\widehat{\mathcal{X}}_g\rightarrow \widehat{\mathcal{A}}_g$. 
This extends a family $\mathcal X_g\to\mathcal A_g$, which is a torsor under the
universal abelian variety over $\mathcal A_g$ and becomes isomorphic to it after a
finite \'etale base change (see \S\ref{sec:extunivPPAV}). By abuse of terminology, we
refer to $\mathcal X_g$ as the universal abelian variety.
The space $\widehat{\mathcal{X}}_g$ carries a transitive action of the Lie algebra $\mathfrak{sp}\left(H' \right)\ltimes H'$.
The terms of degree at most two in the Weyl algebra $\widetilde{\mathscr{U}}(H)$ give a non-trivial central extension $\widetilde{\mathscr{U}}_{2}(H)$ of $\mathfrak{sp}\left(H' \right)\ltimes H'$ (\S\ref{sec:Weyl}).
For a representation  $V$ of $\widetilde{\mathscr{U}}_{2}(H)$,
we define the space of coinvariants of $V$ at $x \in \widehat{\mathcal{X}}_g$ as
\begin{equation}
\label{eq:coinvonXtilde}
\widehat{\mathbb{V}}(V)_{x} := V \, / \, \left(\mathfrak{i}(x) \cdot V\right),
\end{equation}
where $\mathfrak{i}(x)$ is the stabilizer of $\mathfrak{sp}\left(H' \right)\ltimes H'$ at $x$---see \eqref{eq:ix}---and its action on $V$ is given by Proposition \ref{prop:splitting2}(ii).
We show that the spaces \eqref{eq:coinvonXtilde} globalize to a twisted $\mathscr{D}$-module on $\widehat{\mathcal{X}}_g$ (Theorem \ref{thm:AtiyahactingonVhatX}).
In case $V$ is an admissible $\widetilde{\mathscr{U}}_{2}(H)$--representation, as per  Definition \ref{def:Heisenberg--admissible}, we show:

\begin{thm}
\label{thm:maininitVX}
For an admissible $\widetilde{\mathscr{U}}_{2}(H)$--representation $V$, the spaces of coinvariants \eqref{eq:coinvonXtilde} give rise to a twisted $\mathscr{D}$-module $\mathbb{V}(V)$ on~$\mathcal{X}_g$.
\end{thm}

The twisted $\mathscr{D}$-module structure is explicitly determined by the action of an appropriate multiple of the Atiyah algebra $\mathscr{F}_\Xi$ of the canonical line bundle $\Xi=\mathscr{O}(2\theta)$ on~$\mathcal{X}_g$ (see \S\ref{sex:lambdatheta}):

\begin{thm}
\label{thm:mainVX}
For an admissible $\widetilde{\mathscr{U}}_{2}(H)$--representation $V$ of central charge $c$, the sheaf $\mathbb{V}(V)$ on $\mathcal{X}_g$ carries an action of  the Atiyah algebra $\frac{c}{2}\,\mathscr{F}_\Xi$. This action induces the twisted $\mathscr{D}$-module structure on $\mathbb{V}(V)$.
\end{thm}

To prove the above statements, we need to show some auxiliary results, which are of independent interest.
It is shown in \cite{adkp} that there exist canonical homomorphisms 
\begin{align*}
\overline{\nu} \colon H^2\left(\mathrm{Witt}, \mathbb{C} \right)&\rightarrow H^2\left( \mathcal{M}_g, \mathbb{C}\right), \\
\overline{\mu} \colon  H^2\left( \mathrm{Witt}\ltimes H', \mathbb{C} \right)  &\rightarrow H^2\left( \mathcal{P}'_{g-1}, \mathbb{C}\right),
\end{align*}
with $\overline{\nu}$ being an isomorphism for $g\geq 3$ \cite[(4.14)]{adkp}, and $\overline{\mu}$ an isomorphism for $g\geq 5$ \cite[(4.15)]{adkp}.
Here $\mathcal{M}_g$ is the moduli space of smooth genus $g$ curves, and $\mathcal{P}'_{g-1}$ is the moduli space of quadruples $(C,P,v, \mathscr{L})$, where
$C$ is a smooth genus $g$ curve, $P$ is a point of $C$, $v$ is a nonzero tangent vector to $C$ at $P$, and $\mathscr{L}$ is a degree $g-1$ line bundle on $C$. 
(Combining with subsequent results from \cite{melo2014picard}, $\overline{\mu}$ is an isomorphism for $g\geq 3$, see proof of Proposition \ref{prop:H2AXMP}.)
We show  that these maps extend to give:

\begin{thm}
\label{thm:canisointro}
For $g\geq 2$, there exist  canonical isomorphisms 
\begin{align*}
\nu \colon H^2\left(\mathfrak{sp} \left( H' \right), \mathbb{C}\right) &\rightarrow H^2\left( {\mathcal{A}}_g, \mathbb{C}\right), \\ 
\mu \colon H^2\left(\mathfrak{sp} \left( H' \right)\ltimes H', \mathbb{C} \right) &\rightarrow H^2\left( {\mathcal{X}}_g, \mathbb{C}\right).
\end{align*}
\end{thm}

Moreover, we obtain the following result. For simplicity, we omit the notation for the pull-back when referring to the pull-back of the line bundles $\Lambda$ and $\Xi$ via the natural projections $\widehat{\mathcal{A}}_g \rightarrow \mathcal{A}_g$
 and $\widehat{\mathcal{X}}_g \rightarrow \mathcal{X}_g$.

\begin{thm}
\label{thm:mpaction}
\begin{enumerate}[(i)]
\item The line bundle $\Lambda$ on $\widehat{\mathcal{A}}_g$ carries an action of $\mathfrak{mp} \left( H' \right)$ by first-order differential operators  extending the transitive action of \sloppy\mbox{$\mathfrak{sp} \left( H' \right)$} on  $\widehat{\mathcal{A}}_g$ and with the central element $\bm{1}\in \mathfrak{mp} \left( H' \right)$ acting as multiplication by $2$ on the fibers of $\Lambda\rightarrow \widehat{\mathcal{A}}_g$.

\item The  line bundle $\Xi$ on $\widehat{\mathcal{X}}_g$ carries an action of $\widetilde{\mathscr{U}}_{2}(H)$ by first-order differential operators  extending the transitive action of \sloppy\mbox{$\mathfrak{sp} \left( H' \right)\ltimes H'$} on  $\widehat{\mathcal{X}}_g$ and with the central element $\bm{1}\in \widetilde{\mathscr{U}}_{2}(H)$ acting as multiplication by $2$ on the fibers of $\Xi\rightarrow \widehat{\mathcal{X}}_g$.
\end{enumerate}
\end{thm}

Additionally, we study group ind-schemes of symplectic automorphisms in \S\ref{sec:groupind}, generalizing the group ind-schemes of ring automorphisms from \cite{beilinson1991quantization, bzf}. Our constructions of the sheaves of coinvariants on $\mathcal{A}_g$ and $\mathcal{X}_g$ are thus variations on the formalism of localization of modules over Harish--Chandra pairs (\S\ref{sec:geoHCp}).

The present results are intended to initiate the study of coinvariants on abelian varieties and their moduli. We mention some natural directions in \S\ref{sec:final}  and will return to these  in the near future.

\smallskip

The paper is structured as follows. We review the algebraic and geometric background in \S\S\ref{sec:HmpVir} and \ref{sec:Ahat}, respectively.
In \S\ref{sec:groupind} we define and study a group ind-scheme which acts transitively on the fibers of the natural map \mbox{$\widehat{\mathcal{A}}_g\rightarrow {\mathcal{A}}_g$,} together with related group ind-schemes. 
In \S\ref{sec:geoHCp} we review the Harish--Chandra localization on quotient stacks from Beilinson--Bernstein \cite{bb}.
In \S\ref{sec:cohLie} we start the study of the cohomology of the Lie algebras $\mathfrak{sp} \left( H' \right)$ and $\mathfrak{sp} \left( H' \right)\ltimes H'$. 
In \S\ref{sec:isoH2} we determine the maps $\nu$ and $\mu$ explicitly
(Theorem~\ref{thm:canisoH2spAX}), and we deduce Theorems~\ref{thm:canisointro}  and \ref{thm:mpaction}.

Next, an intermezzo in \S\ref{sec:intermezzo} contains some required properties of metaplectic representations.
We construct sheaves of coinvariants on $\widehat{\mathcal{A}}_g$ and $\widehat{\mathcal{X}}_g$ in \S\ref{sec:Vhat} and descend them to $\mathcal{A}_g$ and $\mathcal{X}_g$ in \S\ref{sec:sheafV}, where we prove Theorems \ref{thm:maininitVA}--\ref{thm:mainVX}. We conclude with some final remarks in \S\ref{sec:final}.

%%%%%%%%%%%%%%%%%%%%%%%%%%%%
%%%%%%%%%%%%%%%%%%%%%%%%%%%%
%%%%%%%%%%%%%%%%%%%%%%%%%%%%

\section{Heisenberg, metaplectic, and Virasoro algebras}
\label{sec:HmpVir}
We review here the infinite-dimensional Lie algebras from \cite{arbarello1987infinite, adc91} which will play a role in the paper. 
The complex Lie algebras from these references are here promoted to functors on commutative $\mathbb{C}$-algebras, as needed for our geometric constructions. 
Moreover, emphasis is given to the two-cocycles defining the relevant central extensions, later needed in \S\S\ref{sec:cohLie}--\ref{sec:isoH2}.

Here is some notation used throughout. For a topological vector space $V$, a subset $B\subset V$ is said to \textit{topologically generate} $V$ if finite linear combinations of elements of $B$ form a dense subspace of $V$.
Also, let $\mathfrak{gl}(V)$ be the Lie algebra of continuous linear endomorphisms of $V$. 

For a functor $X$ on $\mathbb{C}$-algebras, we write $X(R)$ for its $R$-points, $X := X(\mathbb{C})$ when no argument is given, and $X(\mathscr{O}_S)$ for the corresponding sheaf on a scheme $S$. Statements involving lattices, positivity, or continuous cohomology are statements over $\mathbb{C}$.

\subsection{The algebra of formal Laurent series}
Let $H$ be the functor which assigns to a $\mathbb{C}$-algebra $R$ the $R$-algebra $H(R):= R(\!(t)\!)$  of formal Laurent series in a parameter $t$. One has a decomposition 
\[
H=H_+ \oplus H_-, \quad \mbox{where}\quad H_+(R):= R\llbracket t \rrbracket \quad\mbox{and}\quad H_-(R):= t^{-1}R[t^{-1}].
\]
Also, let $H'$ be the functor which assigns to $R$ the $R$-module 
\[
H'(R)  := R(\!(t)\!)/R\,t^0
\]
of Laurent series with zero constant term. One has a decomposition
\[
H':=H'_+\oplus H_- \qquad\mbox{where }    H'_+(R) := t \, R\llbracket t \rrbracket .
\]

The $R$-module $H(R)$ is endowed with the $t$-adic topology. This is the topology that has the collection of cosets 
\[
\left\{f + t^N H_+(R) \,\, | \,\, f\in H(R) \,\,\mbox{and}\,\, N\geq 1 \right\}
\]
as a basis of open subsets. This gives $H(R)$ the structure of a complete topological space, topologically generated over $R$ by $b_i:= t^i$ for $i\in \mathbb{Z}$. 

An $R$-linear endomorphism $A$ of $H(R)$ is continuous if for every $n$ there exists $m$ with $A(t^m H_+(R)) \subseteq t^n H_+(R)$; a continuous automorphism is a continuous bijection with continuous inverse. The same applies to $H'(R)$.

We will repeatedly use the following result from \cite{arbarello1987infinite}: the Lie algebra $\mathfrak{gl}(H)$ has a canonical two-cocycle  given by
\begin{equation}
\label{eq:psi}
\psi \left( A,B \right) := \mathrm{Tr}\left( \pi_+ \, A \, \pi_- \, B \, \pi_+ - \pi_+ \, B \, \pi_- \, A \, \pi_+ \right) \qquad \mbox{for $A,B\in \mathfrak{gl}(H)$,}
\end{equation}
 where $\pi_+\colon H\rightarrow H_+$ and $\pi_-\colon H\rightarrow H_-$ are the natural projections.

\subsection{The Heisenberg  algebra}
Consider the skew-symmetric form $\langle \, , \, \rangle$ on $H$ given by
\begin{equation}
\label{eq:symplectic}
\langle f,g\rangle := - \mathop{\mathrm{Res}}_{t=0} \, f \, dg \qquad \mbox{for $f,g\in H$}.
\end{equation}
This only depends on $df$ and $dg$, for $f,g\in H$, and induces a nondegenerate symplectic form on $H'$.
The \textit{Heisenberg  algebra} is the Lie algebra structure on $H$ given by the Lie bracket $\langle \, , \, \rangle$.
As \eqref{eq:symplectic} is continuous with respect to the $t$-adic topology, the Heisenberg  algebra $H$ is a complete topological Lie algebra.
Moreover, \eqref{eq:symplectic} is invariant by changes of the parameter~$t$. 
Note that $\langle f,g\rangle = \psi(f,g)$, where $\psi$ is from \eqref{eq:psi}, and $f,g\in H$ act on $H$ by multiplication. Thus, the Heisenberg  algebra $H$ is the central extension 
\[
0 \rightarrow \mathfrak{gl}_1\, \bm{1} \rightarrow H \rightarrow  H'  \rightarrow 0
\]
given by the restriction of the two-cocycle $\psi$, where $\bm{1}:=b_0$, and $\mathfrak{gl}_1$ is the functor assigning to a $\mathbb{C}$-algebra $R$ the trivial Lie algebra $R$.

\subsection{The enveloping algebra $\mathscr{U}(H)$}
Let $\mathscr{U}(H)$ be the quotient of the universal enveloping algebra of the Heisenberg  algebra $H$ by the two-sided ideal generated by $\bm{1}-1$, where $\bm{1}$ is the central element of $H$, and $1$ is the unit in the universal enveloping algebra.
The algebra $\mathscr{U}(H)$ has a canonical filtration $\mathscr{U}_0(H) \subset \mathscr{U}_1(H)\subset \cdots$, where $\mathscr{U}_0(H) = \mathfrak{gl}_1\, \bm{1}$, and $\mathscr{U}_N(H)$ for $N\in\mathbb{N}$ assigns to a $\mathbb{C}$-algebra $R$ the $R$-module generated by the formal products $f_1\cdots f_m$ with $m\leq N$ and $f_i\in H(R)$ for $i\leq m$.

Remarkably, $\mathscr{U}_{2}(H)$ is a Lie subalgebra of
$\mathscr{U}(H)$ for the Lie bracket given by the commutator.
Explicitly, the Lie bracket in $\mathscr{U}_{2}(H)$ is given by
\begin{align}
\label{eq:U2}
\begin{split}
[f,g] & = \langle f,g\rangle, \\
{[f, \bm{1}]} &= {[fg, \bm{1}]} = 0, \\
{[fg, k]} & = \langle f,k\rangle \, g + \langle g,k\rangle \, f,\\
{[fg,kl]} &= \langle g,k \rangle \, fl + \langle f,k \rangle \, gl + \langle g,l \rangle \, kf + \langle f,l\rangle \, kg
\end{split}
\end{align}
for $f,g,k,l\in H'$.
One has 
\[
\mathscr{U}_{2}(H) / \mathfrak{gl}_1 \,\bm{1} \cong S^2(H') \ltimes H'
\]
as Lie algebras, where $S^2(H')$ is the second symmetric power of $H'$, and
the semidirect product is given by the action of $S^2(H')$ on $H'$ induced by
\begin{equation}
\label{eq:S2H'onH'}
fg \colon H'\rightarrow H', \qquad k\mapsto \langle f,k\rangle \, g + \langle g,k\rangle \, f, \qquad \mbox{for $f,g,k\in H'$.}
\end{equation}
Hence, $\mathscr{U}_{2}(H)$ is a central extension 
\[
0 \rightarrow \mathfrak{gl}_1\, \bm{1} \rightarrow \mathscr{U}_{2}(H) \rightarrow S^2(H') \ltimes H'  \rightarrow 0.
\]
To identify the corresponding two-cocycle, arguing by cases, one shows that 
\begin{eqnarray*}
[:fg:,:kl:] &= & -\frac{1}{2}\,\psi \left(fg,kl \right) \, \bm{1} \\
&& +\langle g,k \rangle \, : fl: + \langle f,k \rangle \, :gl: + \langle g,l \rangle \, :kf: + \langle f,l\rangle \, :kg:
\end{eqnarray*}
in $\mathscr{U}_{2}(H)$ is compatible with \eqref{eq:U2}, where 
$:\,\,:$ denotes the normal order product  giving a specific lift of elements from $S^2(H')$ to $\mathscr{U}_{2}(H)$, and 
$\psi$ is the restriction of the two-cocycle in \eqref{eq:psi}.
Hence, the  two-cocycle $c$ defining $\mathscr{U}_{2}(H) $ is given by 
\begin{align*}
c(fg,kl)&= -\frac{1}{2}\,\psi \left(fg,kl \right), &
c(f,g)&=\langle f,g \rangle, & 
c(fg,k)&=0 
\end{align*}
for $f,g,k,l\in H'$.

We will need a completion of $\mathscr{U}_{2}(H)$. For this, we first introduce the completion of $S^2(H')$ given by the symplectic  algebra and its corresponding extension.

\subsection{Symplectic and metaplectic  algebras}
\label{sec:symp&meta}
The \textit{symplectic  algebra} $\mathfrak{sp}\left(H' \right)$ is the Lie subalgebra of $\mathfrak{gl}(H')$ defined as
\[
\mathfrak{sp}\left(H' \right) := \left\{ X\in \mathfrak{gl}(H') \, | \, \langle X a, b \rangle + \langle a,X b \rangle =0, \mbox{ for all } a,b\in H' \right\}.
\]
The action of $S^2(H')$ on $H'$ induced by \eqref{eq:S2H'onH'} realizes $S^2(H')$ as a dense Lie subalgebra of $\mathfrak{sp}\left(H' \right)$.

The \textit{metaplectic  algebra} $\mathfrak{mp} \left(H' \right)$ 
is the central extension 
\begin{equation}
\label{eq:mp}
0 \rightarrow \mathfrak{gl}_1 \,\bm{1} \rightarrow \mathfrak{mp}\left(H' \right) \rightarrow \mathfrak{sp}\left(H' \right)  \rightarrow 0
\end{equation}
defined by the two-cocycle
\begin{equation}
\label{eq:cocyclemeta}
 -\frac{1}{2}\, \psi\left( X, Y\right)
\quad \mbox{for $X,Y\in \mathfrak{sp}\left(H' \right)$}
\end{equation}
where $\psi$ is the restriction of \eqref{eq:psi}.

\subsection{The Weyl algebra}
\label{sec:Weyl}
The \textit{Weyl algebra} $\widetilde{\mathscr{U}}(H)$ is the completion of 
the algebra $\mathscr{U}(H)$ with respect to the topology in which the basis of open neighborhoods of $0$ is formed by the left ideals generated by the submodules  
\[
t^N H_+\subset H  \subset  {\mathscr{U}}(H) \quad \mbox{for $N\geq 1$} 
\]
\cite[\S 2.1.2]{bzf}. Elements of $\widetilde{\mathscr{U}}(H)$ can be described as possibly infinite series of the form $A_0 + \sum_{i\geq 1} A_i \,b_i$ with $A_i$ equal to a finite linear combination of finite formal products of elements in $\{b_j\}_{j\in \mathbb{Z}}$ for $i\geq 0$. 

The closure $\widetilde{\mathscr{U}}_{2}(H)$ of $\mathscr{U}_{2}(H)$ in $\widetilde{\mathscr{U}}(H)$ is a complete topological Lie subalgebra
of $\widetilde{\mathscr{U}}(H)$ for the Lie bracket given by the commutator.
Explicitly, consider first the semidirect product $\mathfrak{sp}\left(H' \right) \ltimes H'$ given by the action 
\[
[X,f]=X(f) \qquad \mbox{for $X\in \mathfrak{sp}\left(H' \right)$ and $f\in H'$.}
\]
Hence, the Lie bracket for $\mathfrak{sp}\left(H' \right) \ltimes H'$ is 
\[
[X+f, Y+g] = XY-YX + X(g) - Y(f) 
\]
for $X,Y\in \mathfrak{sp}\left(H' \right)$ and $f,g\in H'$. Then, one has a central extension
\begin{equation}
\label{eq:Uleq2}
0 \rightarrow \mathfrak{gl}_1 \,\bm{1} \rightarrow \widetilde{\mathscr{U}}_{2}(H) \rightarrow \mathfrak{sp}\left(H' \right) \ltimes H'  \rightarrow 0
\end{equation}
with   two-cocycle $c$ given by 
\begin{align}
\label{eq:2cocycledefUtildeleq2}
c(X,Y)&= -\frac{1}{2}\, \psi \left( X,Y \right), &
c(f,g)&=\langle f,g \rangle, & 
c(X,f)&=0 
\end{align}
for $X,Y\in \mathfrak{sp}\left(H' \right)$ and $f,g\in H'$,
where $\psi$ is the restriction of \eqref{eq:psi}.
This is compatible with the Lie bracket for $\mathscr{U}_{2}(H)$ in \eqref{eq:U2}.

The metaplectic  algebra $\mathfrak{mp}\left( H' \right)$ is realized as a Lie subalgebra of $\widetilde{\mathscr{U}}_{2}(H)$ via the injection
\begin{equation}
\label{eq:mpintoUtilde}
\mathfrak{mp}\left( H' \right) \hookrightarrow \widetilde{\mathscr{U}}_{2}(H), \qquad \bm{1}\mapsto  \bm{1}, \quad X\mapsto X \quad \mbox{for $X\in \mathfrak{sp}\left(H' \right)$}.
\end{equation}

\subsection{Witt and Virasoro  algebras}
\label{sec:WittVir}
The \textit{Witt  algebra} $\mathrm{Witt}$ is the functor assigning to a $\mathbb{C}$-algebra $R$  the Lie algebra $R(\!(t)\!)\partial_t$ of continuous derivations  of  $R(\!(t)\!)$. It is topologically generated over $R$ by $L_p:= -t^{p+1}\partial_t$ for $p\in\mathbb{Z}$, with relations $[L_p, L_q] = (p-q)L_{p+q}$ for $p,q\in\mathbb{Z}$. 
The oscillator representation of $\mathrm{Witt}$ gives an injection of Lie algebras
\begin{equation}
\label{eq:tau}
\tau\colon \mathrm{Witt} \hookrightarrow \mathfrak{sp}\left(H' \right), \qquad 
L_p \mapsto \frac{1}{2}\, \sum_{i} b_{-i} \,b_{i+p}
\end{equation}
where the sum is over $i\in \mathbb{Z}\setminus \{0,-p\}$. Here and throughout, $b_{i} \,b_{j}$ denotes the tensor $b_{i} \otimes b_{j}\in S^2(H')$, rather than the product in $H$.

The \textit{Virasoro  algebra} $\mathrm{Vir}$ is the central extension
\[
0 \rightarrow \mathfrak{gl}_1 \,\bm{1} \rightarrow \mathrm{Vir} \rightarrow  \mathrm{Witt} \rightarrow 0
\]
with   Lie bracket  given by 
\[
[L_p, L_q] = (p-q)\, L_{p+q} + \frac{1}{12}\, (p^3-p)\, \delta_{p+q,0}\, \bm{1},
\] 
where $\delta_{p+q,0} =1$ when $p+q=0$, and $\delta_{p+q,0} =0$ otherwise.
The two-cocycle defining this extension is the restriction over $\mathrm{Witt} $ of the two-cocycle \eqref{eq:cocyclemeta}.
The map $\tau$ is lifted by the  injection
\begin{equation}
\label{eq:tauhat}
\widehat{\tau}\colon \mathrm{Vir} \hookrightarrow \mathfrak{mp} \left(H' \right), \qquad L_p \mapsto \frac{1}{2}\, \sum_{i} :b_{-i} \,b_{i+p}:, \quad \bm{1} \mapsto \bm{1}
\end{equation}
where the sum is over $i\in \mathbb{Z}\setminus \{0,-p\}$.

\subsection{The map $\sigma$ and the Lie algebra $\mathfrak{D}$}
\label{sec:D}
We will need the Lie algebra injection
\begin{equation}
\label{eq:sigma}
\sigma \colon \mathrm{Witt} \ltimes H' \hookrightarrow \mathfrak{sp}\left(H' \right) \ltimes H', \quad L_p \mapsto \frac{1}{2}\, \sum_{i} b_{-i} \,b_{i+p} \,+\, \frac{p+1}{2}\, b_p, \quad b_q \mapsto b_q
\end{equation}
where the sum is over $i\in \mathbb{Z}\setminus \{0,-p\}$. For $p=0$, one simply has  \sloppy\mbox{$L_0\mapsto\frac{1}{2}\, \sum_{i} b_{-i} \,b_{i}$,} since $b_0$ vanishes in $H'$.

\begin{remark}
There are infinitely many ways of injecting $\mathrm{Witt} \ltimes H'$ in $\mathfrak{sp}\left(H' \right) \ltimes H'$: by replacing in \eqref{eq:sigma} the coefficient $\frac{p+1}{2}$ with $-a(p+1)$ for $a\in\mathbb{C}$, one obtains a one-parameter family of such injections, as in the Segal-Sugawara construction \cite[\S2.5.9]{bzf}.
 The case $a=-\frac{1}{2}$ is specifically needed in \S\ref{sec:curves_lines}.
\end{remark}

Finally, let $\mathfrak{D}$ be the central extension
\[
0 \rightarrow \mathfrak{gl}_1 \,\bm{1} \rightarrow \mathfrak{D} \rightarrow \mathrm{Witt} \ltimes H' \rightarrow 0
\]
defined by the pull-back of the two-cocycle $\psi$ from \eqref{eq:psi} with $\mathrm{Witt}$ acting on $H$ via $L_p = -t^{p+1}\partial_t$ for all $p$ and $H'$ acting on $H$ by multiplication. Here elements of $H'$ are lifted to $H$ with zero constant term; the resulting map to $\mathfrak{gl}(H)$ is a homomorphism modulo scalars, and this is enough since $\psi(, \mathrm{id}) = 0$. The map $\sigma$ is lifted by the  injection
\begin{equation}
\label{eq:sigmahat}
\widehat{\sigma}\colon \mathfrak{D} \hookrightarrow \widetilde{\mathscr{U}}_{2}(H), 
\quad L_p \mapsto \frac{1}{2}\, \sum_{i} :b_{-i} \,b_{i+p}: \,+\, \frac{p+1}{2}\, b_p, \quad b_q \mapsto b_q, \quad \bm{1}\mapsto\bm{1}
\end{equation}
where the sum is over $i\in \mathbb{Z}\setminus \{0,-p\}$.

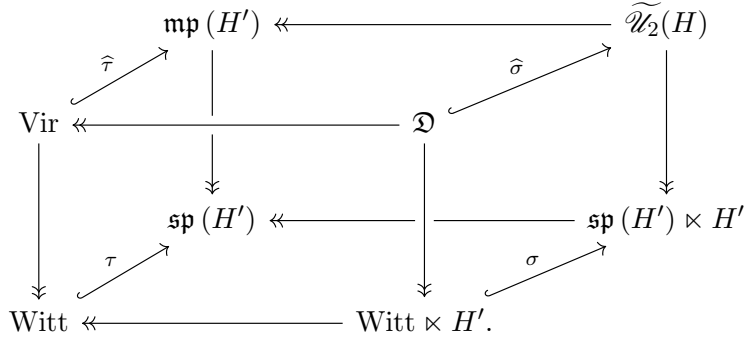
\begin{figure}[t]
\begin{tikzcd}
&\mathfrak{mp}\left(H' \right) \arrow[->>]{dd}  && \widetilde{\mathscr{U}}_{2}(H) \arrow[->>]{dd} \\
\mathrm{Vir} \arrow[->>]{dd} \arrow[hookrightarrow]{ru}{\widehat{\tau}} && \mathfrak{D}  \arrow[hookrightarrow]{ru}{\widehat{\sigma}} \\
&\mathfrak{sp}\left(H' \right) && \mathfrak{sp}\left(H' \right) \ltimes H' \arrow[->>]{ll} \\
\mathrm{Witt} \arrow[hookrightarrow]{ru}{\tau} && \mathrm{Witt} \ltimes H'. \arrow[->>]{ll}  \arrow[hookrightarrow]{ru}{\sigma}  \arrow[<<-, crossing over]{uu}  
\end{tikzcd}
\caption{The main Lie algebra maps.}
\label{fig:bigLiesquare}
\end{figure}

The main Lie algebra maps of this section are summarized by the commutative diagram in Figure \ref{fig:bigLiesquare}.
There, the two horizontal surjections  admit Lie algebra splittings, and the left and right squares are Cartesian.

%%%%%%%%%%%%%%%%%%%%%%%%%%%%
%%%%%%%%%%%%%%%%%%%%%%%%%%%%
%%%%%%%%%%%%%%%%%%%%%%%%%%%%

\section{Extended abelian varieties and their moduli}
\label{sec:Ahat}
First, we review  some geometric constructions and results from \cite{adc91} related to suitable extensions of abelian varieties and their moduli. 
Theorems \ref{thm:unif1} and \ref{thm:unif2}  play a key role in later sections. Then, we conclude with a review of the degree-$2$ cohomology classes on the moduli spaces in consideration.

\subsection{Extended abelian varieties}
\label{sec:ZFL}
Let $H:= H(\mathbb{C})$.
Recall the symplectic form $\langle \, , \, \rangle$ on $H'$ from \eqref{eq:symplectic}.
An \textit{extended principally polarized abelian variety} (extended PPAV, for short) of dimension $g$ is a triple $(Z, F, L)$, where 
\begin{align*}
&\mbox{$Z$ is a Lagrangian subspace of $H'$,} \\
&\mbox{$F$ is a codimension  $g$ subspace of $Z$, and} \\
&\mbox{$L$ is a rank $2g$ lattice in $F^\perp/F$,}
\end{align*}
satisfying four suitable conditions. The form $\langle \, , \, \rangle$ on $H'$ induces a nondegenerate symplectic form on $F^\perp/F$, which will still be denoted as $\langle \, , \, \rangle$.
The first three conditions imposed on the triple $(Z, F, L)$ are:
\begin{align}
\label{eq:cond1}
& Z \cap H'_+ = 0,\\
\label{eq:cond2}
& L_\mathbb{R} \cap  F^\perp_+ =0 \,\,\,\mbox{in $F^\perp/F$, where}\,\,\, L_\mathbb{R} := L\otimes_{\mathbb{Z}} \mathbb{R} \,\,\,\mbox{and}\,\,\, F^\perp_+:= F^\perp \cap H'_+,\\
\label{eq:cond3}
& \frac{1}{2\pi i}\, \langle \, , \, \rangle \,\, \mbox{is unimodular on $L$.}
\end{align}
Condition \eqref{eq:cond1} implies the decompositions 
\[
H'= Z\oplus H'_+ \quad\mbox{and}\quad F^\perp/F = Z/F \oplus F^\perp_+ 
\]
into maximal isotropic subspaces. It follows that $F^\perp/F$ has dimension $2g$. Condition \eqref{eq:cond2} implies that: the projection $F^\perp/F \rightarrow Z/F$ induces a real isomorphism $L_\mathbb{R} \cong Z/F$; the real structure on $F^\perp/F$ induced by $L_\mathbb{R}$ identifies $F^\perp_+$ with the conjugate of $Z/F$; and the form on $Z/F$ given by
\[
B(u,v):= \frac{1}{\pi}\, \langle \overline{u} , v \rangle \quad \mbox{for $u,v\in Z/F$}
\]
is Hermitian. The fourth and last condition on the triple $(Z, F, L)$ is
\begin{equation}
\label{eq:cond4}
\mbox{the form $B$ on $Z/F$ is positive definite.}
\end{equation}

A triple $(Z, F, L)$ satisfying the four conditions \eqref{eq:cond1}-\eqref{eq:cond4} determines an abelian variety $X$ of dimension $g$, which is the quotient of the $g$-dimensional space $Z/F$ by the image of $L$ under the projection $F^\perp/F \rightarrow Z/F$, equipped with the polarization $B$, which is principal  by  \eqref{eq:cond3}.
More precisely, the triple $(Z, F, L)$  is equivalent to an isomorphism class of extensions
\begin{equation}
\label{eq:H_+extX}
0\rightarrow H'_+ \rightarrow H'/K \rightarrow X \rightarrow 0
\end{equation}
where $K$ is the preimage of $L$ under the projection $F^\perp \rightarrow F^\perp/F$.

\subsection{Moduli of extended PPAVs}
\label{sec:moduliofZFL}
Let $\widehat{\mathcal{A}}_g$ be the moduli space of extended PPAVs of dimension $g$.
In \cite{adc91},  $\widehat{\mathcal{A}}_g$  is constructed as an infinite-dimensional analytic manifold.
Here is a sketch of the construction. 
Consider the infinite-dimensional analytic manifold
\begin{equation}
\label{eq:Hghat}
\widehat{\mathcal{H}}_g := \widetilde{S}^2(H'_+) \times \mathcal{B}_g(H'_+) \times \mathcal{H}_g
\end{equation}
where
\begin{align*}
\widetilde{S}^2\left(H'_+\right) &:= \left\{ \varphi \colon H_- \rightarrow H'_+ \, \Bigg| \,  
\begin{array}{l}
\mbox{$\varphi$ is linear and symmetric,} \\[5pt]
\mbox{i.e., $\langle a, \varphi(b)\rangle = \langle b, \varphi(a)\rangle$  for all  $a,b\in H_-$} 
\end{array}
\right\},\\[5pt]
\mathcal{B}_g(H'_+) &:= \mbox{the manifold of frames of $(H'_+)^g$,}\\[5pt]
\mathcal{H}_g &:= \mbox{the Siegel upper half-space of degree $g$.}
\end{align*}
Points of $\mathcal{B}_g(H'_+)$ are the elements $(h_1, \dots, h_g)$ in $(H'_+)^g$ consisting of $g$ linearly independent vectors.
One has a map of sets
\begin{equation}
\label{eq:fromHtoAbeforequot}
\widehat{\mathcal{H}}_g \rightarrow \widehat{\mathcal{A}}_g, \qquad 
(\varphi, (h_1, \dots,h_g),\Omega) \mapsto (Z,F,L)
\end{equation}
where: (i) the Lagrangian $Z$ is defined as the graph of $\varphi$; (ii) considering $h_1, \dots, h_g$ as linear forms on $H_-$ and thus on $Z$ after composing with the projection $Z\rightarrow H_-$, the isotropic space $F\subset Z$ is defined by $h_1=\cdots=h_g=0$; and (iii) the lattice $L$ is constructed from $\Omega$.

The symplectic group $\mathrm{Sp}(2g, \mathbb{Z})$ acts freely and properly discontinuously on $\widehat{\mathcal{H}}_g$, 
hence the quotient $\widehat{\mathcal{H}}_g \, / \, \mathrm{Sp}(2g, \mathbb{Z})$ has a natural structure of an infinite-dimensional analytic manifold.
Moreover, the map \eqref{eq:fromHtoAbeforequot} factors to a bijection of sets
\begin{equation}
\label{eq:fromHtoA}
\widehat{\mathcal{H}}_g \, / \, \mathrm{Sp}(2g, \mathbb{Z}) \rightarrow \widehat{\mathcal{A}}_g.
\end{equation}
Hence, \eqref{eq:fromHtoA} induces on $\widehat{\mathcal{A}}_g$ the structure of an infinite-dimensional analytic manifold \cite[\S 3]{adc91}.

One has a commutative diagram of moduli spaces
\begin{equation*}
\begin{tikzcd}[column sep=6em]
\widehat{\mathcal{H}}_g \arrow{r}{/\, \mathrm{Sp}(2g, \mathbb{Z})} \arrow{d} & \widehat{\mathcal{A}}_g \arrow{d}\\
\mathcal{H}_g \arrow{r}{/\, \mathrm{Sp}(2g, \mathbb{Z})} & \mathcal{A}_g
\end{tikzcd}
\end{equation*}
where $\mathcal{A}_g$ is the moduli space of principally polarized abelian varieties of dimension $g$. The vertical maps are the natural projections.
The action of $\mathrm{Sp}(2g, \mathbb{Z})$ on $\mathcal{H}_g$ has finite nontrivial stabilizers, hence at least an orbifold structure is required on $\mathcal{A}_g$. Here we consider $\mathcal{A}_g$ as the resulting quotient stack. This is a separated, smooth Deligne-Mumford stack. From \cite{adc91}, $\widehat{\mathcal{A}}_g$ is rational homotopy equivalent to ${\mathcal{A}}_g$.

In the following, we will consider families $(Z,F,L)\rightarrow S$ of extended PPAVs of dimension $g$ over a smooth scheme~$S$. 
A family over $S$ consists of $\mathscr{O}_S$-submodules $F \subset Z \subset H'(\mathscr{O}_S)$ with $Z \oplus H'_+(\mathscr{O}_S) = H'(\mathscr{O}_S)$ and $Z/F$ locally free of rank $g$, and a locally constant sheaf $L$ of rank $2g$ lattices in $F^\perp/F$, such that the fibre at every $\mathbb{C}$-point of $S$ is an extended PPAV.
We will construct sheaves of coinvariants on such families in \S\ref{sec:Vhat} and descend them to families of PPAVs in \S\ref{sec:sheafV}, thus yielding sheaves on~${\mathcal{A}}_g$.

\subsection{Symplectic uniformization of moduli of extended PPAVs}
\label{sec:spunif}
An action of a Lie algebra $\mathfrak{g}$ on a variety $X$ over $\mathbb{C}$ is a homomorphism of Lie algebras $\mathfrak{g}\rightarrow \mathrm{Vect}(X)$, where $\mathrm{Vect}(X)$ is the Lie algebra of regular vector fields on $X$. The action is said to be \textit{transitive} if for every $x\in X$, the evaluation map $\mathfrak{g}\rightarrow T_x(X)$ is surjective.

\begin{theorem}[Uniformization of $\widehat{\mathcal{A}}_g$ \cite{adc91}]
\label{thm:unif1}
The moduli space $\widehat{\mathcal{A}}_g$ carries a transitive action of $\mathfrak{sp} \left( H' \right)$. For $(Z, F, L)$ in $\widehat{\mathcal{A}}_g$, one has
\[
T_{(Z, F, L)} \left( \widehat{\mathcal{A}}_g \right) \cong \mathfrak{sp}_{F} \left( H' \right)  \setminus  \mathfrak{sp}  \left( H' \right) 
\]
where
\begin{equation}
\label{eq:spF}
\mathfrak{sp}_{F} \left( H' \right):= \left\{ X\in \mathfrak{sp}  \left( H' \right) \,\, | \,\, X\left( F^\perp \right) \subseteq F  \right\}.
\end{equation}
\end{theorem}

We sketch the proof from \cite{adc91}: using the map \eqref{eq:fromHtoA}, and since the action of $\mathrm{Sp}(2g, \mathbb{Z})$ on $\widehat{\mathcal{H}}_g$ is properly discontinuous, one deduces
\[
T_{(Z, F, L)} \left( \widehat{\mathcal{A}}_g \right) \cong \widetilde{S}^2\left(H'_+\right) \times \mathrm{Hom} \left( H_-, Z/F \right) \times S^2\left(Z/F\right).
\]
Each factor here is the tangent space to the corresponding factor of $\widehat{\mathcal{H}}_g$ in \eqref{eq:Hghat} at a point mapping to $(Z,F,L)$ under \eqref{eq:fromHtoA}.
Identifying the dual of $H_-$ with $H'_+$ via the symplectic form $\langle \, , \, \rangle$, one has
\[
T_{(Z, F, L)} \left( \widehat{\mathcal{A}}_g \right) \cong \widetilde{S}^2\left(H'_+\right) \times \left( H'_+ \,\widehat{\otimes}\, Z/F \right) \times S^2\left(Z/F\right).
\]
From the splitting $H'/F = H'_+ \oplus Z/F$, one has 
\[
T_{(Z, F, L)} \left( \widehat{\mathcal{A}}_g \right) \cong \widetilde{S}^2\left(H'/F\right).
\]
Here $\widetilde{S}^2\left(H'/F\right)$ is the completion of ${S}^2\left(H'/F\right)$, that is, the quotient of $\mathfrak{sp}  \left( H' \right)$ by the closure of $F\cdot H'$.
Since $\mathfrak{sp}_{F} \left( H' \right) \cong \left(H'/F \,\widehat{\otimes}\, F\right) \times S^2(F)$, one has
\[
T_{(Z, F, L)} \left( \widehat{\mathcal{A}}_g \right) \cong \mathfrak{sp}_{F} \left( H' \right)  \setminus  \mathfrak{sp}  \left( H' \right).
\]
We emphasize that the subspace
\begin{equation}
\label{eq:tangenttofiberAtildeA}
\widetilde{S}^2(H'_+) \times \left( H'_+ \,\widehat{\otimes}\, Z/F \right) \subset T_{(Z, F, L)} \left( \widehat{\mathcal{A}}_g \right)
\end{equation}
is the tangent space to the fiber of the forgetful map $\widehat{\mathcal{A}}_g\rightarrow  {\mathcal{A}}_g$ at $(Z,F,L)$, while the subspace $S^2\left(Z/F\right)$ encodes the infinitesimal deformations of the complex structure of the underlying abelian variety.

\subsection{The universal extended PPAV}
\label{sec:extunivPPAV}
The moduli space $\widehat{\mathcal{A}}_g$ admits a universal family $\widehat{\mathcal{X}}_g \rightarrow \widehat{\mathcal{A}}_g$. 
The fiber over a point $(Z,F,L)$ in $\widehat{\mathcal{A}}_g$ is
\[
\left(H' / K \right) \times_{G} \mathscr{Q}
\]
where $G$ is the group of points of order $2$ in $H'/K$, and $\mathscr{Q}$ is the $G$-torsor of integral quadratic forms $q$ on $L$ taken modulo $2$ satisfying 
\[
q(a) + q (b) - q(a+b) \equiv \frac{1}{2\pi i} \, \langle a, b \rangle \quad \mbox{mod $2$.}
\]
This $G$-torsor $\mathscr{Q}$ is isomorphic to the $G$-torsor of theta-characteristics on $X=(Z/F)/L$.
The need for the twist by $G$ becomes clear when one defines the maps in \eqref{eq:bigModulisquare} below.
Points of $\widehat{\mathcal{X}}_g$ are denoted as $(Z, F, L, \overline{h}, q)$.
The map $\widehat{\mathcal{X}}_g\rightarrow \widehat{\mathcal{A}}_g$ has no zero-section, however the fiber over $(Z,F,L)$ in $\widehat{\mathcal{A}}_g$ is non-canonically isomorphic to $H'/K$.

From \cite{adc91}, $\widehat{\mathcal{X}}_g$ is rational homotopy equivalent to $\mathcal{X}_g$.
This is the universal abelian torsor $\mathcal X_g\to\mathcal A_g$, that is, the family
whose fiber over $(X,B)$ is the torsor $P$ under $X$ parametrizing the line
bundles on $X$ inducing the polarization $B$ \cite{shepherd2008thomae}, \cite[\S4]{fringuelli2019picard}. This $P$ carries the universal
symmetric divisor $\theta$, the locus of line bundles whose unique effective divisor
contains the origin in $X$. 
Every line bundle inducing $B$ is the translate of a symmetric one, unique up to
translation by a point of order two. The symmetric ones form a torsor under $G\cong X[2]$,
identified with $\mathscr{Q}$ by \cite[Prop.~4.7]{fringuelli2019picard}. Hence $P\cong X\times_G \mathscr{Q}$,
and the quotient of the fiber
$H'/K\times_G \mathscr{Q}$ of $\widehat{\mathcal X}_g$ by $H'_+$ is the fiber of
$\mathcal X_g$. For $g\geq2$, the torsor $\mathcal X_g\to\mathcal A_g$ has no section,
so it is not isomorphic to the universal abelian variety; it becomes so after base
change to the finite \'etale cover of $\mathcal A_g$ parametrizing theta
characteristics.

More generally, we will consider families $(Z, F, L, \overline{h}, q)\rightarrow S$ over a smooth scheme $S$ for a fixed~$g$.
Such a family consists of a family $(Z,F,L) \rightarrow S$ as in \S\ref{sec:moduliofZFL} together with a section $\overline{h}$ of $H'(\mathscr{O}_S)/K$, where $K \subset F^\perp$ is the preimage of $L$, and a locally constant section $q$ of the torsor $\mathscr{Q}$.

\subsection{Uniformization of the universal extended PPAV}
\label{sec:unif2}
Recall the Lie algebra $\mathfrak{sp}  \left( H' \right) \ltimes H'$ from \S\ref{sec:Weyl}.
For $x=(Z, F, L, \overline{h}, q)$ in $\widehat{\mathcal{X}}_g$, 
select $h\in \overline{h}+K$ and let
\begin{equation}
\label{eq:ix}
\mathfrak{i}(x):= \{X-X(h)+a\in \mathfrak{sp}  \left( H' \right) \ltimes H' \, | \, X\in  \mathfrak{sp}_{F} \left( H' \right) \mbox{ and } a\in F \}.
\end{equation}
The choice of $h\in \overline{h}+K$ determines an isomorphism 
\begin{equation}
\label{eq:ixiso}
 \mathfrak{sp}_{F} \left( H' \right) \ltimes F
 \xrightarrow{\cong}
\mathfrak{i}(x), \qquad
 (X,a)\mapsto X-X(h)+a,
\end{equation}
but $\mathfrak{i}(x)$ is independent of this choice.

\begin{theorem}[Uniformization of $\widehat{\mathcal{X}}_g$ \cite{adc91}]
\label{thm:unif2}
The moduli space $\widehat{\mathcal{X}}_g$ carries a transitive action of $\mathfrak{sp} \left( H' \right)\ltimes H'$. For $x=(Z, F, L, \overline{h}, q)$ in $\widehat{\mathcal{X}}_g$, one has
\[
T_{x} \left( \widehat{\mathcal{X}}_g \right) \cong \mathfrak{i}(x) \setminus  \left( \mathfrak{sp}  \left( H' \right) \ltimes H'\right).
\]
\end{theorem}

\subsection{The case of curves and line bundles}
\label{sec:curves_lines}
The moduli spaces $\widehat{\mathcal{A}}_g$ and $\widehat{\mathcal{X}}_g$ are naturally extensions of moduli spaces of curves and line bundles. To see this, 
let $\widehat{\mathcal{M}}_g$ be the moduli space of triples $(C,P,t)$, where $C$ is a smooth genus $g$ curve, $P$ is a point of $C$, and $t$ is a formal coordinate at $P$ \cite{adkp}. 
From \cite{adc91}, there is an \textit{extended Torelli map} $\widehat{\mathcal{M}}_g \rightarrow \widehat{\mathcal{A}}_g$ mapping $(C,P,t)$ to a triple $(Z,F,L)$ with identifications
\[
F\cong H^0\left(C\setminus P, \mathscr{O}_C\right)/\,\mathbb{C},  \quad
Z/F \cong H^1\left( C, \mathscr{O}_C\right),\quad
L\cong H^1\left(C, \mathbb{Z}\right).
\]
The isotropicity of $F$ with respect to \eqref{eq:symplectic} follows from the residue theorem.

Moreover, let $\widehat{\mathcal{P}}_{d}\rightarrow \widehat{\mathcal{M}}_g$ be the \textit{relative Picard variety} parametrizing isomorphism classes of quintuples $(C,P,t,\mathscr{L},\varphi)$ where $\mathscr{L}$ is a line bundle of degree $d$ on $C$ and $\varphi$ is a  trivialization of the restriction of $\mathscr{L}$ to the formal disc around~$P$ up to a nonzero multiplicative constant. 
For $d=g-1$, one has a commutative diagram % Not Cartesian!
\begin{equation}
\label{eq:bigModulisquare}
\begin{tikzcd}
\widehat{\mathcal{P}}_{g-1} \arrow[]{r} \arrow{d}& \widehat{\mathcal{X}}_g \arrow{d}\\
\widehat{\mathcal{M}}_g  \arrow[]{r}& \widehat{\mathcal{A}}_g
\end{tikzcd}
\end{equation}
where  $(C,P,t,\mathscr{L},\varphi)$ in $\widehat{\mathcal{P}}_{g-1}$ is mapped to $(Z, F, L, \overline{h}, q)$ in $\widehat{\mathcal{X}}_g$ such that $q$ is  arbitrary in $\mathscr{Q}$, and $\overline{h}$ is the unique point of $H'/K$ given by the isomorphism class of 
\[
\left(\mathscr{L}\,\eta^{-1}, \varphi \, dt^{-\frac{1}{2}}\right),
\]
where $\eta$ is the unique theta-characteristic on $C$ identified by $q$ via Riemann's theorem \cite[(3.23)]{adc91}. By the construction of $\widehat{\mathcal{X}}_g$, this map is independent of the choice of $q$.

Recall the commutative diagram 
\begin{equation}
\label{eq:smallLiesquare}
\begin{tikzcd}
\mathrm{Witt} \ltimes H' \arrow[]{r}{\sigma} \arrow{d}& \mathfrak{sp}  \left( H' \right) \ltimes H' \arrow{d}\\
\mathrm{Witt}  \arrow[]{r}{\tau}& \mathfrak{sp}  \left( H' \right)
\end{tikzcd}
\end{equation}
from Figure \ref{fig:bigLiesquare}. It is shown in \cite[(5.11)]{adc91} that the Lie algebras in \eqref{eq:smallLiesquare} act transitively on the corresponding moduli spaces in \eqref{eq:bigModulisquare}, and the maps in \eqref{eq:smallLiesquare} induce the differentials of the corresponding maps in \eqref{eq:bigModulisquare}.

\subsection{Embeddings into the projectivized Boson space}
\label{sec:Boson}
Here we review the embedding of $\widehat{\mathcal X}_g$ into an infinite-dimensional projective space from \cite{adc91}.
Consider the \textit{Boson space} $B:=\mathbb{C}\llbracket x_1, x_2, \dots\rrbracket$  of formal power series in infinitely many variables.
One has a natural identification with the completed symmetric algebra
\[
B \cong \widehat{\mathrm{Sym}}\, H'_+
\]
inducing a topology on $B$. In particular, $B$ is identified with the full dual $V^\vee$ of the Fock space
\[
V:={\mathscr{U}}(H)\otimes_{\mathscr{U}\left(H_+\right)} \mathbb{C}v\quad \cong\quad \mathrm{Sym}(H_-)\,v \qquad\mbox{(as vector spaces)}
\]
with $H'_+(v)=0$.

The space $\widehat{\mathcal P}_{g-1}$ embeds into $\mathbb{P}B$. Indeed, 
the Krichever map gives an embedding 
\[
K\colon\widehat{\mathcal P}_{g-1}\hookrightarrow\mathrm{Gr}_0(H)
\]
\cite[(3.12)]{adkp} where $\mathrm{Gr}(H)$ is the \textit{Sato Grassmannian} consisting of (closed) subspaces $W \subset H$ such that
both $W \cap H_+$ and $H/(W + H_+)$ are finite-dimensional \cite{segal1985loop, adkp,adc91}, and 
\[
\mathrm{Gr}_0(H):=\{ W\in \mathrm{Gr}(H)\,|\, \dim \left(W \cap H_+\right)=\dim \left(H/(W + H_+)\right)\}.
\]
Moreover, $\mathrm{Gr}(H)$ has a determinant line bundle $\mathscr{L}$ with fibers 
\[
\mathscr{L}_W = \mathrm{det}\left( W \cap H_+ \right)\otimes_\mathbb{C} \mathrm{det}\left( H/(W + H_+) \right)^\vee
\]
\cite[(1.19)]{adkp}, and $\mathscr{L}^{-1}$ yields a Pl\"ucker embedding. Restricting to $\mathrm{Gr}_0(H)$ and composing with the Boson-Fermion correspondence, this is
\[
\Pi\colon\mathrm{Gr}_0(H)\hookrightarrow\mathbb{P}B
\]
\cite[\S1]{adc91}. Also, $K^*\mathscr L^{-1}=\mathscr{O}(\theta)$ \cite[(3.16)]{adkp}.

From \cite{adc91}, the embedding of $\widehat{\mathcal P}_{g-1}$ extends to an embedding $\Phi$ of $\widehat{\mathcal X}_g$ into $\mathbb{P}B$ yielding a commutative diagram 
\begin{equation}
\label{eq:Shiota}
\begin{tikzcd}
\widehat{\mathcal{X}}_g \arrow[hookrightarrow]{r}{\Phi}& \mathbb{P}B\\
\widehat{\mathcal P}_{g-1}\arrow[hookrightarrow]{u} \arrow[hookrightarrow]{r}{K}& \mathrm{Gr}_0(H). \arrow[hookrightarrow]{u}{\Pi}
\end{tikzcd}
\end{equation}

Finally, the differential of $\Phi$ is described as follows.
The natural action of $\widetilde{\mathscr{U}}_{2}(H)$ on the Fock space $V$ with $\bm{1}$ acting as $\mathrm{id}_V$ induces a contragredient action of $\widetilde{\mathscr{U}}_{2}(H)$ on $V^\vee\cong B$:
\begin{equation}
\label{eq:U2HactingonB}
\widetilde{\mathscr{U}}_{2}(H) \rightarrow \mathfrak{gl}(B)
\end{equation}
with $X\in \widetilde{\mathscr{U}}_{2}(H)$ acting as $(X\cdot \varphi)(w):=-\varphi(X\cdot w)$ for $\varphi\in B$ and $w\in V$,
and thus $\bm{1}$ acting as $-\mathrm{id}_B$. 
For $x\in \widehat{\mathcal X}_g$, after factoring by $\mathbb{C}\bm{1}$ and the stabilizers on both sides, this map induces the differential $d\Phi_x$ \cite[(5.11)]{adc91}.

\subsection{The Hodge and canonical line bundles}
\label{sex:lambdatheta}
Here we review the degree-$2$ cohomology classes on the moduli spaces in consideration.
Let $\Lambda\rightarrow {\mathcal{A}}_g$ be the \textit{Hodge line bundle} defined by 
\[
\Lambda:= \det \left( \pi_*\left(\omega_{{\mathcal{X}}_g/{\mathcal{A}}_g}\right)\right)
\]
where $\pi\colon {\mathcal{X}}_g\rightarrow {\mathcal{A}}_g$ is the natural projection and $\omega_{{\mathcal{X}}_g/{\mathcal{A}}_g}$ is the relative dualizing sheaf.
Let $\Xi:=\mathscr{O}(2\theta)$ on $\mathcal X_g$, with $\theta$ as in
\S\ref{sec:extunivPPAV}. It induces twice the principal polarization on each fiber, and
its pull-back via \mbox{$\widehat{\mathcal P}_{g-1}\to\widehat{\mathcal X}_g\to\mathcal X_g$}
in \eqref{eq:bigModulisquare} is $\mathscr{O}(2\theta)$, with $\theta$ the locus of quintuples
$(C,P,t,\mathscr{L},\varphi)$ with $h^0(C,\mathscr{L})>0$.
For simplicity, we omit the notation for the pull-back when referring to the pull-back of $\Lambda$ and $\Xi$ via the natural projections $\widehat{\mathcal{A}}_g \rightarrow \mathcal{A}_g$
 and $\widehat{\mathcal{X}}_g \rightarrow \mathcal{X}_g$. Let $\lambda$ and $\xi$ be the first Chern classes of $\Lambda$ and $\Xi$, respectively.
 
\begin{proposition}
\label{prop:H2AXMP}
For $g\geq 2$, one has
\begin{align}
\label{eq:H2Aghat}
&H^2\left( \widehat{\mathcal{A}}_g, \mathbb{Q}\right) = H^2\left( {\mathcal{A}}_g, \mathbb{Q}\right) = \mathbb{Q}\lambda,\\
\label{eq:H2Xghat}
&H^2\left( \widehat{\mathcal{X}}_g, \mathbb{Q}\right) = H^2\left( {\mathcal{X}}_g, \mathbb{Q}\right) = \mathbb{Q}\lambda \oplus \mathbb{Q}\xi.
\end{align} 
Moreover, for $g\geq 3$, the pull-back via the extended Torelli map induces an identification 
\begin{equation}
\label{eq:H2AghatH2Mghat}
H^2(\widehat{\mathcal{A}}_g, \mathbb{Q})\xrightarrow{=}  H^2\left(\widehat{\mathcal{M}}_g, \mathbb{Q}\right) = H^2\left({\mathcal{M}}_g, \mathbb{Q}\right)
\end{equation}
and an inclusion 
\begin{equation}
\label{eq:H2XghatH2Pghat}
H^2\left(\widehat{\mathcal{X}}_g, \mathbb{Q}\right)\hookrightarrow H^2\left(\widehat{\mathcal{P}}_{g-1}, \mathbb{Q}\right)= H^2\left({\mathcal{P}}_{g-1}, \mathbb{Q}\right).
\end{equation}
\end{proposition}

Here $\mathcal{P}_{g-1}$ is the moduli space of pairs $(C,\mathscr{L})$ where
$C$ is a smooth genus $g$ curve, and $\mathscr{L}$ is a degree $g-1$ line bundle on $C$. 

The statement combines various results from the literature.
Note that the identification $H^2(\mathcal{A}_g, \mathbb{Q})= H^2(\mathcal{M}_g, \mathbb{Q})$ fails for $g=2$, since $H^2(\mathcal{M}_2, \mathbb{Q})=0$.

\begin{proof}
For $g\geq 2$, one has $H^2(\mathcal{A}_g, \mathbb{Q})= \mathbb{Q}\lambda$, and for $g\geq 3$, the Torelli map induces an identification $H^2(\mathcal{A}_g, \mathbb{Q})= H^2(\mathcal{M}_g, \mathbb{Q})$  (e.g., \cite[8.1]{vanderGeer13cohomology} and \cite[17.3--17.4]{hain1moduli}). 

Since $\mathcal{A}_g$ and $\widehat{\mathcal{A}}_g$ have the same rational homotopy type \cite[pg.~16--17]{adc91}, 
\eqref{eq:H2Aghat} follows.
Similarly, $\mathcal{X}_g$ and $\widehat{\mathcal{X}}_g$ have the same rational homotopy type \cite[pg.~17]{adc91}, hence
the first isomorphism in \eqref{eq:H2Xghat}.
The second identification in \eqref{eq:H2Xghat} follows from the analogous statement on the genuine universal abelian variety from \cite[Prop.~4.3]{grushevsky2018stable} for $g\geq 3$ and from \cite{lee2022cohomology} for $g=2$. 
The same argument applies to $\mathcal{X}_g$, since it has the same local systems of cohomology groups of the fibers as the universal abelian variety.

Next, consider the composition of forgetful maps 
\begin{equation}
\label{eq:MhatM'M}
\widehat{\mathcal{M}}_g \rightarrow \mathcal{M}'_{g} \rightarrow \mathcal{M}_g
\end{equation}
where $\mathcal{M}'_{g}$ is the moduli space of triples $(C,P,v)$, where $C$ is a curve of genus $g$, $P$ is a point of $C$, and $v$ is a nonzero tangent vector to $C$ at $P$.
Since $\widehat{\mathcal{M}}_g$ and $\mathcal{M}'_{g}$ have the same rational homotopy type \cite[3.3]{adkp} and 
\[
H^2\left(\mathcal{M}'_{g}, \mathbb{Q}\right)=H^2\left(\mathcal{M}_{g}, \mathbb{Q}\right)=\mathbb{Q}\lambda
\]
\cite[5.8]{adkp}, the composition of the maps in \eqref{eq:MhatM'M} induces an identification $H^2(\widehat{\mathcal{M}}_g, \mathbb{Q})= \mathbb{Q}\lambda$. 
Using \eqref{eq:H2Aghat}, we deduce  \eqref{eq:H2AghatH2Mghat}. Note that the class of the cotangent line bundle at the marked point vanishes on  $\mathcal{M}'_{g}$ and $\widehat{\mathcal{M}}_g$.

Similarly, consider the composition of forgetful maps
\[
\widehat{\mathcal{P}}_{g-1} \rightarrow \mathcal{P}'_{g-1} \rightarrow \mathcal{P}_{g-1}
\]
where $\mathcal{P}'_{g-1}$ has been defined in the introduction. 
One has that $\widehat{\mathcal{P}}_{g-1}$ and $\mathcal{P}'_{g-1}$ have the same rational homotopy type \cite{adkp}.
Moreover, one has 
\[
H^2\left(\mathcal{P}'_{g-1}, \mathbb{Q}\right)=H^2\left(\mathcal{P}_{g-1}, \mathbb{Q}\right)=\mathbb{Q}^{\oplus 3}
\]
with $\lambda$ and $\xi$ being linearly independent (see \cite[5.7]{adkp} for $g \geq 5$ and \cite[Thm A and \S 7.1]{melo2014picard} for $g \geq 3$).
 The inclusion \eqref{eq:H2XghatH2Pghat} follows, hence the statement.
\end{proof}

%%%%%%%%%%%%%%%%%%%%%%%%%%%%
%%%%%%%%%%%%%%%%%%%%%%%%%%%%
%%%%%%%%%%%%%%%%%%%%%%%%%%%%

\section{Group ind-schemes  of symplectic automorphisms}
\label{sec:groupind}

Here we study how to integrate the symplectic algebra $\mathfrak{sp}\left(H' \right)$ and related Lie algebras.
We use group ind-schemes (countable unions of schemes along closed embeddings)
as in \cite[\S A.1.4]{bzf}. For a topological vector space $E$, let $\mathbb{GL}(E)$ be the group of continuous linear automorphisms of $E$.

\subsection{The Lie algebra $\mathfrak{sp}^+\left(H' \right)$}
\label{sec:sp+}
Let $\mathfrak{sp}^+\left(H' \right)$ be the Lie subalgebra of $\mathfrak{sp}\left(H' \right)$ which preserves the subspace $H'_+$, that is:
\[
\mathfrak{sp}^+\left(H' \right) := \left\{ X\in\mathfrak{sp}\left(H' \right) \,\, | \,\, X(H'_+)\subseteq H'_+ \right\}.
\]
This is topologically generated by $S^2\left(H'_+\right)$ and the elements $b_ib_j$ in $S^2\left(H'\right) \subset \mathfrak{sp}\left(H' \right)$ such that $i>0>j$.

The Lie algebra $\mathfrak{sp}^+\left(H' \right)$ has the following geometric incarnation.
Recall
the tangent space to the fibers of the forgetful map $\widehat{\mathcal{A}}_g\rightarrow  {\mathcal{A}}_g$ from \eqref{eq:tangenttofiberAtildeA}.
Since $\mathfrak{sp}^+\left(H' \right)$ surjects onto \eqref{eq:tangenttofiberAtildeA}, the transitive action of $\mathfrak{sp} \left( H' \right)$ on $\widehat{\mathcal{A}}_g$ from \S\ref{sec:spunif}
restricts to a transitive action of $\mathfrak{sp}^+\left(H' \right)$
 on the fibers of $\widehat{\mathcal{A}}_g\rightarrow  {\mathcal{A}}_g$. 
The stabilizer of $\mathfrak{sp}^+\left(H' \right)$ at $(Z,F,L)$ is  
\begin{equation}
\label{eq:sp+F}
\mathfrak{sp}^+_F\left(H' \right) := \mathfrak{sp}^+\left(H' \right) \cap  \mathfrak{sp}_F \left(H' \right)
\end{equation}
which is a Lie subalgebra of $\mathfrak{sp}^+\left(H' \right)$ and $\mathfrak{sp}_F \left(H' \right)$. 

Moreover, the symplectic uniformization of $\widehat{\mathcal{A}}_g$ induces the following double coset realization of the tangent space of ${\mathcal{A}}_g$.
For $(Z,F,L)\in \widehat{\mathcal{A}}_g$, let $(X,B)\in {\mathcal{A}}_g$ be the image of $(Z,F,L)$ via $\widehat{\mathcal{A}}_g\rightarrow  {\mathcal{A}}_g$ as in \eqref{eq:H_+extX}. One has
\begin{equation}
\label{eq:DCAg}
T_{(X,B)} \left( \mathcal{A}_g \right) \cong \mathfrak{sp}_{F} \left( H' \right)  \setminus  \mathfrak{sp}  \left( H' \right) \, / \, \mathfrak{sp}^+\left(H' \right).
\end{equation}

\subsection{The group ind-schemes $\mathbb{S}\mathrm{p}^+\left(H'\right)$ and $\mathbb{S}\mathrm{p}^+_F\left(H'\right)$}
\label{sec:Sp+}
We will need to integrate (a dense Lie subalgebra of) the Lie algebra $\mathfrak{sp}^+\left(H' \right)$ to a group ind-scheme.
For this, consider first the group functor which assigns to a $\mathbb{C}$-algebra $R$ the group $\mathbb{GL}_b\left(H'(R)\right)$ of continuous linear automorphisms $\rho$ of $H'(R)$ such that there exists $k\in\mathbb{N}$ with 
\begin{equation}
\label{eq:rhohasak}
t^{n+k}H'_+(R) \subseteq \rho\left(t^n H'_+(R)\right)\subseteq t^{n-k}H'_+(R) \qquad\mbox{for all $n\in \mathbb{Z}$.}
\end{equation}
Here the image of $t^m H'_+(R)$ in $H'(R)$ is intended for $m\in\mathbb{Z}$.
This functor is represented by a group ind-scheme, here denoted $\mathbb{GL}_b\left(H'\right)$. 

The ind-scheme structure is explained as follows.
For each element $\rho$ in $\mathbb{GL}_b\left(H'\right)$, there exists $k\in\mathbb{N}$ such that \eqref{eq:rhohasak} holds. The closed sub-scheme $\mathbb{GL}_k\left(H'\right)$ consisting of such elements $\rho$ with a given $k$ is an affine scheme (specifically, an inverse limit of affine schemes of finite type, as for $\mathbb{A}\mathrm{ut}\left(\mathscr{O}\right)$ in \cite[\S6.2]{bzf}, see \S\ref{sec:DerAut}). 
The $\mathbb{GL}_k\left(H'\right)$ form an increasing chain of closed sub-schemes, and thus 
\[
\mathbb{GL}_b\left(H'\right):= \varinjlim\mathbb{GL}_k\left(H'\right)
\]
is a group ind-scheme, with multiplication and inverse maps induced by 
\[
\mathbb{GL}_k\left(H'\right)\times \mathbb{GL}_\ell\left(H'\right)\rightarrow \mathbb{GL}_{k+\ell}\left(H'\right)\qquad \mbox{and} \qquad\mathbb{GL}_k\left(H'\right)\rightarrow \mathbb{GL}_k\left(H'\right).
\]

Then consider the group functor which assigns to a $\mathbb{C}$-algebra $R$ the subgroup $\mathbb{S}\mathrm{p}^+\left(H'(R)\right)\subset \mathbb{GL}_b\left(H'(R)\right)$ of continuous symplectic automorphisms of $H' (R)$ preserving the subspace $H'_+(R)=tR\llbracket t \rrbracket$: 
\[
\mathbb{S}\mathrm{p}^+\left(H'(R) \right) := \left\{\rho\in \mathbb{GL}_b\left(H'(R)\right) \, \Bigg| \,
\begin{array}{l}
\langle \rho(a), \rho(b) \rangle = \langle a, b \rangle\\[5pt]
\rho\left( H'_+(R)\right)= H'_+(R) 
\end{array}
\right\}. 
\]
This is represented by a subgroup ind-scheme $\mathbb{S}\mathrm{p}^+\left(H' \right)\subset \mathbb{GL}_b\left(H'\right)$.

One easily checks that
\[
\mathrm{Lie}\left(\mathbb{S}\mathrm{p}^+\left(H' \right)\right)=\mathfrak{sp}^+_b\left(H' \right)
\]
where 
\[
\mathfrak{sp}^+_b\left(H' \right):= \left\{ X\in\mathfrak{sp}^+\left(H' \right) \,\, | \,\,\exists k\in\mathbb{N}: X(t^n H'_+)\subseteq t^{n-k}H'_+ \quad\mbox{for all $n$}\right\}.
\]
This is a dense Lie subalgebra of $\mathfrak{sp}^+\left(H' \right)$ containing the dense subset $S^2(H')\cap \mathfrak{sp}^+\left(H' \right)$. 

Moreover, for $(Z,F,L)$ in $\widehat{\mathcal{A}}_g$, let $\mathbb{S}\mathrm{p}^+_F\left(H'\right)$ be the subgroup ind-scheme of $\mathbb{S}\mathrm{p}^+\left(H' \right)$ defined as 
\[
\mathbb{S}\mathrm{p}^+_F\left(H'\right) := \left\{\rho\in \mathbb{S}\mathrm{p}^+\left(H' \right) \, \Bigg| \,
\begin{array}{l}
\rho\left( F^\perp\right)= F^\perp, \\[5pt]
\rho\left( F\right)= F, \\[5pt]
\rho _{F^\perp/F} = \mathrm{Id}_{F^\perp/F}
\end{array}
\right\}.
\]
Here $\rho _{F^\perp/F}$ stands for the linear automorphism of $F^\perp/F$ induced by $\rho$, well defined, since 
$\rho\left( F^\perp\right)= F^\perp$ and $\rho\left( F\right)= F$.
One easily checks that
\[
\mathrm{Lie}\left(\mathbb{S}\mathrm{p}^+_F\left(H'\right)\right)=\mathfrak{sp}^+_F\left(H' \right)\cap \mathfrak{sp}^+_b\left(H' \right),
\]
with $\mathfrak{sp}^+_F\left(H' \right)$ from \eqref{eq:sp+F}. 

Theorem \ref{thm:unif1} has the following extension:

\begin{theorem}
\label{thm:unif1+}
\begin{enumerate}[(i)]
\item The moduli space $\widehat{\mathcal{A}}_g$ carries a transitive action of $\mathfrak{sp} \left( H' \right)$ compatible with a  
transitive action of $\mathbb{S}\mathrm{p}^+\left(H' \right)$ on the fibers of  \mbox{$\widehat{\mathcal{A}}_g\rightarrow \mathcal{A}_g$}.
\item The group ind-scheme $\mathbb{S}\mathrm{p}^+_F\left(H'\right)$ equals the kernel of the restriction map 
\[
\mathrm{Stab}_{(Z,F,L)}\left(\mathbb{S}\mathrm{p}^+\left(H' \right)\right) \rightarrow \mathrm{Aut}\left( F^\perp/F, L\right).
\]
Here the source is the stabilizer of $\mathbb{S}\mathrm{p}^+\left(H' \right)$ at $(Z,F,L)$, and 
$\mathrm{Aut}\left( F^\perp/F, L\right)$ is the group of linear automorphisms of  $F^\perp/F$ that fix $L$.

\item The subgroup ind-schemes $\mathbb{S}\mathrm{p}^+_F\left(H'\right)$ obtained by varying $(Z,F)$ are pairwise conjugated inside $\mathbb{S}\mathrm{p}^+\left(H' \right)$.
\end{enumerate}
\end{theorem}

\begin{proof}
The natural action of $\mathbb{S}\mathrm{p}^+\left(H' \right)$ on $H'$ induces a transitive action of $\mathbb{S}\mathrm{p}^+\left(H' \right)$ on the fibers of $\widehat{\mathcal{H}}_g\rightarrow {\mathcal{H}}_g$. 
Explicitly,  
recall from \eqref{eq:Hghat} that the fibers of $\widehat{\mathcal{H}}_g\rightarrow {\mathcal{H}}_g$ are isomorphic to
$\widetilde{S}^2(H'_+) \times \mathcal{B}_g(H'_+)$. The space $\widetilde{S}^2(H'_+)$ consists of linear symmetric maps $\varphi\colon H_-\rightarrow H'_+$. The assignment $\varphi\mapsto Z:=\mathrm{graph}(\varphi)$ gives a bijection of $\widetilde{S}^2(H'_+)$ with the set of Lagrangian subspaces $Z$ of $H'$ such that $Z\cap H'_+=0$ --- this is \eqref{eq:cond1}.
Clearly, $\mathbb{S}\mathrm{p}^+\left(H' \right)$ acts transitively on the set of such Lagrangians, hence on $\widetilde{S}^2(H'_+)$. 
Moreover, 
the stabilizer of each such $Z$ in $\mathbb{S}\mathrm{p}^+\left(H' \right)$ maps onto
 the group $\mathbb{GL}_b(H'_+)$, for which $\mathcal{B}_g(H'_+)$ is a homogeneous space. Hence $\mathbb{S}\mathrm{p}^+\left(H' \right)$ acts transitively on the fibers of $\widehat{\mathcal{H}}_g\rightarrow {\mathcal{H}}_g$. 

The actions of $\mathrm{Sp}(2g, \mathbb{Z})$ and $\mathbb{S}\mathrm{p}^+\left(H' \right)$ on $\widehat{\mathcal{H}}_g$ commute, hence the $\mathbb{S}\mathrm{p}^+\left(H' \right)$--action descends to $\widehat{\mathcal{A}}_g=\widehat{\mathcal{H}}_g\,/\,\mathrm{Sp}(2g, \mathbb{Z})$, and it is transitive on the fibers of $\widehat{\mathcal{A}}_g\rightarrow {\mathcal{A}}_g$. 

Finally, the transitive action of $\mathfrak{sp} \left( H' \right)$ on $\widehat{\mathcal{A}}_g$ from Theorem \ref{thm:unif1} restricts to
the transitive action of $\mathfrak{sp}^+\left(H' \right)$ on the fiber of $\widehat{\mathcal{A}}_g\rightarrow  {\mathcal{A}}_g$ at each $(Z,F,L)$, and the actions of $\mathfrak{sp} \left( H' \right)$ and $\mathbb{S}\mathrm{p}^+\left(H' \right)$ are compatible: $\rho X \rho^{-1}= \mathrm{Ad}_\rho(X)$ as vector fields. This proves (i).

Part (ii) is clear by definition.
Part (iii) follows from the transitive action of $\mathbb{S}\mathrm{p}^+\left(H' \right)$ on the set of pairs $(Z, F)$.
\end{proof}

\subsection{The group ind-scheme $\mathbb{S}\mathrm{p}\left(H'\right)$}
More generally, consider the group functor which assigns to  a $\mathbb{C}$-algebra $R$ the group $\mathbb{S}\mathrm{p}\left(H'(R)\right)$  of  continuous linear automorphisms of $H' (R)$ of the following type 
\[
\mathbb{S}\mathrm{p}\left(H'(R)\right) := \left\{\rho\in \mathbb{GL}_b(H'(R)) \, \Bigg| \,
\begin{array}{l}
\langle \rho(a), \rho(b) \rangle = \langle a, b \rangle \quad \mbox{for all $a,b\in H'$,}\\[5pt]
\rho\left( H'_+(R)\right) \subseteq H'_+(R) \oplus H_-(R_{\mathrm{nil}})  
\end{array}
\right\}
\]
where $R_{\mathrm{nil}}\subseteq R$ is the ideal of nilpotent elements.
Let  $\mathbb{S}\mathrm{p}\left(H'\right)$ be the group ind-scheme  representing this functor. Naturally, $\mathbb{S}\mathrm{p}^+\left(H' \right)$ is a subgroup of $\mathbb{S}\mathrm{p}\left(H'\right)$.
One checks that 
\[
\mathrm{Lie}\left(\mathbb{S}\mathrm{p}\left(H'\right)\right)=\mathfrak{sp}_b\left(H' \right)
\]
where 
\[
\mathfrak{sp}_b\left(H' \right):= \left\{ X\in\mathfrak{sp}\left(H' \right) \,\, | \,\,\exists k\in\mathbb{N}: X(t^n H'_+)\subseteq t^{n-k}H'_+ \quad\mbox{for all $n$}\right\}.
\]
This is a dense Lie subalgebra of $\mathfrak{sp}\left(H' \right)$ containing the dense subset $S^2(H')$.

\subsection{Derivations and ring automorphisms}
\label{sec:DerAut}
Here we remark how the Lie algebras and group ind-schemes of this section naturally extend analogous objects studied in \cite{beilinson1991quantization, bzf} and
related to the geometry of curves. Namely, let $\mathrm{Der}(\mathscr{O})$ 
 be the functor assigning to a $\mathbb{C}$-algebra $R$ the Lie algebra
\[
\mathrm{Der}\left(\mathscr{O}(R)\right) = R\llbracket t \rrbracket \partial_t.
\]
This is  the Lie subalgebra of the Lie algebra $\mathrm{Witt}(R)$ from \S\ref{sec:WittVir} topologically generated over $R$ by $L_p$ for $p\geq -1$.
The inclusion $\tau\colon \mathrm{Witt} \hookrightarrow \mathfrak{sp}\left(H' \right)$ from \eqref{eq:tau} maps $\mathrm{Der}\left(\mathscr{O}\right)$ to $\mathfrak{sp}^+\left(H' \right)$, so that one has a commutative diagram of Lie algebras
\begin{equation}
\label{eq:Der0Osquare}
\begin{tikzcd}
\mathrm{Witt} \arrow[hookrightarrow]{r}{\tau} & \mathfrak{sp}\left(H' \right)\\
\mathrm{Der}\left(\mathscr{O}\right) \arrow[hookrightarrow]{u} \arrow[hookrightarrow]{r}{} & \mathfrak{sp}^+\left(H' \right). \arrow[hookrightarrow]{u}
\end{tikzcd}
\end{equation}

From \cite{adkp}, the moduli space $\widehat{\mathcal{M}}_g$ carries a transitive action of $\mathrm{Witt}$. This action is compatible with a  transitive action of $\mathrm{Der}\left(\mathscr{O}\right)$ on the fibers of \mbox{$\widehat{\mathcal{M}}_g\rightarrow \mathcal{M}_g$}: the elements $L_p$ with $p\geq 0$ topologically generate the tangent directions to the changes of the formal coordinate at the marked point, and the element $L_{-1}$ spans the tangent space to the curve at the  marked point.
Hence, for $(C, P, t)$ in $\widehat{\mathcal{M}}_g$, one has 
\begin{equation}
\label{eq:DCMg}
T_C\left( \mathcal{M}_g \right) \cong \mathrm{Vect}(C\setminus P) \setminus \mathrm{Witt} \, / \, \mathrm{Der}\left(\mathscr{O}\right).
\end{equation}

From \cite{adc91}, the moduli space $\widehat{\mathcal{A}}_g$ carries a transitive action of $\mathfrak{sp}\left(H' \right)$ and by quotienting $\mathrm{Witt}$ and $\mathfrak{sp}\left(H' \right)$ by appropriate Lie subalgebras, the inclusion $\tau$ from \eqref{eq:tau} induces the differential of the extended Torelli map $\widehat{\mathcal{M}}_g\rightarrow \widehat{\mathcal{A}}_g$.
We emphasize how \eqref{eq:DCAg} extends  \eqref{eq:DCMg}. 

\smallskip

One has a similar picture for the related group ind-schemes. Namely, 
consider first the group scheme $\mathbb{A}\mathrm{ut}\left(\mathscr{O}\right)$ representing the functor which assigns to a $\mathbb{C}$-algebra $R$ the group of continuous ring automorphisms of $R\llbracket t \rrbracket$ preserving the ideal $tR\llbracket t \rrbracket$. Explicitly, this is:
\[
\mathbb{A}\mathrm{ut}\left(\mathscr{O}(R)\right):=
\left\{ t\mapsto a_1\,t + a_2\,t^2+\dots \, | \: \mbox{$a_i\in R $ and $a_1$ is a unit} \right\}.
\]
Since the symplectic form \eqref{eq:symplectic} is invariant by changes of the parameter $t$ and $tR\llbracket t \rrbracket=H'_+(R)$ as vector spaces,
it follows that $\mathbb{A}\mathrm{ut}\left(\mathscr{O}\right)$ is a subgroup of $\mathbb{S}\mathrm{p}^+\left(H' \right)$. 
However, the Lie algebra of $\mathbb{A}\mathrm{ut}\left(\mathscr{O}\right)$ is smaller than $\mathrm{Der}\left(\mathscr{O}\right)$:
As remarked in \cite[\S6.2.3]{bzf}, one has 
\[
\mathrm{Lie}\left(\mathbb{A}\mathrm{ut}\left(\mathscr{O}\right)\right)=t \,\mathrm{Der}\left(\mathscr{O}\right),
\]
and in order to exponentiate $\mathrm{Der}\left(\mathscr{O}\right)$, one needs to consider a group ind-scheme. For this, let $\underline{\mathbb{A}\mathrm{ut}}\left(\mathscr{O}\right)$ be the group ind-scheme representing the functor which assigns to a $\mathbb{C}$-algebra $R$ the group of continuous ring automorphisms of $R\llbracket t \rrbracket$:
\[
\underline{\mathbb{A}\mathrm{ut}}\left(\mathscr{O}(R)\right):=
\left\{ t\mapsto a_0 + a_1\,t + a_2\,t^2+\dots \, \Bigg| \: 
\begin{array}{l}
\mbox{$a_i\in R $, $a_1$ is a unit,} \\[5pt]
\mbox{and $a_0$ is nilpotent} 
\end{array}
\right\}.
\]
Similarly, let $\mathbb{A}\mathrm{ut}\left(\mathcal{K}\right)$ be the group ind-scheme representing the functor which assigns to a $\mathbb{C}$-algebra $R$
the group of continuous ring automorphisms of $R(\!( t )\!)$:
\[
\mathbb{A}\mathrm{ut}\left(\mathcal{K}(R)\right):=
\left\{ t\mapsto \sum_{i\geq i_0} a_i \, t^i \, \Bigg| \: 
\begin{array}{l}
\mbox{$a_i\in R $, $a_1$ is a unit,} \\[5pt]
\mbox{$a_i$ nilpotent for $i\leq 0$}
\end{array}
 \right\}.
\]
One has 
\begin{align*}
\mathrm{Lie}\left(\underline{\mathbb{A}\mathrm{ut}}\left(\mathscr{O}\right)\right)=\mathrm{Der}\left(\mathscr{O}\right)
\qquad \mbox{and}\qquad
\mathrm{Lie}\left(\mathbb{A}\mathrm{ut}\left(\mathcal{K}\right)\right)=\mathrm{Witt}.
\end{align*}
We refer to \cite{beilinson1991quantization, bzf} for more on $\mathrm{Der}\left(\mathscr{O}\right)$, $\underline{\mathbb{A}\mathrm{ut}}\left(\mathscr{O}\right)$ and $\mathbb{A}\mathrm{ut}\left(\mathcal{K}\right)$.

Elements in $\underline{\mathbb{A}\mathrm{ut}}\left(\mathscr{O}\right)$ and $\mathbb{A}\mathrm{ut}\left(\mathcal{K}\right)$ naturally induce elements in $\mathbb{GL}_b(H)$, and after composing with the projection  $H\rightarrow H'$, they induce symplectic elements in $\mathbb{GL}_b(H')$.
Hence, one has a commutative diagram of group ind-schemes
\[
\begin{tikzcd}
\mathbb{A}\mathrm{ut}\left(\mathcal{K}\right) \arrow[hookrightarrow]{r} & \mathbb{S}\mathrm{p}\left(H'\right)\\
\underline{\mathbb{A}\mathrm{ut}}\left(\mathscr{O}\right) \arrow[hookrightarrow]{u} \arrow[hookrightarrow]{r}{} & \mathbb{S}\mathrm{p}^+\left(H' \right) \arrow[hookrightarrow]{u}
\end{tikzcd}
\]
such that the differentials of these maps restricted to the tangent spaces at the identities are given by the restriction of \eqref{eq:Der0Osquare} to $\mathfrak{sp}_b\left(H' \right)\subset \mathfrak{sp}\left(H' \right)$.

\subsection{The group ind-scheme $\mathbb{S}\mathrm{p}^+\left(H' \right)\ltimes \mathscr{O}_1^\times$}
\label{sec:Sp+Ox}
Finally, we describe  the analogous Lie algebras and group ind-schemes for the moduli space $\widehat{\mathcal{X}}_g$.
The transitive action of $\mathfrak{sp} \left( H' \right)\ltimes H'$ on $\widehat{\mathcal{X}}_g$ from \S\ref{sec:unif2} restricts to a transitive action of
$\mathfrak{sp}^+\left(H' \right)\ltimes H'_+$ on the fibers of $\widehat{\mathcal{X}}_g\rightarrow  {\mathcal{X}}_g$. 
Similar to the double coset realization of the tangent space of ${\mathcal{A}}_g$ in \eqref{eq:DCAg}, 
for $x\in \widehat{\mathcal{X}}_g$ and its image $\overline{x}$ via $\widehat{\mathcal{X}}_g\rightarrow  {\mathcal{X}}_g$, one has  
\begin{equation}
\label{eq:DCXg}
T_{\overline{x}} \left( \mathcal{X}_g \right) \cong  \mathfrak{i}(x) \, \setminus \left( \mathfrak{sp}  \left( H' \right) \ltimes H' \right) / \left( \mathfrak{sp}^+\left(H'\right) \ltimes H'_+\right).
\end{equation}

Let $\mathscr{O}_1^\times$ be the group scheme of invertible Taylor series with constant term~$1$. 
This is a subgroup of the group scheme $\mathscr{O}^\times$ of invertible Taylor series from \cite[\S 20.3.4]{bzf}. One has $\mathrm{Lie}\left(\mathscr{O}^\times \right)=H_+$ and $\mathrm{Lie}\left(\mathscr{O}_1^\times \right)=H'_+$.
An element $\exp(u)\in \mathscr{O}_1^\times$ naturally acts on $H'$ as translation by $u$.
The natural action of $\mathbb{S}\mathrm{p}^+\left(H' \right)\ltimes \mathscr{O}_1^\times$ on $H'$ by affine transformations 
\[
(\rho, \exp(u))\colon \ell\mapsto \rho(\ell)+u \quad \mbox{for $\ell\in H'$}
\]
induces a natural action of $\mathbb{S}\mathrm{p}^+\left(H' \right)\ltimes \mathscr{O}_1^\times$ on the fibers of $\widehat{\mathcal{X}}_g\rightarrow {\mathcal{X}}_g$. 
For $x=(Z, F, L, \overline{h}, q)$ in $\widehat{\mathcal{X}}_g$, define
\begin{equation}
\label{eq:stabxinXg}
\left(\mathbb{S}\mathrm{p}^+\left(H' \right)\ltimes \mathscr{O}_1^\times\right)_x :=\left\{ 
(\rho, \exp(u_\rho))\,:\, \rho\in \mathbb{S}\mathrm{p}^+_F\left(H'\right)
\right\}
\end{equation}
where for each $\rho$, the element $u_\rho$ is defined as the projection in $H'_+$ of $(1-\rho)h\in F\oplus H'_+$ for some lift $h\in H'$ of $\overline{h}\in H'/K$. The definition of $u_\rho$ follows from the condition that $(\rho, \exp(u_\rho))$ fix $\overline{h}$, equivalently, 
\[
\rho(h)+u_\rho \equiv h \quad\mbox{mod $K$.}
\]
The element $(1-\rho)h$ is indeed in $F\oplus H'_+$: writing $h=a+v$ for some $a\in F^\perp$ and $v\in H'_+$, one has
$(1-\rho)a\in F$ since $\rho$ is the identity on $F^\perp/F$, and $(1-\rho)v\in H'_+$ since $\rho(H'_+)\subseteq H'_+$.
In particular, $u_\rho$ is independent of $h\in \overline{h}+K$. Similarly, $u_\rho$ is independent of the choice of $(\overline{h},q)$ under the action of the group of points of order $2$ in $H'/K$ from \S\ref{sec:extunivPPAV}.

One has
\[
\mathrm{Lie} \left(\mathbb{S}\mathrm{p}^+\left(H' \right)\ltimes \mathscr{O}_1^\times\right)_x = \mathfrak{i}(x)\cap \left( \mathfrak{sp}^+_b\left(H'\right) \ltimes H'_+\right).
\]
Finally, Theorem \ref{thm:unif2} has the following extension:

\begin{theorem}
\label{thm:unif2+}
\begin{enumerate}[(i)]
\item The moduli space $\widehat{\mathcal{X}}_g$ carries a transitive action of \mbox{$\mathfrak{sp} \left( H' \right)\ltimes H'$} compatible with the transitive action of \sloppy \mbox{$\mathbb{S}\mathrm{p}^+\left(H' \right)\ltimes \mathscr{O}_1^\times$} on the fibers of  \mbox{$\widehat{\mathcal{X}}_g\rightarrow \mathcal{X}_g$}.
\item The group ind-scheme $\left(\mathbb{S}\mathrm{p}^+\left(H' \right)\ltimes \mathscr{O}_1^\times\right)_x$ equals the kernel of the restriction map 
\[
\mathrm{Stab}_x\left( \mathbb{S}\mathrm{p}^+\left(H' \right)\ltimes \mathscr{O}_1^\times\right) \rightarrow \mathrm{Aut}\left( F^\perp/F, L\right).
\]
\end{enumerate}
\end{theorem}

\begin{proof}
Part (i) is as in Theorem \ref{thm:unif1+}(i).
For part (ii), if $(\rho, \exp(u)) \in \mathrm{Stab}_x$ and $\rho_{F^\perp/F} = \mathrm{id}$, then $\rho \in \mathbb{S}\mathrm{p}^+_F\left(H' \right)$ and $u \equiv(1 - \rho)h$ mod $K$, so $u - u_\rho$ is in $H'_+\cap K$; from \eqref{eq:cond2}, one has $u = u_\rho$. Conversely $(\rho, \exp(u_\rho))$ fixes $x$ by construction.
\end{proof}

%%%%%%%%%%%%%%%%%%%%%%%%%%%%
%%%%%%%%%%%%%%%%%%%%%%%%%%%%
%%%%%%%%%%%%%%%%%%%%%%%%%%%%

\section{Harish--Chandra localization}
\label{sec:geoHCp}

Here we review the formalism of localization of modules over Harish--Chandra pairs first introduced in \cite{bb} and further developed in \cite{bfm, besh} and \cite[\S 17.2]{bzf}. This will be used in \S\ref{sec:sheafV}
for the constructions of the sheaves of coinvariants on $\mathcal{A}_g$ and $\mathcal{X}_g$.

\subsection{Harish--Chandra pairs and moduli of curves}
\label{sec:HCloc}
A \textit{Harish--Chandra pair} $(\mathfrak{g}, K)$ consists of a Lie algebra $\mathfrak{g}$ and a group (ind-)scheme $K$ 
equipped with an embedding $\mathfrak{k}:=\mathrm{Lie}(K)\hookrightarrow \mathfrak{g}$ and an action $\mathrm{Ad}$ of $K$ on $\mathfrak{g}$. The action $\mathrm{Ad}$ is required to restrict to the adjoint action of $K$ on $\mathfrak{k}\subset \mathfrak{g}$, and the differential of $\mathrm{Ad}$ at the identity element is required to be equal to the action of $\mathfrak{k}$ on $\mathfrak{g}$.

Let $V$ be a $(\mathfrak{g}, K)$-module, i.e., a vector space with compatible actions of $\mathfrak{g}$ and $K$. 
An action of $K$ on $V$ means $R$-linear actions of $K(R)$ on $V \otimes_{\mathbb{C}} R$, functorial in $R$, whose differential is the $\mathfrak{g}$-action.
Compatibility of the actions means $\rho X \rho^{-1}= \mathrm{Ad}_\rho(X)$ for $\rho\in K$ and $X\in\mathfrak{g}$.
The localization functor assigns to $V$ a (possibly twisted) $\mathscr{D}$-module on a variety or stack identified by the Harish--Chandra pair. 
For instance,  localizations of: 

\begin{enumerate}[---]

\item modules over $\left(\mathrm{Der}\left(\mathscr{O}\right),  \mathbb{A}\mathrm{ut}\left(\mathscr{O}\right)\right)$
 yield $\mathscr{D}$-modules on a smooth curve  (see notation from \S\ref{sec:DerAut}); 
 
\item modules over $\left(\mathrm{Witt}, \mathbb{A}\mathrm{ut}\left(\mathscr{O}\right)\right)$
 yield $\mathscr{D}$-modules on the moduli space $\mathcal{M}_{g,1}$ of pointed smooth curves of genus $g$; and 
 
\item modules over $\left(\mathrm{Vir}, \mathbb{A}\mathrm{ut}\left(\mathscr{O}\right)\right)$
yield \textit{twisted} $\mathscr{D}$-modules on $\mathcal{M}_{g,1}$. 
\end{enumerate}
We refer to \cite[Ch.~17--18]{bzf} for a discussion of these and more geometries.

\subsection{Harish--Chandra pairs and moduli of abelian varieties}
\label{sec:HCpairsmp}
It is easy to verify that 
\[
\left(\mathfrak{sp} \left( H' \right), \, \mathbb{S}\mathrm{p}^+\left(H' \right)\right)
\qquad
\mbox{and}
\qquad
\left( \mathfrak{sp}\left(H' \right) \ltimes H', \, \mathbb{S}\mathrm{p}^+\left(H' \right)\ltimes  \mathscr{O}_1^\times \right),
\]
with notation as in \S\ref{sec:groupind}, are Harish--Chandra pairs.
Also, one has that
\[
\left(\mathfrak{mp} \left( H' \right), \, \mathbb{S}\mathrm{p}^+\left(H' \right)\right)
\qquad
\mbox{and}
\qquad
\left( \widetilde{\mathscr{U}}_{2}(H), \, \mathbb{S}\mathrm{p}^+\left(H' \right)\ltimes  \mathscr{O}_1^\times \right)
\]
are Harish--Chandra pairs.
Indeed, the main point is to show the required Lie algebra embeddings. For this, it is immediate to see that  the two-cocycle on $\mathfrak{sp}\left(H' \right)$ defining $\mathfrak{mp}\left(H' \right)$ from \eqref{eq:cocyclemeta} vanishes on the Lie algebra $\mathfrak{sp}^+\left(H' \right)$ from \S\ref{sec:sp+}, hence one has a Lie algebra splitting 
\begin{equation}
\label{eq:sp+inmp}
\mathfrak{sp}^+_b\left(H' \right)\subseteq \mathfrak{sp}^+\left(H' \right)\hookrightarrow \mathfrak{mp}\left(H' \right).
\end{equation}
Similarly, the two-cocycle on $\mathfrak{sp}\left(H' \right)\ltimes H'$ defining $\widetilde{\mathscr{U}}_{2}(H)$ from \eqref{eq:2cocycledefUtildeleq2} vanishes on $\mathfrak{sp}^+\left(H' \right)\ltimes H'_+$, hence one has a Lie algebra splitting 
\begin{equation}
\label{eq:sp+H+inUleq2}
\mathfrak{sp}^+_b\left(H' \right)\ltimes H'_+\subset \mathfrak{sp}^+\left(H' \right)\ltimes H'_+\hookrightarrow \widetilde{\mathscr{U}}_{2}(H).
\end{equation}
Finally, $\mathbb{S}\mathrm{p}^+\left(H' \right)\ltimes  \mathscr{O}_1^\times$ acts on $\widetilde{\mathscr{U}}_{2}(H)$ via its affine action on $H'$, with $\mathbb{S}\mathrm{p}^+\left(H' \right)$ preserving  $\mathfrak{mp} \left( H' \right)$, the differential of this action is ad of the lifts \eqref{eq:sp+inmp}-\eqref{eq:sp+H+inUleq2}, and \eqref{eq:sp+inmp}-\eqref{eq:sp+H+inUleq2} are equivariant for this action.

We emphasize a difference here with respect to the examples in \S\ref{sec:HCloc} involving moduli spaces of curves.
The localization functor mentioned there yields (possibly twisted) $\mathscr{D}$-modules on $\mathcal{M}_{g,1}$ via descent along the 
principal $\mathbb{A}\mathrm{ut}\left(\mathscr{O}\right)$-bundle $\widehat{\mathcal{M}}_g \rightarrow\mathcal{M}_{g,1}$.
We will construct (possibly twisted) $\mathscr{D}$-modules on $\mathcal{A}_{g}$ via descent along $\widehat{\mathcal{A}}_g \rightarrow\mathcal{A}_{g}$.
However, while the group $\mathbb{A}\mathrm{ut}\left(\mathscr{O}\right)$ acts simply transitively on the fibers of $\widehat{\mathcal{M}}_g \rightarrow\mathcal{M}_{g,1}$, the group $\mathbb{S}\mathrm{p}^+\left(H' \right)$  acts  transitively, but not simply, on the fibers of $\widehat{\mathcal{A}}_g \rightarrow\mathcal{A}_g$, as discussed in \S\ref{sec:groupind}. 

Thus, considering the quotient stack $[\widehat{\mathcal{A}}_g/ \mathbb{S}\mathrm{p}^+\left(H' \right)]$, the map 
$\widehat{\mathcal{A}}_g \rightarrow\mathcal{A}_g$ is realized as the composition 
\[
\widehat{\mathcal{A}}_g \rightarrow [\widehat{\mathcal{A}}_g/\mathbb{S}\mathrm{p}^+\left(H' \right)]
\rightarrow \mathcal{A}_g
\]
where the first map is a principal $\mathbb{S}\mathrm{p}^+\left(H' \right)$-bundle, and the second one is the forgetful map.
Similarly, the map 
$\widehat{\mathcal{X}}_g \rightarrow\mathcal{X}_g$ is realized as the composition 
\[
\widehat{\mathcal{X}}_g \rightarrow [\widehat{\mathcal{X}}_g/(\mathbb{S}\mathrm{p}^+\left(H' \right)\ltimes \mathscr{O}_1^\times)]\rightarrow \mathcal{X}_g.
\]

For this, we need the following result from \cite{bb} on the descent to a quotient stack, more general than the descent along principal bundles treated in \cite{bzf}.
The final descent from the quotient stack to $\mathcal{A}_g$ (respectively, $\mathcal{X}_g$) will be treated in \S\ref{sec:descentAg} (resp., \S\ref{sec:descentXg}).

\subsection{$\mathscr{D}$-algebras on quotient stacks}
\label{sec:BBDmodquotstack}
We review a result from \cite{bb}  characterizing  $\mathscr{D}$-algebras and their modules on quotient stacks. We start with the definitions of Harish--Chandra Lie algebroids and their modules.

Let $S$ be an arbitrary smooth scheme.
A \textit{Lie algebroid} on $S$ is a quasicoherent $\mathscr{O}_S$-module $P$ equipped with a map of $\mathscr{O}_S$-modules $\alpha\colon P\rightarrow \mathscr{T}_S$ and a $\mathbb{C}$-linear map $[,]\colon P\otimes_{\mathbb{C}} P\rightarrow P$ such that: $[,]$ is a Lie algebra bracket, $\alpha$ is a Lie algebra map, and for local sections $p_1,p_2\in P$ and $f\in \mathscr{O}_S$, one has
\[
\left[p_1, fp_2\right] = f[p_1, p_2] + \alpha(p_1)(f)p_2.
\]
Examples of Lie algebroids are given by Atiyah algebras of  vector bundles.

To a Lie algebroid $\alpha\colon P\rightarrow \mathscr{T}_S$ on $S$, one may assign its \textit{universal enveloping $\mathscr{D}$-algebra} $\mathscr{U}(P)$. This is the associative $\mathscr{O}_S$-algebra generated by a subalgebra $i\colon \mathscr{O}_S\rightarrow \mathscr{U}(P)$ and a Lie subalgebra \mbox{$j\colon P\rightarrow \mathscr{U}(P)$} subject to the relations
\[
j(fp)=i(f)j(p) \qquad \mbox{and}\qquad [j(p), i(f)]= i(\alpha(p)\cdot f)
\]
for $p\in P$ and $f\in \mathscr{O}_S$. 
The $\mathscr{O}_S$-algebra $\mathscr{U}(P)$ is a \textit{$\mathscr{D}$-algebra}, as defined in \cite[\S 1.1]{bb}. The functor assigning to a Lie algebroid $P$ its universal enveloping $\mathscr{D}$-algebra $\mathscr{U}(P)$ is a functor from the category of Lie algebroids to the category of $\mathscr{D}$-algebras and admits a right adjoint functor, see \cite[\S 1.2.5]{bb}. Modules for $\mathscr{U}(P)$ are the same as $\mathscr{O}_S$-modules with a compatible $P$-action. Modules for $\mathscr{U}(\mathscr{T}_S)$ are called \textit{$\mathscr{D}$-modules}, and modules for $\mathscr{U}(P)$, 
where $P$ is an extension of $\mathscr{T}_S$ by $\mathscr{O}_S$,
are called \textit{twisted $\mathscr{D}$-modules}.

Next we review the equivariant setting. 
Let $K$ be an algebraic group, and let $X$ be a smooth scheme equipped with a $K$-action \mbox{$\mu\colon K\times X \rightarrow X$.}
 A quasicoherent $\mathscr{O}_X$-module $P$ is \textit{$K$-equivariant} if there exists an isomorphism $\mu_P\colon \mu^*P \rightarrow p_X^*P$ of $\mathscr{O}_{K\times X}$-modules that satisfies the cocycle condition (see \cite[03LF]{TheStacksproject}).
Here $p_X\colon K\times X \rightarrow X$ is the natural projection. 

A quasi-coherent sheaf on the quotient stack $[X/K]$ is the same as a $K$-equivariant quasi-coherent sheaf on $X$.

A \textit{Harish--Chandra Lie algebroid} on $X$ is a Lie algebroid $\alpha\colon P\rightarrow \mathscr{T}_X$ equipped with a $K$-equivariant structure as an $\mathscr{O}_X$-module and a Lie algebra map $\iota_{\mathfrak{k}}\colon\mathfrak{k}\rightarrow P$ with $\mathfrak{k}:=\mathrm{Lie}(K)$ such that:
\begin{enumerate}[---]
\item One has $k^*[p_1, p_2] = \left[k^* p_1, k^* p_2\right]$ and $k^*\left(\alpha(p) \right)= \alpha\left(k^* p \right)$ for $k\in K$ and $p_i, p\in P$. Here, $k^*$ is induced by the $K$-equivariant structure on $P$;

\item One has $\iota_{\mathfrak{k}}( \mathrm{Ad}_k \,\gamma)=k^*\iota_{\mathfrak{k}}(\gamma)$ for $k\in K$ and $\gamma \in \mathfrak{k}$;

\item The action of $\mathfrak{k}$ on $P$ induced from the $K$-equivariant structure on $P$ coincides with $\mathrm{ad}_{\iota_{\mathfrak{k}}}$.
\end{enumerate}
If $X$ is a point, a {Harish--Chandra Lie algebroid} on $X$ is the same as a Harish--Chandra pair as in \S\ref{sec:HCloc}.

As for Lie algebroids, to a Harish--Chandra Lie algebroid $P$ on $X$, one may assign its universal enveloping $\mathscr{D}$-algebra $\mathscr{U}(P)$. This is now 
a \textit{Harish--Chandra algebra}, i.e., a $\mathscr{D}$-algebra equipped with a $K$-action, as in \cite[\S 1.8.3]{bb}. The functor assigning 
to a Harish--Chandra Lie algebroid $P$ its universal enveloping $\mathscr{D}$-algebra $\mathscr{U}(P)$ is a functor from the category of Harish--Chandra Lie algebroids to the category of Harish--Chandra algebras and admits a right adjoint functor, see \cite[\S 1.8.4]{bb}.

For a Harish--Chandra algebra $\mathscr{U}$ on $X$, a \textit{$\left(\mathscr{U}, K \right)$-module} is a $\mathscr{U}$-module $M$ with a $K$-equivariant structure as an $\mathscr{O}_X$-module compatible with the $K$-equivariant structure on $\mathscr{U}$ (i.e., $k^*(um)=k^*(u)k^*(m)$ for $k\in K$, $u\in \mathscr{U}$ and $m\in M$) and such that the action of $\mathfrak{k}$ on $M$ induced from the $K$-equivariant structure coincides with the one coming from $\iota_{\mathfrak{k}}\colon\mathfrak{k}\rightarrow \mathscr{U}$ and the $\mathscr{U}$-module structure on $M$.

We will apply the following key result from \cite{bb}:

\begin{lemma}[{\cite[Lemma 1.8.7]{bb}}]
\label{lemma:descentquotstack}
\begin{enumerate}[(i)]
\item Harish--Chandra algebras on $X$ are the same as $\mathscr{D}$-algebras on $[X/K]$.

\item For a Harish--Chandra  algebra $\mathscr{U}$, 
the $\left(\mathscr{U}, K \right)$-modules are the same as the modules over the corresponding $\mathscr{D}$-algebra on $[X/K]$.
\end{enumerate}
\end{lemma}

\begin{example}
When the action of $K$ is free, Lemma \ref{lemma:descentquotstack}(ii) is the usual descent of $\left(\mathscr{U}, K \right)$-modules  along the principal $K$-bundle $X\rightarrow [X/K]$ to twisted $\mathscr{D}$-modules on $[X/K]$, see \cite[\S1.8.9]{bb}.
\end{example}

The definitions in this subsection and Lemma \ref{lemma:descentquotstack} are statements about equivariant quasi-coherent sheaves on $X$, $K \times X$ and $K \times K \times X$ and their functors of points, and \cite[\S 1.8]{bb} uses no further property of $K$; they therefore apply verbatim to a group ind-scheme $K$ acting on a (not necessarily finite-type) scheme or ind-scheme $X$, which is the setting used here. Since $K$ is the direct limit of affine schemes $K_k$, an action and an equivariant structure amount to compatible data on $K_k \times X$ and $K_k \times K_\ell \times X$.

\subsection{Harish--Chandra localization on quotient stacks}
Let $K$ and $X$ be as in \S\ref{sec:BBDmodquotstack}.
Let $\mathfrak{g}$ be a Lie algebra such that $\alpha\colon \mathfrak{g}\otimes_\mathbb{C} \mathscr{O}_X\rightarrow \mathscr{T}_X$ is a Harish--Chandra Lie algebroid on $X$ with $\alpha$ surjective.

\begin{proposition}
\label{prop:HClocvar1}
For a $(\mathfrak{g}, K)$-module $V$, 
the diagonal action of $K$ on $V\otimes_{\mathbb{C}} \mathscr{O}_X$ preserves $\mathrm{Ker}\,\alpha \cdot (V\otimes_{\mathbb{C}} \mathscr{O}_X)$, hence it induces a $K$-equivariant structure on 
the sheaf
\begin{equation}
\label{eq:IVhatV}
\widehat{\mathbb{V}}(V) := V\otimes_{\mathbb{C}} \mathscr{O}_X / \mathrm{Ker}\,\alpha \cdot (V\otimes_{\mathbb{C}} \mathscr{O}_X)
\end{equation}
making it a $\left(\mathscr{U}, K \right)$-module for the Harish--Chandra algebra $\mathscr{U} = \mathscr{U}\left(\mathscr{T}_X\right)$. Thus $\widehat{\mathbb{V}}(V)$ corresponds to a 
$\mathscr{D}$-module on $[X/K]$.
\end{proposition}

\begin{proof}
The statement that $\mathrm{Ker}\,\alpha \subset \mathfrak{g}\otimes_\mathbb{C} \mathscr{O}_X$ is $K$-stable follows from the fact that $\alpha$ is $K$-equivariant, and the action is compatible by the Harish--Chandra axioms. The last statement follows from Lemma \ref{lemma:descentquotstack}.
\end{proof}

Furthermore, assume that $\mathfrak{g}$ has a central extension 
$
0\rightarrow \mathbb{C} \rightarrow \widehat{\mathfrak{g}} \rightarrow \mathfrak{g} \rightarrow 0
$
that admits $K$-equivariant splittings $\mathfrak{k}\hookrightarrow \widehat{\mathfrak{g}}$
and $\mathrm{Ker}\,\alpha\hookrightarrow \widehat{\mathfrak{g}} \otimes_\mathbb{C} \mathscr{O}_X$,
the latter with image an ideal, so that
$\widehat{\mathfrak g}\otimes_{\mathbb C}\mathscr O_X/\mathrm{Ker}\,\alpha$ is a Lie algebroid.

\begin{proposition}
\label{prop:HClocvar2}
For a $(\widehat{\mathfrak{g}}, K)$-module $V$, the sheaf $\widehat{\mathbb{V}}(V)$ as in \eqref{eq:IVhatV} 
is a
$\left(\mathscr{U}, K \right)$-module for the Harish--Chandra algebra $\mathscr{U} = \mathscr{U}\left(\widehat{\mathfrak{g}}\otimes_\mathbb{C} \mathscr{O}_X/\mathrm{Ker}\,\alpha\right)$ and corresponds to a twisted
$\mathscr{D}$-module $\mathbb{V}(V)$ on $[X/K]$.
\end{proposition}

\begin{proof}
The action of $\widehat{\mathfrak{g}}$ on $V$ and on $X$  via the projection $\widehat{\mathfrak{g}}\rightarrow \mathfrak{g}$
 induce an action of the Lie algebroid $\widehat{\mathfrak{g}}\otimes_\mathbb{C} \mathscr{O}_X$ on $V\otimes_{\mathbb{C}} \mathscr{O}_X$.
 By the $K$-equivariance of the splitting $\mathrm{Ker}\,\alpha\hookrightarrow \widehat{\mathfrak{g}} \otimes_\mathbb{C} \mathscr{O}_X$, the subsheaf $\mathrm{Ker} \, \alpha\cdot (V\otimes_{\mathbb{C}} \mathscr{O}_X)$ is $K$-stable, 
  so $\widehat{\mathfrak{g}}\otimes_\mathbb{C} \mathscr{O}_X/\mathrm{Ker}\,\alpha$ acts on 
$\widehat{\mathbb{V}}(V)$ compatibly with the $K$-action, as in Proposition \ref{prop:HClocvar1}. The last statement follows from Lemma \ref{lemma:descentquotstack}.
\end{proof}

In case the Harish--Chandra Lie algebroid $\widehat{\mathfrak{g}}\otimes_\mathbb{C} \mathscr{O}_X/\mathrm{Ker}\,\alpha$ on $X$ is isomorphic to the Atiyah algebra $\mathscr{F}_L$ of the pull-back of a line bundle $L$ from $[X/K]$ as Harish--Chandra Lie algebroids, one has:

\begin{proposition}
\label{prop:HClocvar3}
For a $(\widehat{\mathfrak{g}}, K)$-module $V$,
the twisted $\mathscr{D}$-module structure on $\mathbb{V}(V)$ is induced by the action of $\mathscr{F}_L$.
\end{proposition}

\begin{proof}
From Proposition \ref{prop:HClocvar2}, $\widehat{\mathbb{V}}(V)$ is a $\left(\mathscr{U}, K \right)$-module for the Harish--Chandra algebra $\mathscr{U} = \mathscr{U}\left(\mathscr{F}_L\right)$ on $X$. Thus from Lemma \ref{lemma:descentquotstack},
$\mathbb{V}(V)$ is a module for the $\mathscr{D}$-algebra $\mathscr{U}\left(\mathscr{F}_L\right)$ on $[X/K]$. Hence the statement.
\end{proof}

\begin{remark}
When the action of $K$ is free, Propositions \ref{prop:HClocvar1}--\ref{prop:HClocvar3} are in \cite[\S 17.2]{bzf}.
The statements here are similar and follow from Lemma \ref{lemma:descentquotstack}.
\end{remark}

%%%%%%%%%%%%%%%%%%%%%%%%%%%%
%%%%%%%%%%%%%%%%%%%%%%%%%%%%
%%%%%%%%%%%%%%%%%%%%%%%%%%%%

\section{On the cohomology of the Lie algebras}
\label{sec:cohLie}

Here we study the cohomology of the Lie algebras $\mathfrak{sp} \left( H' \right)$ and $\mathfrak{sp} \left( H' \right)\ltimes H'$. 
We  will then show in Theorem \ref{thm:caniso} that for $g \geq 2$ their $H^2$ spaces are canonically isomorphic to $H^2\left( \mathcal{A}_g, \mathbb{C}\right)$ and $H^2\left( \mathcal{X}_g, \mathbb{C}\right)$, respectively. 

We equip $\mathfrak{sp} \left( H' \right)\ltimes H'$ with the topology induced from the topology of the Weyl algebra $\widetilde{\mathscr{U}}(H)$ in \S\ref{sec:Weyl}.
In particular, a basis of neighbourhoods of $0$ in $\mathfrak{sp} \left( H' \right)$ is given by the subspaces 
\begin{equation}
\label{eq:openneigh0}
\left\{X\in \mathfrak{sp} \left( H' \right) \, : \, X\left(t^{1-n}H'_+\right)\subseteq t^nH'_+\right\} \qquad \mbox{for $n\geq 1$}.
\end{equation}
For a topological Lie algebra $\mathfrak{g}$, let $H^*(\mathfrak{g}, \mathbb{C})$ be its continuous Lie algebra cohomology with complex coefficients.

Recall the two-cocycle $\psi$ of $\mathfrak{gl}(H)$ from $\eqref{eq:psi}$, and consider the  two-cocycles on $\mathfrak{sp} \left( H' \right)\ltimes H'$ 
\begin{align*}
\alpha \left( X +f,Y+g \right) &:= \psi(X,Y),\\
\beta \left( X +f,Y+g \right) &:= \psi(f,g) = -\mathrm{Res}_{t=0}  \, f \, dg,
\end{align*}
for $X,Y\in \mathfrak{sp}\left(H' \right)$ and $f,g\in H'$, where $f,g$ act on $H$ by multiplication. 
Their restrictions to $\mathrm{Witt}\ltimes H'$, with $\mathrm{Witt}\subset\mathfrak{sp}(H')$
via $\tau$, are still denoted $\alpha$ and $\beta$. 
On $\mathrm{Witt}\ltimes H'$, consider
moreover the two-cocycle
\[
\gamma(f\partial_t+g,h\partial_t+k):=\psi(f\partial_t,k)-\psi(h\partial_t,g).
\]
Hence $\psi=\alpha+\beta+\gamma$ on $\mathrm{Witt}\ltimes H'$. 
Explicitly, one has
\begin{align*}
\alpha \left( f\partial_t +g, h\partial_t+k \right) &= \frac{1}{6}\,\mathrm{Res}_{t=0} \, f \, dh'', \\
\gamma \left( f\partial_t +g, h\partial_t+k \right) &= -\frac{1}{2}\, \mathrm{Res}_{t=0}  \left( f \, dk' - h \, dg' \right).
\end{align*}
Here we use the notation $\ell' := \partial_t \ell$ for $\ell\in H'$.
Similarly, consider the two-cocycle $\alpha(X,Y):=\psi(X,Y)$ on $\mathfrak{sp} \left( H' \right)$.
It is shown in \cite[2.1]{adkp} that $H^1\left(\mathrm{Witt}, \mathbb{C} \right)= H^1\left( \mathrm{Witt}\ltimes H', \mathbb{C} \right)=0$, and 
\[
H^2\left(\mathrm{Witt}, \mathbb{C} \right)= \mathbb{C}\,\overline{\alpha} \qquad \mbox{and}\qquad
H^2\left( \mathrm{Witt}\ltimes H', \mathbb{C} \right)= \mathbb{C}\,\overline{\alpha}\oplus\mathbb{C}\,\overline{\beta}\oplus\mathbb{C}\,\overline{\gamma}
\]
where $\overline{\alpha}, \overline{\beta}$, and $\overline{\gamma}$ are the classes of $\alpha, \beta$, and $\gamma$, respectively.
Recall  $\tau\colon \mathrm{Witt} \hookrightarrow \mathfrak{sp}\left(H' \right)$ from \eqref{eq:tau} and $\sigma\colon \mathrm{Witt} \ltimes H' \hookrightarrow \mathfrak{sp}\left(H' \right) \ltimes H'$ from~\eqref{eq:sigma}.

\begin{proposition}
\label{prop:cohLiespH}
One has 
\begin{align}
\label{eq:H1_1}
H^1\left(\mathfrak{sp} \left( H' \right), \mathbb{C}\right)& =0,\\ 
\label{eq:H1_2}
H^1\left( \mathfrak{sp} \left( H' \right)\ltimes H', \mathbb{C} \right) &=0,\\
\label{eq:H2sp}
H^2\left(\mathfrak{sp} \left( H' \right), \mathbb{C}\right) &= \mathbb{C} \,\overline{\alpha},\\
\label{eq:H2spH}
H^2\left(\mathfrak{sp} \left( H' \right)\ltimes H', \mathbb{C} \right) &= \mathbb{C}\,\overline{\alpha} \oplus \mathbb{C}\,\overline{\beta}.
\end{align}
Moreover, the pull-back via $\tau$ induces the identification
\[
\tau^*\colon H^2\left(\mathfrak{sp} \left( H' \right), \mathbb{C}\right) \xrightarrow{=} H^2\left(\mathrm{Witt}, \mathbb{C}\right),
\]
and the pull-back via $\sigma$ induces the injection
\[
\sigma^*\colon H^2\left(\mathfrak{sp} \left( H' \right)\ltimes H', \mathbb{C} \right) \hookrightarrow H^2\left(\mathrm{Witt}\ltimes H', \mathbb{C}\right),
\quad \overline{\alpha}\mapsto \overline{\alpha}, \quad \overline{\beta} \mapsto \frac{1}{2}\,\overline{\alpha}+\overline{\psi}.
\]
\end{proposition}

\begin{remark}
\label{rmk:metaplectivealpha}
The two-cocycle defining the central extension $\mathfrak{mp}\left(H' \right)$
of $\mathfrak{sp}\left(H' \right)$ in \eqref{eq:cocyclemeta} is 
$-\frac{1}{2}\,\alpha$. 
Also, the two-cocycle defining the central extension
$\widetilde{\mathscr{U}}_{2}(H)$ of $\mathfrak{sp} \left( H' \right)\ltimes H'$ in \eqref{eq:2cocycledefUtildeleq2} is
$-\frac{1}{2}\, \alpha + \beta.$ 
\end{remark}

\begin{proof}
Let $H'_f$ be the functor which assigns to a $\mathbb{C}$-algebra $R$ the $R$-submodule of Laurent polynomials $H'_f(R):= R[t, t^{-1}]\subset H'(R)$.
One checks that
\begin{align*}
S^2\left( H'_f \right) &= \left[S^2\left( H'_f \right),  \, S^2\left( H'_f \right)\right] \quad \mbox{and}\\
 S^2\left( H'_f \right) \ltimes H'_f &= \left[S^2\left( H'_f \right) \ltimes H'_f,  \, S^2\left( H'_f \right) \ltimes H'_f\right].
\end{align*}
Since $H'_f$ is dense in $H'$ and $S^2\left( H'\right)$ is dense in $\mathfrak{sp} \left( H' \right)$, we deduce
\begin{align*}
\mathfrak{sp} \left( H' \right) &= \overline{\left[\mathfrak{sp} \left( H' \right), \, \mathfrak{sp} \left( H' \right)\right]} \quad \mbox{and}\\
\mathfrak{sp} \left( H' \right) \ltimes H' &= \overline{\left[\mathfrak{sp} \left( H' \right) \ltimes H', \, \mathfrak{sp} \left( H' \right) \ltimes H'\right]}
\end{align*}
hence \eqref{eq:H1_1} and \eqref{eq:H1_2} follow.

Next we prove \eqref{eq:H2spH}. Assume $\delta$ is an arbitrary continuous two-cocycle on $\mathfrak{sp} \left( H' \right)\ltimes H' $. 
That is, $\delta$ is a continuous linear functional on $\bigwedge^2\left( \mathfrak{sp} \left( H' \right)\ltimes H' \right)$ satisfying the two-cocycle relation 
\begin{equation}
\label{eq:2cocyclerel}
\delta(x, [y,z])+\delta(y, [z,x])+\delta(z, [x,y])=0
\end{equation}
for $x,y,z\in \mathfrak{sp} \left( H' \right)\ltimes H'$.
The Lie algebra $\mathfrak{sp} \left( H' \right)\ltimes H'$ has a natural grading given by $\deg(b_ib_j)=i+j$ and $\deg(b_k)=k$, and the action of $-\mathrm{ad}\,L_0$ on $\mathfrak{sp} \left( H' \right)\ltimes H'$ is diagonalizable with eigenvalues given by such degrees.
Applying \eqref{eq:2cocyclerel} with $z=-L_0$ and homogeneous $x,y$ yields
\[
\delta(-L_0, [x,y]) = (\deg \,x+\deg \,y)\,\delta(x,y).
\] 
Thus for $\deg \,x+\deg \,y\neq 0$, the two-cocycle $\delta$ is the two-coboundary $\delta = d\eta$, where $d\eta(x,y):=\eta([x,y])$ with $\eta(w):=\delta(-L_0, w)/\deg\,w$ for homogeneous $w$ of non-zero degree.

From now on, we replace $\delta$ with $\delta-d\eta$. It remains to consider the values $\delta(x,y)$ with $\deg \,x+\deg \,y= 0$.
Consider first $\delta(b_i, b_{-i})$.  For $|i|\neq 1$, applying \eqref{eq:2cocyclerel} with $x=b_ib_{-1}$, $y=b_1$, and $z=b_{-i}$ and using the brackets
\[
[b_ib_{-1}, b_1]=-b_i, \quad 
[b_1, b_{-i}]=0, \quad
[b_ib_{-1}, b_{-i}]=ib_{-1},
\]
one has
\[
\delta(b_i, b_{-i}) = i\,\delta(b_1, b_{-1}),
\]
and thus there exists $B\in\mathbb{C}$ such that
\[
\delta(f,h)= B \,\beta(f,h)\qquad \mbox{for $f,h\in H'_f$.}
\]

Next, consider $\delta(b_ib_j, b_k)$ with $i+j+k=0$ and necessarily non-zero $i,j,k$.
Applying \eqref{eq:2cocyclerel} with $x=b_ib_{j}$, $y=b_k$, and $z=b_ib_{-i}$
and using the brackets
\[
[b_ib_j, b_k]=0,\quad
[b_k, b_ib_{-i}]\in \mathbb{C}\langle  b_{i}\rangle,\quad
[b_ib_{-i}, b_ib_j]\in\{-i, -2i\}b_ib_j,
\]
one has $\delta(b_ib_j, b_k)=0$.
It follows that 
\[
\delta(X,f)= 0 \qquad \mbox{for $X\in S^2\left( H'_f \right)$ and $f\in H'_f$.}
\]

Next, consider $\delta(b_ib_{-i}, b_jb_{-j})$ where necessarily $|i|\neq |j|$. Applying \eqref{eq:2cocyclerel} with $x=b_ib_{j}$, $y=b_{-i}b_{-j}$, and $z=b_ib_{-i}$ and using the brackets
\begin{align*}
&[b_ib_j, b_{-i}b_{-j}]=ib_jb_{-j}+jb_ib_{-i},\quad
[b_{-i}b_{-j}, b_ib_{-i}]=-ib_{-i}b_{-j},\\
&[b_ib_{-i}, b_ib_j]=-ib_ib_j,
\end{align*}
one has $\delta(b_ib_{-i}, b_jb_{-j})=0$. Thus
\[
\delta(X,Y)= 0 \qquad \mbox{for $X,Y\in S^2\left( H'_f \right)$ with $\deg\, X = \deg\, Y=0$.}
\]

Next, consider $\delta(b_ib_{j}, b_kb_{\ell})$ with 
\[
i+j+k+\ell=0,\quad i+j\neq 0, \quad \{k,\ell\}\neq \{-i,-j\},\quad \mbox{and}\quad ijk\ell\neq 0. 
\]
Applying \eqref{eq:2cocyclerel} with $x=b_ib_{j}$, $y=b_{k}b_{\ell}$, and $z=b_mb_{-m}$
for $m\neq 0$, one has
\[
\delta(b_ib_j, [b_m b_{-m},b_k b_\ell])+\delta([b_m b_{-m},b_ib_j], b_k b_\ell)=\delta(b_m b_{-m}, [b_ib_j,b_k b_\ell])=0.
\]
The vanishing follows since $[b_ib_j, b_kb_\ell]=0$. On the other hand, one has
\[
[b_m b_{-m},b_ib_j] = \mu_m(b_ib_j)\,b_ib_j,\qquad
[b_m b_{-m},b_k b_\ell] = \mu_m(b_kb_\ell)\, b_k b_\ell 
\]
where $\mu_m(b_ib_j)$ and $\mu_m(b_kb_\ell)$ are the eigenvalues of $\mathrm{ad}(b_mb_{-m})$ applied to $b_ib_j$ and $b_kb_\ell$, respectively. Thus one has
\[
\left( \mu_m(b_ib_j) +  \mu_m(b_kb_\ell)\right)
\delta(b_ib_j, b_k b_\ell)=0 \qquad \mbox{for all $m\neq 0$.}
\]
It follows that $\delta(b_ib_j, b_k b_\ell)=0$. 

It remains to consider $\delta(b_ib_j, b_{-i}b_{-j})$ with necessarily $i+j\neq 0$ and $ij\neq 0$. We claim that 
\begin{equation}
\label{eq:deltaisijs}
\delta(b_ib_j, b_{-i}b_{-j}) = \left\{
\begin{array}{ll}
ij(s_i+s_j) & \mbox{for $|i|\neq |j|$,}\\[5pt]
4i^2s_i & \mbox{for $i=j$}
\end{array}
\right.
\end{equation}
for some \textit{odd} function $s\colon \mathbb{Z}\setminus\{0\}\rightarrow\mathbb{C}$, $a\mapsto s_a$.
To show the claim, let
\[
q(i,j):= \frac{1}{ij}\, \delta(b_ib_j, b_{-i}b_{-j}),
\]
symmetric in $i$ and $j$. 
Consider first $|i|\neq |j|$, and
select $k$ such that $|i|, |j|, |k|$ are distinct.
Applying \eqref{eq:2cocyclerel} with $x=b_ib_{k}$, $y=b_{j}b_{-k}$, and $z=b_{-i}b_{-j}$ and using the brackets
\[
[b_ib_k, b_jb_{-k}] = kb_ib_j, \quad
[b_jb_{-k}, b_{-i}b_{-j}]=jb_{-i}b_{-k},\quad
[b_{-i}b_{-j}, b_ib_k]=-ib_{-j}b_k, 
\]
one has
\begin{equation}
\label{eq:qijadditive}
q(i,j)= q(i,k)+ q(j,-k).
\end{equation}
Exchanging $i$ and $j$ in \eqref{eq:qijadditive} and using that $q$ is symmetric, it follows that 
\begin{equation}
\label{eq:qakqaminuskinda}
q(a,k)-q(a,-k) \quad\mbox{is independent of $a$.}
\end{equation}
Define
\[
s_a:=  \frac{1}{2}\left( q(a,1)+q(a,-1)\right) \qquad \mbox{for $|a|\geq 2$,}
\]
and 
\[
s_1 := \frac{1}{2}\left( q(a,1)-q(a,-1)\right) \quad\mbox{for some $|a|\geq 2$,} \qquad s_{-1}:=-s_1.
\]
The definition of $s_1$ is independent of $a$ by \eqref{eq:qakqaminuskinda}.
Then for $|i|, |j|\geq 2$, applying \eqref{eq:qijadditive} with $k=1$ yields
\begin{eqnarray*}
q(i,j) &=& q(i,1)+ q(j,-1)\\
&=&q(i,1) -\frac{1}{2}\left(q(i,1)-q(i,-1)\right) +\frac{1}{2}\left(q(j,1)-q(j,-1)\right) + q(j,-1)\\
&=& s_i+s_j.
\end{eqnarray*}
In the second identity we used \eqref{eq:qakqaminuskinda}. Similarly, for $j=\pm 1$, one has
\[
q(i,\pm 1) = s_i + s_{\pm 1}
\]
by a direct computation. Thus $q(i,j) = s_i+s_j$ for $|i|\neq |j|$, and substituting into \eqref{eq:qijadditive} yields $s_k=-s_{-k}$, hence $s$ is odd.
This proves \eqref{eq:deltaisijs} for $|i|\neq |j|$.

When $i=j$, select $k$ such that $|i|\neq |k|$. Applying \eqref{eq:2cocyclerel} with $x=b_ib_{i}$, $y=b_{-i}b_{k}$, and $z=b_{-i}b_{-k}$ and using the brackets
\[
[b_ib_i, b_{-i}b_{k}] = 2ib_ib_k, \quad
[b_{-i}b_{k}, b_{-i}b_{-k}]=kb_{-i}b_{-i},\quad
[b_{-i}b_{-k}, b_ib_i]=-2ib_ib_{-k}, 
\]
one has
\[
q(i,i) = 2q(i,k)-2q(-i,k).
\]
From the previous  case applied to the right-hand side, one has
\[
q(i,i) = 2s_i+2s_k -2s_{-i}-2s_k = 4s_i.
\]
This concludes the proof of \eqref{eq:deltaisijs}.

Next, we claim that 
\begin{equation}
\label{eq:sievconst}
s_a=s_{a+1} \qquad \mbox{for $a\gg 0$.}
\end{equation}
Indeed, consider $L_1=\sum_{k> 0}b_{-k}b_{k+1}$. For $a\geq 1$, one has
\begin{equation}
\label{eq:deltaLibib}
\delta(L_1, b_a b_{-a-1})= \delta(b_{-a}b_{a+1}, b_a b_{-a-1}) =-a(a+1)(s_{a+1}-s_a)
\end{equation}
where the first identity follows by continuity and the second one by \eqref{eq:deltaisijs}. On the other hand, $\delta(L_1,)$ is a continuous linear functional on $\mathfrak{sp} \left( H' \right)$.
Linear functionals on a topological vector space where neighborhoods of $0$ are linear subspaces are continuous if and only if they vanish identically on some open neighborhood of $0$.
Thus $\delta(L_1,)$ vanishes identically on some open neighborhood \eqref{eq:openneigh0} of the zero element in $\mathfrak{sp} \left( H' \right)$. 
As any open neighborhood of $0$ contains $b_a b_{-a-1}$ for $a\gg 0$, one has $\delta(L_1,b_a b_{-a-1})=0$ for $a\gg 0$, and thus combining this with \eqref{eq:deltaLibib} implies the claim \eqref{eq:sievconst}.

A direct computation shows that the two-cocycle $\alpha$ also satisfies \eqref{eq:deltaisijs} and \eqref{eq:sievconst} with $s_a^\alpha:=-\mathrm{sign}\,(a)$ in place of $s_a$.
Thus $\delta$ satisfying \eqref{eq:deltaisijs} and \eqref{eq:sievconst} implies that 
\[
s_a=:s_a ^\delta= A\,s_a^\alpha + t_a \qquad \mbox{with $A:=-s_a^\delta$ for $a\gg 0$,}
\]
for some function $a\mapsto t_a$ odd and finitely supported. This implies that
\[
\delta(X,Y) = A\,\alpha (X,Y) + \zeta(X,Y) \qquad \mbox{for $X,Y\in \mathfrak{sp} \left( H' \right)$}
\]
for a two-cocycle $\zeta$ such that 
\[
\zeta(b_ib_j, b_{-i}b_{-j}) := \left\{
\begin{array}{ll}
ij(t_i+t_j) & \mbox{for $|i|\neq |j|$,}\\[5pt]
4i^2t_i & \mbox{for $i=j$,}
\end{array}
\right.
\] 
and $\zeta=0$ otherwise.

Finally, we show that $\zeta$ is a two-coboundary. The idea is to use that $t_a$ is finitely supported. 
Consider the linear functional
\[
\eta\colon \mathfrak{sp} \left( H' \right) \rightarrow \mathbb{C}, \qquad X\mapsto -\sum_{a>0}t_a\,X_{aa},
\]
where $X_{aa}$ is the coefficient of $b_a$ in $X(b_a)$. The sum has only finitely many nonzero summands since $t_a$ is finitely supported, and thus $\eta$ is continuous. 
Since for $|i|\neq |j|$ one has
\[
[b_ib_j, b_{-i}b_{-j}]=ib_jb_{-j}+jb_ib_{-i},\qquad
[b_ib_i, b_{-i}b_{-i}] = 4ib_ib_{-i},
\]
and $\eta(b_ab_{-a})=|a|\,t_{|a|}$ for all $a$,
it follows that $d\eta= \zeta$. 

In conclusion, one has
\[
\delta - d\eta= A\,\alpha + B \,\beta \qquad\mbox{with $A,B\in\mathbb{C}$ on $S^2\left( H'_f \right) \ltimes H'_f$,}
\]
and thus by continuity on $\mathfrak{sp} \left( H' \right) \ltimes H'$.
Restricting via $\sigma^*$ shows that the classes of $\alpha$ and $\beta$ are independent.
Hence \eqref{eq:H2spH} follows.

The identification in \eqref{eq:H2sp} follows similarly from the argument for \eqref{eq:H2spH} after ignoring the factor $H'$.
 
For the final part of the statement, recall the map $\widehat{\sigma}\colon \mathfrak{D} \hookrightarrow \widetilde{\mathscr{U}}_{2}(H)$ from \eqref{eq:sigmahat}. 
Since the two-cocycle on $\mathrm{Witt}\ltimes H'$ defining $\mathfrak{D}$ is $\psi$, and the two-cocycle on $\mathfrak{sp} \left( H' \right)\ltimes H'$ defining $\widetilde{\mathscr{U}}_{2}(H)$ is $-\frac{1}{2}\, \alpha + \beta$, we deduce $\sigma^*\left( -\frac{1}{2}\, \alpha + \beta \right)=\psi$.
A similar argument involving the map $\widehat{\tau}$ from \eqref{eq:tauhat} shows that $\sigma^*(\alpha)=\tau^*(\alpha)=\alpha$, hence the statement.
\end{proof}

Recall the Lie algebra $\mathfrak{sp}_F\left(H' \right)$ for $(Z, F, L)\in\widehat{\mathcal{A}}_g$ from \eqref{eq:spF} and the algebra $\mathfrak{i}(x)$ for $x\in\widehat{\mathcal{X}}_g$ from \eqref{eq:ix}.
For a topological Lie algebra $\mathfrak{g}$, the space $H^2(\mathfrak{g}, \mathbb{C})$ classifies Lie algebra continuous central extensions of $\mathfrak{g}$ up to isomorphism. 
Thus as a consequence of the previous statement, we have:

\begin{proposition}
\label{prop:splittings}
\begin{enumerate}[(i)]
\item Every Lie algebra continuous central extension
\[
0 \rightarrow \mathfrak{gl}_1  \rightarrow  \widehat{\mathfrak{sp}\left(H' \right)} \rightarrow \mathfrak{sp}\left(H' \right) \rightarrow 0
\]
---e.g., $\mathfrak{mp}\left(H' \right)$---splits over $\mathfrak{sp}_F\left(H' \right)$ for all $(Z, F, L)\in\widehat{\mathcal{A}}_g$. 

\item Similarly, every Lie algebra continuous central extension
\[
0 \rightarrow \mathfrak{gl}_1  \rightarrow \widehat{\mathfrak{sp}\left(H' \right) \ltimes H'} \rightarrow \mathfrak{sp}\left(H' \right) \ltimes H'  \rightarrow 0
\]
---e.g., $\widetilde{\mathscr{U}}_{2}(H)$---splits over $\mathfrak{i}(x)$ for all $x\in\widehat{\mathcal{X}}_g$.

\item The splittings in (i) and (ii) are unique.

\end{enumerate}
\end{proposition}

\begin{proof}
For (i), one needs to show that the restriction of $H^2\left( \mathfrak{sp}\left(H' \right), \mathbb{C}\right)$ to $\mathfrak{sp}_F\left(H' \right)$ vanishes for all $(Z, F, L)\in\widehat{\mathcal{A}}_g$.
From Proposition \ref{prop:cohLiespH}, it is enough to consider the class of the two-cocycle $\alpha$. 
Also, after rescaling to $-\frac{1}{2}\alpha$ as in Remark \ref{rmk:metaplectivealpha}, we can reduce to the case $\widehat{\mathfrak{sp}\left(H' \right)}=\mathfrak{mp}\left(H' \right)$.

For $Z=H_-$ and $F\subset Z$, one has $F \subset H_-\subset F^\perp$, thus for $X\in \mathfrak{sp}_F\left(H'\right)$, one has
\[
\mathrm{Im}\left( \pi_+ \,X\, \pi_-\right) =  \pi_+ \,X(H_-) \subset \pi_+ \,X(F^\perp)\subseteq \pi_+ \, F =0.
\]
It follows that for  $F\subset H_-$, one has $\alpha(X,Y)=0$ for all $X,Y$ in $\mathfrak{sp}_F\left(H' \right)$. 
Next, consider an arbitrary $(Z,F,L) \in \widehat{\mathcal{A}}_g$. From \S\ref{sec:Sp+}, $\mathbb{S}\mathrm{p}^+\left(H' \right)$ acts transitively on the pairs $(Z,F)$, hence 
\[
(Z,F) = \rho(H_-,F_0)\quad \mbox{for some $\rho\in \mathbb{S}\mathrm{p}^+\left(H' \right)$ and $F_0 \subset H_-$.}
\]
Then $\mathrm{Ad}_\rho\colon X \mapsto \rho X\rho^{-1}$ is a continuous automorphism of ${\mathfrak{sp}\left(H' \right)}$ with $\mathrm{Ad}_\rho\left({\mathfrak{sp}_{F_0}\left(H' \right)}\right) = {\mathfrak{sp}_F\left(H' \right)}$. The pull-back $\mathrm{Ad}_\rho^*(\alpha)$ is a continuous two-cocycle on ${\mathfrak{sp}\left(H' \right)}$, hence by \eqref{eq:H2sp} one has $\mathrm{Ad}_\rho^*(\alpha) = c\alpha + d\eta$ for some $c \in \mathbb{C}$ and some continuous linear map $\eta\colon \mathfrak{sp}\left(H' \right) \rightarrow \mathbb{C}$. Restricting to $\mathfrak{sp}_{F_0}\left(H' \right)$, where $\alpha$ vanishes by the previous paragraph, one obtains $\alpha(\mathrm{Ad}_\rho X, \mathrm{Ad}_\rho Y) = d\eta(X,Y)$ for $X,Y \in \mathfrak{sp}_{F_0}\left(H' \right)$, that is, the restriction of $\alpha$ to $\mathfrak{sp}_F\left(H' \right)$ is the coboundary of $\eta \circ \mathrm{Ad}_\rho^{-1}$. Hence $[\alpha]$ restricts to zero on $\mathfrak{sp}_F\left(H' \right)$.

For (ii), start by observing that the isomorphism 
$ \mathfrak{sp}_{F} \left( H' \right) \ltimes F \cong\mathfrak{i}(x)$ from \eqref{eq:ixiso} extends to a continuous inner automorphism  of $\mathfrak{sp} \left( H' \right) \ltimes H'$
\[
1+\mathrm{ad}_h\colon (X,a)\mapsto X-X(h)+a
\]
 with $1+\mathrm{ad}_h = \exp(\mathrm{ad}_h)$ since $\mathrm{ad}_h^2=0$. Since continuous inner automorphisms act trivially on continuous cohomology, 
we can prove (ii) after replacing $\mathfrak{i}(x)$ with $\mathfrak{sp}_{F} \left( H' \right) \ltimes F$.
For this, the case of $\alpha$ on $\mathfrak{sp}\left(H' \right) \ltimes H' $ follows from $\alpha$ on $\mathfrak{sp}\left(H' \right)$ by pull-back, and the case of $\beta$ follows immediately from $F$ being isotropic with respect to~\eqref{eq:symplectic}.

For (iii), any two continuous splittings differ by a continuous character of $\mathfrak{g}:=\mathfrak{sp}_F\left(H' \right)$ for $(Z, F, L)\in\widehat{\mathcal{A}}_g$ (respectively, $\mathfrak{g}:=\mathfrak{i}(x)$ for $x\in\widehat{\mathcal{X}}_g$) vanishing on $[\mathfrak{g}, \mathfrak{g}]$. On the other hand, $[\mathfrak{g}, \mathfrak{g}]$ is dense in $\mathfrak{g}$ by a direct computation. The statement follows.
\end{proof}

We conclude this section with:

\begin{proposition}
\label{prop:splitting2}
\begin{enumerate}[(i)]
\item For $(Z, F, L)\in\widehat{\mathcal{A}}_g$, the unique splitting of $\mathfrak{sp}_F\left(H' \right)\rightarrow \mathfrak{mp}\left(H' \right)$ is
\[
f\ell \mapsto f\ell, \qquad \mbox{for $f\in F$ and $\ell\in H'$,}
\]
where the target is the product with the $F$-factor on the left.
This map is compatible with $\mathfrak{sp}^+\left(H' \right)\hookrightarrow \mathfrak{mp}\left(H' \right)$ from \eqref{eq:sp+inmp}.

\item For $x=(Z, F, L, \overline{h}, q)\in\widehat{\mathcal{X}}_g$, the unique splitting of $\mathfrak{i}(x)\rightarrow \widetilde{\mathscr{U}}_{2}(H)$ is
\begin{align*}
f\ell - (f\ell)(h) &\mapsto \left(f-\langle f,h\rangle\bm{1} \right)\left(\ell-\langle \ell,h\rangle\bm{1} \right),&
f &\mapsto f-\langle f,h\rangle\bm{1},
\end{align*}
for $f\in F$, $\ell\in H'$, and some $h\in \overline{h}+K$. 
Equivalently,
\begin{align*}
X - X(h) &\mapsto X - X(h) + \frac{1}{2}\langle X(h),h\rangle\bm{1},&
f &\mapsto f-\langle f,h\rangle\bm{1},
\end{align*}
for $X\in \mathfrak{sp}_F\left(H' \right)$ and $f\in F$.
This is independent of the choice of $h\in \overline{h}+K$ and is
compatible with $\mathfrak{sp}^+\left(H' \right)\ltimes H'_+\hookrightarrow \widetilde{\mathscr{U}}_{2}(H)$
from \eqref{eq:sp+H+inUleq2}.
\end{enumerate}
\end{proposition}

\begin{proof}
A direct computation on the topological basis shows that the two given maps are Lie algebra maps. Since the splittings are unique by Proposition \ref{prop:splittings}, the statement follows.
\end{proof}

%%%%%%%%%%%%%%%%%%%%%%%%%%%%
%%%%%%%%%%%%%%%%%%%%%%%%%%%%
%%%%%%%%%%%%%%%%%%%%%%%%%%%%

\section{The isomorphism of second cohomology spaces}
\label{sec:isoH2}

Consider a complex manifold $X$, possibly infinite-dimensional, carrying a transitive action of a Lie algebra $\mathfrak{g}$ (as defined in \S\ref{sec:spunif}).
For $x\in X$, let $\mathfrak{g}_x:= \mathrm{Ker}(\mathfrak{g}\rightarrow T_x(X))$.
It is shown in \cite[4.1]{adkp} that if $\mathfrak{g}_x=\overline{[\mathfrak{g}_x, \mathfrak{g}_x]}$ for all $x\in X$, then every Lie algebra continuous central extension
\[
0 \rightarrow \mathbb{C}  \rightarrow  \widehat{\mathfrak{g}} \rightarrow \mathfrak{g} \rightarrow 0
\]
which splits over $\mathfrak{g}_x$ for all $x\in X$ yields a continuous extension
\[
0 \rightarrow X\times \mathbb{C}  \rightarrow  E \rightarrow T(X) \rightarrow 0.
\]
This defines  a canonical map
\begin{equation}
\label{eq:homH2}
H^2( \mathfrak{g}, \mathbb{C})_0\rightarrow \mathrm{Ext}^1\left(\mathscr{T}_X, \mathscr{O}_X \right) = H^1\left( \Omega^1_X \right)
\end{equation}
where $H^2( \mathfrak{g}, \mathbb{C})_0\subset H^2( \mathfrak{g}, \mathbb{C})$ is the subspace obtained by intersecting the kernels of the restriction maps $H^2( \mathfrak{g}, \mathbb{C})\rightarrow H^2( \mathfrak{g}_x, \mathbb{C})$ for all $x\in X$.
Moreover, one has a canonical homomorphism
\[
c\colon H^1\left( \mathscr{O}^*_{X} \right) \rightarrow H^1\left( \Omega^1_X \right)
\]
mapping the class of each line bundle $L$ to the extension class given by the sheaf of differential operators of order less than or equal to $1$ acting on $L$.

Combining Theorems \ref{thm:unif1} and \ref{thm:unif2} with Proposition \ref{prop:splittings}, the above result from \cite{adkp} applies to give  canonical maps
\begin{align*}
\nu\colon H^2\left(\mathfrak{sp} \left( H' \right), \mathbb{C}\right) &\rightarrow H^1\left(\Omega^1_{\widehat{\mathcal{A}}_g}\right),\\
\mu\colon H^2\left(\mathfrak{sp} \left( H' \right)\ltimes H' , \mathbb{C} \right) &\rightarrow H^1\left(\Omega^1_{\widehat{\mathcal{X}}_g}\right).
\end{align*}
We determine these maps explicitly in terms of the basis elements of the $H^2$ spaces given in Proposition \ref{prop:cohLiespH} and the line bundle classes from \S\ref{sex:lambdatheta}. Here $\lambda=c[\Lambda]$ and $\xi=c[\Xi]$.

\begin{theorem}
\label{thm:canisoH2spAX}
One has
\begin{align}
\label{eq:lanu}
&\lambda = - \nu \left[ \alpha\right], \\
\label{eq:lathetamu}
&\lambda = - \mu \left[ \alpha\right] , \qquad 
\xi = -\mu\left[\alpha\right]+2\mu\left[\beta\right].
\end{align}
In particular, the image of $\nu$ is spanned by $\lambda$, and the image of $\mu$
is spanned by $\lambda$ and $\xi$.
\end{theorem}

To prove this statement, first we show the following. Recall the Sato Grassmannian $\mathrm{Gr}(H)$ and its determinant line bundle $\mathscr{L}$ from \S\ref{sec:Boson}. Define $\mathrm{Gr}(H')$ similarly, with inclusion $\mathrm{Gr}(H')\hookrightarrow \mathrm{Gr}(H)$ given by $W\mapsto W\oplus \mathbb{C}$.

\begin{lemma}
\label{lemma:phi}
The map
\[
\phi\colon \widehat{\mathcal{A}}_g\to \mathrm{Gr}(H'),\qquad (Z,F,L)\mapsto F^\perp
\]
is $\mathfrak{sp}(H')$-equivariant, and $\phi^*\mathscr{L}\cong\Lambda$.
\end{lemma}

\begin{proof}
Since $Z\oplus H'_+=H'$ and $Z\subset F^\perp$, one has $F^\perp=Z\oplus F^\perp_+$ with
$F^\perp_+:=F^\perp\cap H'_+$. Hence the projection $F^\perp\to H_-$ is surjective with
kernel $F^\perp_+$ of dimension $g$, so $F^\perp\in\mathrm{Gr}(H')$ and
$\mathscr{L}_{F^\perp}=\det F^\perp_+$. As $Z/F$ and $F^\perp_+$ are complementary
Lagrangian subspaces of $F^\perp/F$, the form $\langle\,,\rangle$ identifies $F^\perp_+$
with $(Z/F)^\vee=H^0(X,\Omega^1_X)$, where $X=(Z/F)/L$. Thus
\[
\phi^*\mathscr{L}=\det F^\perp_+\cong\Lambda, 
\]
and similarly for families of extended PPAVs. 

Finally, the
action of $\mathfrak{sp}(H')$ on $\widehat{\mathcal{A}}_g$ from Theorem \ref{thm:unif1} is induced
by the action of $\mathfrak{sp}(H')\subset\mathfrak{gl}(H')$ on subspaces of $H'$. Also,
for $X\in\mathfrak{sp}_F(H')$, one has
$X(F^\perp)\subseteq F\subseteq F^\perp$. Hence $\phi$ is
equivariant and maps the stabilizer $\mathfrak{sp}_F(H')$ of $(Z,F,L)$ in $\mathfrak{sp}(H')$ into the stabilizer of $F^\perp$ in
$\mathfrak{gl}(H')$.
\end{proof}

Thus we return to Theorem \ref{thm:canisoH2spAX}:

\begin{proof}[Proof of Theorem \ref{thm:canisoH2spAX}]
We argue as in \cite[4.10(i)]{adkp}. By \cite[Lemma 4.7]{adkp}, 
the two-cocycle $\psi$ on $\mathfrak{gl}(H')$ from \eqref{eq:psi} maps via \eqref{eq:homH2} to 
the Atiyah algebra $\mathscr{F}_{\mathscr{L}^{-1}}$ on $\mathrm{Gr}(H')$. 
Indeed, one has an action on the Boson space $B$ of the central extension $\widehat{\mathfrak{gl}}(H)$ of $\mathfrak{gl}(H)$ defined by $\psi$:
\[
\widehat{\mathfrak{gl}}(H) \rightarrow \mathfrak{gl}(B),
\]
for which $\Pi$ is equivariant, with $\bm{1}$ acting as $-\mathrm{id}_B$ \cite[(4.16)]{adc91}.
(This is compatible with \eqref{eq:U2HactingonB}: the map
$H\to\widehat{\mathfrak{gl}}(H)$ given by the action of $H$ on itself by multiplication
is a map of Lie algebras, and the induced action of $H$ on $B$ is the one in \eqref{eq:U2HactingonB}.)
This induces an action of $\widehat{\mathfrak{gl}}(H)$ on the line bundle $\Pi^*\mathscr{O}_{\mathbb{P}B}(-1)$ on $\mathrm{Gr}_0(H)$ by first-order differential operators with $\bm{1}$ acting as $-\mathrm{id}$, by pulling back the action on $\mathscr{O}_{\mathbb{P}B}(-1) \subset B \otimes_\mathbb{C} \mathscr{O}_{\mathbb{P}B}$ via $\Pi$. Since $\Pi^*\mathscr{O}_{\mathbb{P}B}(-1) = \mathscr{L}$, the central element acts as the identity on $\mathscr{L}^{-1}$.
The same holds on the other connected components of $\mathrm{Gr}(H)$,
which are identified with $\mathrm{Gr}_0(H)$ by multiplication by powers of $t$. 

By Lemma~\ref{lemma:phi}, one has a commutative diagram
\[
\begin{tikzcd}
\mathfrak{sp} \left( H' \left( \mathscr{O}_{\widehat{\mathcal{A}}_g}\right)\right)  \arrow{r}\arrow{d} 
&\mathfrak{gl} \left( H' \left( \mathscr{O}_{\mathrm{Gr}(H')}\right)\right)\arrow{d}\\
\mathscr{T}_{\widehat{\mathcal{A}}_g} \arrow{r}{ d\phi } & \mathscr{T}_{\mathrm{Gr}(H')}
\end{tikzcd}
\]
and the canonical splitting over the stabilizer of $F^\perp$ in $\mathfrak{gl}(H')$
restricts to a continuous splitting over $\mathfrak{sp}_F(H')$, which is the one of
Proposition~5.4 by Proposition~5.3(iii). Since $\alpha$ is the restriction of $\psi$,
by functoriality, the extension of $\mathscr{T}_{\widehat{\mathcal{A}}_g}$ assigned to $\alpha$ is
 \[
 \phi^*\mathscr{F}_{\mathscr{L}^{-1}}=\mathscr{F}_{\phi^*\mathscr{L}^{-1}}=\mathscr{F}_{\Lambda^{-1}}.
 \]
  Hence $\nu[\alpha]=-\lambda$. This proves \eqref{eq:lanu}.

Recall the embedding $\Phi\colon \widehat{\mathcal{X}}_g\rightarrow \mathbb{P}B$ from \S\ref{sec:Boson}.
The action of $\widetilde{\mathscr{U}}_{2}(H)$ on the Boson space $B$ in \eqref{eq:U2HactingonB} with $\bm{1}$ acting as $-\mathrm{id}_B$ 
extends the action of $\mathfrak{sp} \left( H' \right)\ltimes H'$ on  $\widehat{\mathcal{X}}_g$.
This induces an action of $\widetilde{\mathscr{U}}_{2}(H)$ on $\Phi^*\mathscr{O}_{\mathbb{P}B}(-1)$ by first-order differential operators with $\bm{1}$ acting as $-\mathrm{id}$, by pulling back the action on $\mathscr{O}_{\mathbb{P}B}(-1) \subset B \otimes_\mathbb{C} \mathscr{O}_{\mathbb{P}B}$ via $\Phi$.

On the other hand, one has $\Xi=\Phi^*\mathscr{O}_{\mathbb{P}B}(2)$. Indeed, the vacuum
$v\in V$ from \S\ref{sec:Boson} is a linear form on $B=V^\vee$, hence a section $s_v$ of
$\mathscr{O}_{\mathbb{P}B}(1)$. By \cite[(3.16)]{adc91}, for $x\in\widehat{\mathcal X}_g$ the
value $\Phi(x)(v)$ is the product of a nowhere vanishing function of $x$ and the theta
function evaluated at the image of $x$ in $\mathcal X_g$. Hence the divisor of
$\Phi^*s_v$ is the pull-back of $\theta$ via
$\widehat{\mathcal X}_g\to\mathcal X_g$, and
$\Phi^*\mathscr{O}_{\mathbb{P}B}(1)=\mathscr{O}(\theta)$. Restricting to
$\widehat{\mathcal P}_{g-1}$, by the commutativity of \eqref{eq:Shiota}, this is the identity
$K^*\Pi^*\mathscr{O}_{\mathbb{P}B}(1)=K^*\mathscr L^{-1}=\mathscr{O}(\theta)$
\cite[(3.16)]{adkp}.

Hence, $\widetilde{\mathscr{U}}_{2}(H)$ acts on $\Xi$ with $\bm{1}$ acting as multiplication by $2$. 
From Remark \ref{rmk:metaplectivealpha}, $\widetilde{\mathscr{U}}_{2}(H)$ is defined by the two-cocycle  
$-\frac{1}{2}\, \alpha + \beta$ on $\mathfrak{sp} \left( H' \right)\ltimes H'$. Taking classes we deduce $\mu[-\tfrac12\alpha+\beta]=\tfrac12\xi$.
The identity $\lambda=-\mu[\alpha]$ follows from \eqref{eq:lanu} by pull-back along
$\widehat{\mathcal X}_g\to\widehat{\mathcal A}_g$,
thus \eqref{eq:lathetamu}.
\end{proof}

Next, we prove Theorem \ref{thm:mpaction}, here restated:

\begin{theorem}
\begin{enumerate}[(i)]
\item The line bundle $\Lambda$ on $\widehat{\mathcal{A}}_g$ carries an action of $\mathfrak{mp} \left( H' \right)$ by first-order differential operators  extending the transitive action of \sloppy\mbox{$\mathfrak{sp} \left( H' \right)$} on  $\widehat{\mathcal{A}}_g$ and with the central element $\bm{1}\in \mathfrak{mp} \left( H' \right)$ acting as multiplication by $2$ on the fibers of $\Lambda\rightarrow \widehat{\mathcal{A}}_g$.

\item The  line bundle $\Xi$ on $\widehat{\mathcal{X}}_g$ carries an action of $\widetilde{\mathscr{U}}_{2}(H)$ by first-order differential operators  extending the transitive action of \sloppy\mbox{$\mathfrak{sp} \left( H' \right)\ltimes H'$} on  $\widehat{\mathcal{X}}_g$ and with the central element $\bm{1}\in \widetilde{\mathscr{U}}_{2}(H)$ acting as multiplication by $2$ on the fibers of $\Xi\rightarrow \widehat{\mathcal{X}}_g$.
\end{enumerate}
\end{theorem}

\begin{proof}
From \eqref{eq:lanu}, the central extension of $\mathfrak{sp}(H')$ defined by
$\alpha$ acts on $\Lambda^{-1}$ by first-order differential operators, with the central
element acting as the identity.
As $\mathfrak{mp}(H')$ is defined by $-\tfrac12\alpha$ (Remark \ref{rmk:metaplectivealpha}),
its central element $\mathbf{1}$ corresponds to $-2$ times that element; thus
$\mathfrak{mp}(H')$ acts on $\Lambda^{-1}$ with $\mathbf{1}$ acting by $-2$, equivalently
on $\Lambda$ with $\mathbf{1}$ acting by $2$, thus part (i).
Part (ii) follows similarly and was proven in the proof of \eqref{eq:lathetamu}.
\end{proof}

Finally, we prove Theorem \ref{thm:canisointro} from the intro, here restated:

\begin{theorem}
\label{thm:caniso}
For $g\geq 2$, the maps $\nu$ and $\mu$ induce canonical isomorphisms 
\begin{align*}
\nu \colon H^2\left(\mathfrak{sp} \left( H' \right), \mathbb{C}\right) &\rightarrow H^2\left( {\mathcal{A}}_g, \mathbb{C}\right), \\ 
\mu \colon H^2\left(\mathfrak{sp} \left( H' \right)\ltimes H', \mathbb{C} \right) &\rightarrow H^2\left( {\mathcal{X}}_g, \mathbb{C}\right).
\end{align*}
\end{theorem}

\begin{proof}
The argument is similar to \cite[pg.~30]{adkp}.
Consider first the map~$\nu$. Since $\widehat{\mathcal{A}}_g$ is not complete, we cannot apply Hodge theory to obtain a natural inclusion 
\[
H^1\left(\Omega^1_{\widehat{\mathcal{A}}_g}\right)\rightarrow H^2\left( \widehat{\mathcal{A}}_g, \mathbb{C}\right),
\]
 and it is unclear whether  such an inclusion exists.
Instead, one proceeds as follows.
 Let 
\[
c_1\colon H^1\left(\mathscr{O}^*_{\widehat{\mathcal{A}}_g}\right) \rightarrow H^2\left( \widehat{\mathcal{A}}_g, \mathbb{C}\right)
\]
be the first Chern class map. For the class of a line bundle $L$, its image via $c_1$ is the de Rham class of $c[L]$.
From Theorem \ref{thm:canisoH2spAX}, 
the extension assigned by $\nu$ to each class is a multiple
of the Atiyah algebra of $\Lambda$. Assigning to it the corresponding multiple of
$c_1(\Lambda)$ yields a canonical map, still denoted $\nu$, 
$H^2(\mathfrak{sp}(H'),\mathbb C)\to H^2(\widehat{\mathcal A}_g,\mathbb C)$,
which is an isomorphism since $c_1(\Lambda)$ generates the target by
Proposition~\ref{prop:H2AXMP}.
Since one has $H^2( \widehat{\mathcal{A}}_g, \mathbb{C}) \cong H^2\left( {\mathcal{A}}_g, \mathbb{C}\right)$ from Proposition \ref{prop:H2AXMP}, the statement follows.
The statement about $\mu$ follows similarly from Theorem \ref{thm:canisoH2spAX} and Proposition \ref{prop:H2AXMP}.
\end{proof}

The spaces and line bundles appearing in Theorem \ref{thm:mpaction} are summarized by the commutative diagram:
\begin{equation}
\label{eq:bigsquare}
\begin{tikzcd}
&\Lambda \arrow{dd}  && \Xi \arrow{dd}  \\
\Lambda \arrow{dd} \arrow{ru} && \Xi  \arrow{ru}   \\
&\widehat{\mathcal{A}}_g && \widehat{\mathcal{X}}_{g}  \arrow{ll} \\
\widehat{\mathcal{M}}_g \arrow[hookrightarrow]{ru} && \widehat{\mathcal{P}}_{g-1} \arrow{ll}  \arrow[hookrightarrow]{ru}  \arrow[crossing over, leftarrow]{uu} 
\end{tikzcd}
\end{equation}
where the left and right squares are Cartesian. This diagram provides a geometric counterpart to Figure \ref{fig:bigLiesquare}.
The conclusions of Theorem \ref{thm:mpaction} are summarized by  the following diagrams.
Let $({Z}, {F}, {L})\rightarrow S$ be a family  of extended abelian varieties of dimension $g$ over a smooth base $S$ with $S\rightarrow \widehat{\mathcal{A}}_g$ \'etale.
Recall that Proposition \ref{prop:splitting2} gives the splitting
$\mathfrak{sp}_{F} \left(H'(\mathscr{O}_S) \right) \hookrightarrow \mathfrak{mp}\left(H'(\mathscr{O}_S) \right)$. 
Let $\mathscr{F}_{\Lambda}$ be the Atiyah algebra of the line bundle $\Lambda$ on $\widehat{\mathcal{A}}_g$. 
Theorem \ref{thm:mpaction} implies the following commutative diagram of sheaves of Lie algebras with exact rows and columns:
\begin{equation}
\label{eq:AtiyahLambdadiagram}
\begin{tikzcd}
&\mathscr{O}_S \arrow[hookrightarrow]{d} \arrow[rightarrow]{r}{2} & \mathscr{O}_S \arrow[hookrightarrow]{d}\\
\mathfrak{sp}_{F} \left( H' \left( \mathscr{O}_S\right)\right) \arrow[rightarrow]{d}[above, rotate=-90]{=} \arrow[hookrightarrow]{r} & \mathfrak{mp} \left( H' \left( \mathscr{O}_S\right)\right) \arrow[->>]{d} \arrow[->>]{r} & \mathscr{F}_{\Lambda|S} \arrow[->>]{d}\\
\mathfrak{sp}_{F} \left( H' \left( \mathscr{O}_S\right)\right) \arrow[hookrightarrow]{r} & \mathfrak{sp} \left( H' \left( \mathscr{O}_S\right)\right) \arrow[->>]{r} & \mathscr{T}_S.
\end{tikzcd}
\end{equation}

Similarly, consider a family $({Z}, {F}, {L}, \overline{h}, q)\rightarrow S$ over a smooth base $S$ with $g$ fixed as in \S\ref{sec:extunivPPAV} and with $S\rightarrow \widehat{\mathcal{X}}_g$ \'etale. Let $\mathfrak{i}$ be the sheaf of Lie algebras globalizing $\mathfrak{i}(x)$ from \eqref{eq:ix}.
Proposition \ref{prop:splitting2} gives the splitting $\mathfrak{i} \hookrightarrow \widetilde{\mathscr{U}}_{2}(H(\mathscr{O}_S))$.
Let  $\mathscr{F}_{\Xi}$ be the Atiyah algebra of the line bundle $\Xi$ on $\widehat{\mathcal{X}}_g$.
Theorem \ref{thm:mpaction} implies the following commutative diagram of sheaves of Lie algebras with exact rows and columns:
\begin{equation}
\label{eq:AtiyahThetadiagram}
\begin{tikzcd}
&\mathscr{O}_S \arrow[hookrightarrow]{d} \arrow[rightarrow]{r}{2} & \mathscr{O}_S \arrow[hookrightarrow]{d}\\
\mathfrak{i} \arrow[rightarrow]{d}[above, rotate=-90]{=} \arrow[hookrightarrow]{r} & \widetilde{\mathscr{U}}_{2}\left(H \left( \mathscr{O}_S\right)\right) \arrow[->>]{d} \arrow[->>]{r} & \mathscr{F}_{\Xi|S} \arrow[->>]{d}\\
\mathfrak{i} \arrow[hookrightarrow]{r} & \mathfrak{sp} \left( H' \left( \mathscr{O}_S\right)\right) \ltimes H'\left( \mathscr{O}_S\right) \arrow[->>]{r} & \mathscr{T}_S.
\end{tikzcd}
\end{equation}

In \eqref{eq:AtiyahLambdadiagram} and \eqref{eq:AtiyahThetadiagram}, we use the following statement:

\begin{lemma}
\label{lemma:idealalgebroid}
For all $g$:
\begin{enumerate}[(i)]
\item $\mathfrak{sp}_F \left( H' (\mathscr{O}_S)\right)$ is an ideal for the algebroid bracket in $\mathfrak{mp} \left( H' (\mathscr{O}_S)\right)$;

\item $\mathfrak{i}$ is an ideal for the algebroid bracket in $\widetilde{\mathscr{U}}_{2}\left(H \left( \mathscr{O}_S\right)\right)$.
\end{enumerate}
\end{lemma}

\begin{proof}
For local sections $X \in \mathfrak{mp} \left( H' \left( \mathscr{O}_S\right)\right)$ and $Y \in \mathfrak{sp}_{F} \left( H' \left( \mathscr{O}_S\right)\right)$, consider 
\[
D(Y) := [X, Y] -s([\overline{X}, Y]),
\]
 where $\overline{X}$ is the image of $X$ in $\mathfrak{sp} \left( H' \left( \mathscr{O}_S\right)\right)$ and $s$ is the splitting of Proposition \ref{prop:splitting2}. 
 By Theorem \ref{thm:unif1}, $\mathfrak{sp}_{F} \left( H' \left( \mathscr{O}_S\right)\right)$ is the kernel of the anchor map $\mathfrak{sp} \left( H' \left( \mathscr{O}_S\right)\right) \rightarrow \mathscr{T}_S$, hence an ideal in $\mathfrak{sp} \left( H' \left( \mathscr{O}_S\right)\right)$; so $[\overline{X}, Y]$ is a section of $\mathfrak{sp}_{F} \left( H' \left( \mathscr{O}_S\right)\right)$. Moreover, $D(Y)$ is continuous, $\mathscr{O}_S$-linear in $Y$, and takes values in $\mathscr{O}_S\,\bm{1}$, since its image in $\mathfrak{sp} \left( H' \left( \mathscr{O}_S\right)\right)$ vanishes.
 By the Jacobi identity, $D$ vanishes on brackets of sections of $\mathfrak{sp}_{F} \left( H' \left( \mathscr{O}_S\right)\right)$. Hence $D = 0$ by the density of 
 \[
 \left[\mathfrak{sp}_{F} \left( H' \left( \mathscr{O}_S\right)\right), \mathfrak{sp}_{F} \left( H' \left( \mathscr{O}_S\right)\right)\right] \mbox{ in } \mathfrak{sp}_{F} \left( H' \left( \mathscr{O}_S\right)\right),
 \]
  as in the proof of Proposition \ref{prop:splittings}(iii). It follows that $[X,Y]$ is a section of $\mathfrak{sp}_{F} \left( H' \left( \mathscr{O}_S\right)\right)$, hence part (i).
Similarly for part (ii).
\end{proof}

%%%%%%%%%%%%%%%%%%%%%%%%%%%%
%%%%%%%%%%%%%%%%%%%%%%%%%%%%
%%%%%%%%%%%%%%%%%%%%%%%%%%%%

\section{Intermezzo on metaplectic representations}
\label{sec:intermezzo}

Here we define some properties of 
  representations of the metaplectic algebra which will be used in the following sections.
Recall the Harish--Chandra pairs from \S\ref{sec:HCpairsmp} and $(\mathfrak{g}, K)$--modules as defined in \S\ref{sec:HCloc}.
All representations $V$ are assumed to be continuous with respect to the discrete topology on $V$.

\begin{definition}
\label{def:Heisenberg--admissible}
\begin{enumerate}[(i)]
\item An \textit{admissible $\mathfrak{mp} \left( H' (\mathbb{C})\right)$--representation} is an 
\[
\left(\mathfrak{mp} \left( H' \right), \, \mathbb{S}\mathrm{p}^+\left(H' \right)\right)\mbox{--module $V$}
\]
such that for every $(Z,F,L)\in \widehat{\mathcal{A}}_g$, the subgroup 
$\mathbb{S}\mathrm{p}^+_F\left(H' \right)\subset \mathbb{S}\mathrm{p}^+\left(H' \right)$ acts trivially on $V\,/\,\mathfrak{sp}_F \left( H' \right)(V)$.

\item  An \textit{admissible $\widetilde{\mathscr{U}}_{2}(H(\mathbb{C}))$--representation} is  an 
\[
\left( \widetilde{\mathscr{U}}_{2}(H), \, \mathbb{S}\mathrm{p}^+\left(H' \right)\ltimes  \mathscr{O}_1^\times \right)\mbox{--module $V$}
\]
such that for every $x\in \widehat{\mathcal{X}}_g$, the subgroup 
$\left(\mathbb{S}\mathrm{p}^+\left(H' \right)\ltimes \mathscr{O}_1^\times\right)_x$ in $\mathbb{S}\mathrm{p}^+\left(H' \right)\ltimes  \mathscr{O}_1^\times$ from \eqref{eq:stabxinXg} acts trivially on $V\,/\,\mathfrak{i}(x)(V)$.
\end{enumerate}
\end{definition}

In part (i), an $\mathfrak{mp} \left( H' (\mathbb{C})\right)$--module has the action of $\mathfrak{sp}_F \left( H' \right)$ via the splitting in Proposition \ref{prop:splitting2}(i). Similarly in part (ii), a $\widetilde{\mathscr{U}}_{2}(H)$--module has the action of $\mathfrak{i}$ via Proposition \ref{prop:splitting2}(ii).

An admissible $\widetilde{\mathscr{U}}_{2}(H(\mathbb{C}))$--representation with $H'(V)=V$ is in particular an admissible $\mathfrak{mp} \left( H' (\mathbb{C})\right)$--representation: for $(Z,F,L) \in \widehat{\mathcal{A}}_g$ take $x = (Z,F,L,h,q)$ with $h = 0$ in $H'/K$; then $\mathfrak{i}(x)(V) = \mathfrak{sp}_F \left( H' \right)(V)$ by Proposition \ref{prop:splitting2}, as $H'(V)=V$, and the stabilizer \eqref{eq:stabxinXg} at $x$ is $\mathbb{S}\mathrm{p}^+_F\left(H' \right)\times  \{1\}$, since $u_\rho = 0$.

We will also use the following induced characterization.
Consider that a representation $V$ of the Virasoro algebra  is said to be of central charge $c\in\mathbb{C}$ if the actions of the Virasoro operators on $V$ satisfy
\[
[L_p, L_q] = (p-q)\, L_{p+q} + \frac{c}{12}\, (p^3-p)\, \delta_{p+q,0}\, \mathrm{id}_V \qquad \mbox{for $p,q\in\mathbb{Z}$.}
\]
Recall the inclusion \eqref{eq:tauhat}.

\begin{definition}
\label{def:centralcharge}
A  representation $V$ of $\mathfrak{mp} \left( H' (\mathbb{C})\right)$ (respectively, $\widetilde{\mathscr{U}}_{2}(H(\mathbb{C}))$)  is said to be \textit{of central charge $c\in\mathbb{C}$} if 
the action of $\mathfrak{mp} \left( H' (\mathbb{C})\right)$ (resp., $\widetilde{\mathscr{U}}_{2}(H(\mathbb{C}))$) induces an action of $\mathrm{Vir}$ on $V$ of central charge $c$. 
\end{definition}

\subsection{Examples}
Examples of admissible representations are provided by:

\begin{lemma}
\label{lemma:exadm}
Heisenberg vertex algebras of arbitrary rank and lattice vertex (super)-algebras for lattices $\Gamma$ with non-degenerate forms $(,)$ are admissible $\widetilde{\mathscr{U}}_{2}(H(\mathbb{C}))$--representations (depending on $N\in \Gamma_\mathbb{C}$ with $(N,N)=\mathrm{rank}\,\Gamma$), hence admissible $\mathfrak{mp} \left( H' (\mathbb{C})\right)$--representations (canonically).
\end{lemma}

\begin{proof}
It is enough to consider the case of a Fock-type module $V=\oplus_\lambda V_\lambda$ with 
\[
V_\lambda={\mathscr{U}}(H)\otimes_{\mathscr{U}\left(H_+\right)} \mathbb{C}v_\lambda,
\]
where $H'_+(v_\lambda)=0$. In particular, 
\[
V_\lambda\cong \mathrm{Sym}(Z)\,v_\lambda \qquad\mbox{as vector spaces}
\]
for any Lagrangian $Z$ with $H'=Z\oplus H'_+$.
Heisenberg and lattice vertex (super)-algebras are indeed of this form.

We work for simplicity in the rank one case. For higher rank $d$, replace $H'$ here with $\Gamma_\mathbb{C}\otimes_\mathbb{C} H'$ for a complex vector space $\Gamma_\mathbb{C}$ with a non-degenerate form $(,)$. The action of $\widetilde{\mathscr{U}}_2(H)$ on $V$ is given by
\[
\bm{1}\mapsto d\,\mathrm{id}_V, \quad
:\! b_mb_n \!: \,\mapsto \sum_{i=1}^d :\! H^i_mH^i_n\!:, \quad
b_k \mapsto N_k
\]
for some orthonormal basis $\{H^1, \dots, H^d\}$ of $\Gamma_\mathbb{C}$ and some $N\in \Gamma_\mathbb{C}$ such that $(N,N)=d$. 
This action is independent of the choice of the orthonormal basis $\{H^i\}_i$ but depends on the choice of $N$. The action of $\rho\in\mathbb{S}\mathrm{p}^+\left(H' \right)$ is via $1\otimes \rho$. 

First, we show that $V$ has an action of $\mathbb{S}\mathrm{p}^+\left(H' \right)\ltimes \mathscr{O}_1^\times$. Elements $\rho$ in $\mathbb{S}\mathrm{p}^+\left(H' \right)$ have an adjoint action on $\mathscr{U}(H)$ and this induces an action on $V_\lambda$: $\rho(uv_\lambda)=\rho(u)v_\lambda$ for $u\in \mathscr{U}(H)$. This is well defined, since $\rho(H'_+)\subseteq H'_+$ by definition.
Elements $\exp(u)\in \mathscr{O}_1^\times$ act as the exp of the operator $u\in H'_+$. Since such $u$ is a sum of annihilation operators, $u$ is locally nilpotent, hence its exp is a finite sum and thus well defined. Since $\rho\circ u\circ \rho^{-1}=\rho(u)$, an arbitrary element $(\rho, \exp(u))$ in $\mathbb{S}\mathrm{p}^+\left(H' \right)\ltimes \mathscr{O}_1^\times$ acts as $\exp(u)\circ\rho$. This is compatible with the affine action on $H'$ in \S\ref{sec:Sp+Ox}.

The compatibility of the actions of $\widetilde{\mathscr{U}}_{2}(H(\mathbb{C}))$ and $\mathbb{S}\mathrm{p}^+\left(H' \right)\ltimes \mathscr{O}_1^\times$ on $V$ is immediate. Also, the differential of the action of $\mathbb{S}\mathrm{p}^+\left(H' \right)$ is the
action of $\mathfrak{sp}^+_b\left(H' \right)$ on $V$ via \eqref{eq:sp+inmp}: 
for $\rho = 1 + \epsilon X$ with $X \in \mathfrak{sp}^+_b\left(H' \right)$ and $\epsilon^2 = 0$, and $\widetilde{X}$ the normally ordered lift of $X$ via \eqref{eq:sp+inmp}, one has $\widetilde{X}(v_\lambda) = 0$ and $[\widetilde{X}, f] = X(f)$ for $f\in H'$, so 
\[
\rho(u v_\lambda) = (u + \epsilon[\widetilde{X}, u])v_\lambda = uv_\lambda + \epsilon(\widetilde{X}u - u\widetilde{X})v_\lambda = (1 + \epsilon\widetilde{X})(uv_\lambda);
\]
for $\mathscr{O}_1^\times$, the differential is the action of $H'_+$ via \eqref{eq:sp+H+inUleq2}.

Next, for $(Z,F,L)\in \widehat{\mathcal{A}}_g$, we show that $\mathbb{S}\mathrm{p}^+_F\left(H' \right)$ acts trivially on $V\,/\,\mathfrak{sp}_F \left( H' \right)(V)$.
Note that 
\begin{equation}
\label{eq:spFVisFV}
\mathfrak{sp}_F \left( H' \right)(V) = F(V).
\end{equation}
Indeed, $\mathfrak{sp}_F \left( H' \right)(V) \subseteq F(V)$ follows from 
the fact that elements of $\mathfrak{sp}_F$ act as in Proposition \ref{prop:splitting2}(i) via multiplication by elements $f\ell$ with $f\in F$ on the left and $\ell\in H'$; the opposite inclusion follows since $H'(V)=V$.
In rank $d>1$, where $F(V)$ here stands for $(\Gamma_{\mathbb C}\otimes_\mathbb{C} F)(V)$, one argues as
follows. Let $f\in F$, $v\in V$, and let $H^1$ be a unit vector in $\Gamma_{\mathbb C}$,
completed to an orthonormal basis $\{H^i\}_i$. Select $n>0$ such that $v$ does not
involve $H^i_{-n}$ for all $i$, and let $w:=\tfrac1n H^1_{-n}v$. Then $H^1_nw=v$ and
$H^i_nw=0$ for $i\neq1$, hence $f\,b_n\in\mathfrak{sp}_F(H')$ maps $w$ to
$(H^1\otimes f)v$. Since unit vectors span $\Gamma_\mathbb{C}$, this implies \eqref{eq:spFVisFV}.
Then, for $\rho\in \mathbb{S}\mathrm{p}^+_F\left(H' \right)$ and $z\in Z$, one has $\rho(z)\in z+F$, since $Z\subset F^\perp$.
Hence, for $z_1 \cdots z_n v_\lambda\in V_\lambda$, one has
\[
\rho\left( z_1 \cdots z_n v_\lambda \right) \equiv z_1 \cdots z_n v_\lambda \quad\mbox{mod $F(V)$,}
\]
since $\langle Z, F\rangle=0$. Thus $\rho$ acts trivially on $V\,/\,\mathfrak{sp}_F \left( H' \right)(V)$.

Finally, for $x=(Z, F, L, \overline{h}, q)\in \widehat{\mathcal{X}}_g$, we show that $\left(\mathbb{S}\mathrm{p}^+\left(H' \right)\ltimes \mathscr{O}_1^\times\right)_x$ acts trivially on $V\,/\,\mathfrak{i}(x)(V)$. 
Select a lift $h\in \overline{h}+K$.
By Proposition \ref{prop:splitting2}(ii), the action of $\mathfrak{i}(x)$ on $V$ is via elements
$\widehat{f}\widehat{\ell}$ and $\widehat{f}$ where 
\[
\widehat{f}:=f-\langle f,h\rangle\bm{1} \quad\mbox{and}\quad \widehat{\ell}:=\ell-\langle \ell,h\rangle\bm{1}
\]
for $f\in F$ and $\ell\in H'$. As for \eqref{eq:spFVisFV}, one has
\[
\mathfrak{i}(x)(V) = \widehat{F}(V)
\]
where $\widehat{F}\subset \widetilde{\mathscr{U}}_{2}(H)$ is the image of $F$ via $f\mapsto \widehat{f}$.
For rank one this is clear. For higher rank, $\widehat F(V)$ here stands for the span
of the vectors 
\[
\left(A\otimes f-(A,N)\langle f,h\rangle\right)v \qquad \mbox{with
$A\in\Gamma_{\mathbb C}$, $f\in F$, and $v\in V$,} 
\]
and the same argument as for
\eqref{eq:spFVisFV} applies, with $n$ such that also $\langle b_n,h\rangle=0$.
Then, for $(\rho, \exp(u_\rho))$ in $\left(\mathbb{S}\mathrm{p}^+\left(H' \right)\ltimes \mathscr{O}_1^\times\right)_x$ as in \eqref{eq:stabxinXg} and $z\in Z$,
one has 
\[
(\rho, \exp(u_\rho))(zv_\lambda)=\left(\exp(u_\rho)\circ\rho\right)(zv_\lambda)
=\rho(z)v_\lambda+\langle u_\rho, \rho(z)\rangle v_\lambda.
\]
Write $\rho(z)=z+f$ for some $f\in F$ and notice that $u_\rho\equiv (1-\rho)h$ mod $F$ by definition. 
One has
\[
\langle u_\rho, \rho(z)\rangle = \langle (1-\rho)h, z+f\rangle = \langle h, z\rangle + \langle h, f\rangle -\langle \rho(h), \rho(z)\rangle = \langle h, f\rangle.
\]
In the first equality we use that $\rho(z)\in Z\subset F^\perp$. It follows that
\[
(\rho, \exp(u_\rho))(zv_\lambda)=\left( z+\widehat{f}\right)v_\lambda \equiv zv_\lambda  \quad\mbox{mod $\widehat{F}(V)$.}
\]
More generally, for $z_1 \cdots z_n v_\lambda\in V_\lambda$, one has
\[
(\rho, \exp(u_\rho))\left( z_1 \cdots z_n v_\lambda \right) \equiv z_1 \cdots z_n v_\lambda \quad\mbox{mod $\widehat{F}(V)$,}
\]
since $\langle Z, \widehat{F}\rangle=0$. Thus $(\rho, \exp(u_\rho))$ acts trivially on $V\,/\,\mathfrak{i}(x)(V)$. The statement follows.
\end{proof}

%%%%%%%%%%%%%%%%%%%%%%%%%%%%
%%%%%%%%%%%%%%%%%%%%%%%%%%%%
%%%%%%%%%%%%%%%%%%%%%%%%%%%%

\section{Coinvariants on families of extended abelian varieties}
\label{sec:Vhat}

We define here spaces of coinvariants at extended abelian varieties and show how these yield twisted $\mathscr{D}$-modules on $\widehat{\mathcal{A}}_g$.
Using results from \S\ref{sec:isoH2}, we identify a multiple of the Atiyah algebra of the line bundle $\Lambda$ on $\widehat{\mathcal{A}}_g$ which acts on the sheaves of coinvariants and thus  determines their twisted $\mathscr{D}$-module structure. Similarly, we define and study twisted $\mathscr{D}$-modules of coinvariants on $\widehat{\mathcal{X}}_g$.

\subsection{Spaces of coinvariants at extended PPAVs}
Let $(Z, F, L)$ be 
a family of extended abelian varieties over a
smooth affine scheme $S=\mathrm{Spec}(R)$ for a $\mathbb{C}$-algebra $R$ as in \S\ref{sec:moduliofZFL}.
Consider a representation $V$ of $\mathfrak{mp} \left( H' (\mathbb{C})\right)$.
The action of $\mathfrak{mp} \left( H' (\mathbb{C})\right)$ on $V$ extends $R$-linearly to an action of $\mathfrak{mp} \left( H' (R)\right)$ on $V\otimes_{\mathbb{C}} R$. 
Composing with the splitting \sloppy \mbox{$\mathfrak{sp}_{F} \left(H'(R)\right)\hookrightarrow \mathfrak{mp} \left( H'(R) \right)$} from Proposition \ref{prop:splitting2}, 
one has an action of $\mathfrak{sp}_{F} \left(H' (R)\right)$ on $V\otimes_{\mathbb{C}} R$. 
We define the \textit{space of coinvariants} of  $V$ at $(Z, F, L)$ as 
\[
\widehat{\mathbb{V}}(V)_{(Z, F, L)} := V\otimes_{\mathbb{C}} R \,/\, \mathfrak{sp}_{F} \left(H' (R)\right) \left( V \otimes_{\mathbb{C}} R \right).
\]

\subsection{Sheaves of coinvariants on $\widehat{\mathcal{A}}_g$}
The spaces of coinvariants  induce sheaves of coinvariants as follows.
Let $({Z}, {F}, {L})\rightarrow S$ be a family of extended abelian varieties of dimension $g$ over a smooth base $S$.
Consider a representation $V$ of $\mathfrak{mp} \left( H' (\mathbb{C})\right)$. The action of $\mathfrak{mp} \left( H' (\mathbb{C})\right)$ on $V$
extends $\mathscr{O}_S$-linearly to an action of the sheaf of Lie algebras $\mathfrak{mp} \left( H' (\mathscr{O}_S)\right)$ on the sheaf  $V\otimes_{\mathbb{C}} \mathscr{O}_S$. 
This action restricts to an action of the sheaf of Lie algebras $\mathfrak{sp}_{F} \left(H' (\mathscr{O}_S)\right)$  on $V\otimes_{\mathbb{C}} \mathscr{O}_S$ via the splitting $\mathfrak{sp}_{F} \left(H' (\mathscr{O}_S)\right)\hookrightarrow \mathfrak{mp} \left( H' (\mathscr{O}_S)\right)$ as in \eqref{eq:AtiyahLambdadiagram}. 
We define the \textit{sheaf of coinvariants} of $V$  on $({Z}, {F}, {L})\rightarrow S$ as the quasi-coherent sheaf of $\mathscr{O}_S$-modules
\[
\widehat{\mathbb{V}}(V)_{({Z}, {F}, {L})\rightarrow S} := V\otimes_{\mathbb{C}} \mathscr{O}_S \,\big/\, \mathfrak{sp}_{F} \left(H' (\mathscr{O}_S)\right)  \left(V \otimes_{\mathbb{C}} \mathscr{O}_S \right).
\]
This gives rise to a quasi-coherent sheaf $\widehat{\mathbb{V}}(V)$ on $\widehat{\mathcal{A}}_g$.

\begin{theorem}
\label{thm:AtiyahactingonVhatA}
For a representation $V$ of $\mathfrak{mp} \left( H' (\mathbb{C})\right)$:
\begin{enumerate}[(i)]
\item the sheaf $\widehat{\mathbb{V}}(V)$ on $\widehat{\mathcal{A}}_g$ is a twisted $\mathscr{D}$-module;
\item  if $V$ has central charge $c$, then
the twisted $\mathscr{D}$-module structure is induced by an action of the Atiyah algebra $\frac{c}{2}\,\mathscr{F}_\Lambda$ on $\widehat{\mathbb{V}}(V)$. 
\end{enumerate}
\end{theorem}

\begin{proof}
For a family of extended abelian varieties of dimension $g$ over a smooth base $S$ with $S\rightarrow \widehat{\mathcal{A}}_g$ \'etale, the action of $\mathfrak{mp} \left( H' (\mathbb{C})\right)$  on $V$ 
and its action on $\widehat{\mathcal{A}}_g$ via the projection $\mathfrak{mp} \left( H' \right)\rightarrow \mathfrak{sp} \left( H' \right)$
induce an action of $\mathfrak{mp} \left( H' (\mathscr{O}_S)\right)$ on the sheaf $\widehat{\mathbb{V}}(V)$.
Explicitly, the action of $\mathfrak{mp} \left( H' (\mathscr{O}_S)\right)$ on $V\otimes_{\mathbb{C}} \mathscr{O}_S$ is given by
\[
X\cdot(v\otimes_{\mathbb{C}} f):=  (X\cdot v) f + v\otimes_{\mathbb{C}} (X\cdot f)
\]
for local sections $X\in \mathfrak{mp} \left( H' (\mathscr{O}_S)\right)$ and $v\otimes_{\mathbb{C}} f\in V\otimes_{\mathbb{C}} \mathscr{O}_S$.
While it is not $\mathscr{O}_S$-linear, this action extends the $\mathscr{O}_S$-linear action of $\mathfrak{sp}_{F} \left(H' (\mathscr{O}_S)\right)$  on $V\otimes_{\mathbb{C}} \mathscr{O}_S$.
From Lemma \ref{lemma:idealalgebroid}, $\mathfrak{sp}_F \left( H' (\mathscr{O}_S)\right)$ is an ideal for the algebroid bracket in $\mathfrak{mp} \left( H' (\mathscr{O}_S)\right)$.
Hence the Lie algebroid $\mathfrak{mp} \left( H' (\mathscr{O}_S)\right)/\mathfrak{sp}_F \left( H' (\mathscr{O}_S)\right)$ acts on $\widehat{\mathbb{V}}(V)$. Part (i) follows from \S\ref{sec:BBDmodquotstack}.

For part (ii), applying the central row of \eqref{eq:AtiyahLambdadiagram}, such a Lie algebroid is the Atiyah algebra $a\,\mathscr{F}_\Lambda$, for some $a\in\mathbb{C}$.
Since the central element $\bm{1}$ of $\mathfrak{mp} \left( H' \right)$ acts on $V$ as multiplication by the central charge $c$, the first row of \eqref{eq:AtiyahLambdadiagram} implies  $a=\frac{c}{2}$, hence the statement.
\end{proof}

\subsection{Sheaves of coinvariants on $\widehat{\mathcal{X}}_g$}
Let $({Z}, {F}, {L}, \overline{h}, q)\rightarrow S$ be a family as in \S\ref{sec:extunivPPAV} over a smooth base $S$ with $g$ fixed.
Consider a representation  $V$  of $\widetilde{\mathscr{U}}_{2}(H(\mathbb{C}))$. The action of $\widetilde{\mathscr{U}}_{2}(H(\mathbb{C}))$ on $V$
 extends $\mathscr{O}_S$-linearly to an action of the sheaf of Lie algebras $\widetilde{\mathscr{U}}_{2}(H(\mathscr{O}_S))$ on the sheaf $V\otimes_{\mathbb{C}} \mathscr{O}_S$.
Composing with the splitting 
\[
\mathfrak{i} \hookrightarrow \widetilde{\mathscr{U}}_{2}(H(\mathscr{O}_S))
\]
as in \eqref{eq:AtiyahThetadiagram}, one has an action of $\mathfrak{i}$ on $V \otimes_{\mathbb{C}} \mathscr{O}_S$. 
We define the \textit{sheaf of coinvariants} of $V$  on $({Z}, {F}, {L}, \overline{h}, q)\rightarrow S$ as the quasi-coherent sheaf of $\mathscr{O}_S$-modules
\[
\widehat{\mathbb{V}}(V)_{({Z}, {F}, {L}, \overline{h}, q)\rightarrow S} := V\otimes_{\mathbb{C}} \mathscr{O}_S \,\big/\, \mathfrak{i} \left(V \otimes_{\mathbb{C}} \mathscr{O}_S \right).
\]
This gives rise to a quasi-coherent sheaf $\widehat{\mathbb{V}}(V)$ on $\widehat{\mathcal{X}}_g$.

\begin{theorem}
\label{thm:AtiyahactingonVhatX}
For a representation  $V$  of $\widetilde{\mathscr{U}}_{2}(H(\mathbb{C}))$:
\begin{enumerate}[(i)]
\item the sheaf $\widehat{\mathbb{V}}(V)$ on $\widehat{\mathcal{X}}_g$ is a twisted $\mathscr{D}$-module;
\item if $V$ has central charge $c$, then
the twisted $\mathscr{D}$-module structure is induced by an action of the Atiyah algebra $\frac{c}{2}\,\mathscr{F}_\Xi$ on $\widehat{\mathbb{V}}(V)$. 
\end{enumerate}
\end{theorem}

\begin{proof}
For a family as in \S\ref{sec:extunivPPAV} over a smooth base $S$ with $S\rightarrow \widehat{\mathcal{X}}_g$ \'etale,
the action of $\widetilde{\mathscr{U}}_{2}(H(\mathbb{C}))$ on $V$  and its action on $\widehat{\mathcal{X}}_g$
via the projection \mbox{$\widetilde{\mathscr{U}}_{2}(H) \rightarrow \mathfrak{sp} \left(H' \right) \ltimes H' $}
 induce an action of $\widetilde{\mathscr{U}}_{2}(H(\mathscr{O}_S))$ on the sheaf $\widehat{\mathbb{V}}(V)$.
Explicitly, the action of $\widetilde{\mathscr{U}}_{2}(H(\mathscr{O}_S))$ on $V\otimes_{\mathbb{C}} \mathscr{O}_S$ is given by
\[
X\cdot(v\otimes_{\mathbb{C}} f):=  (X\cdot v) f + v\otimes_{\mathbb{C}} (X\cdot f)
\]
for local sections $X\in \widetilde{\mathscr{U}}_{2}(H(\mathscr{O}_S))$ and $v\otimes_{\mathbb{C}} f\in V\otimes_{\mathbb{C}} \mathscr{O}_S$.
From Lemma \ref{lemma:idealalgebroid}, $\mathfrak{i}$ is an ideal for the algebroid bracket in $\widetilde{\mathscr{U}}_{2}\left(H \left( \mathscr{O}_S\right)\right)$.
Hence the Lie algebroid $\widetilde{\mathscr{U}}_{2}(H(\mathscr{O}_S))/\mathfrak{i}$ acts on $\widehat{\mathbb{V}}(V)$. Part (i) follows.

For part (ii), applying the central row of \eqref{eq:AtiyahThetadiagram}, such a Lie algebroid is the Atiyah algebra $a\,\mathscr{F}_\Xi$, for some $a\in\mathbb{C}$.
Since the central element $\bm{1}$ of $\widetilde{\mathscr{U}}_{2}(H)$ acts on $V$ as multiplication by $c$, the first row of \eqref{eq:AtiyahThetadiagram} implies  $a=\frac{c}{2}$, hence the statement.
\end{proof}

%%%%%%%%%%%%%%%%%%%%%%%%%%%%
%%%%%%%%%%%%%%%%%%%%%%%%%%%%
%%%%%%%%%%%%%%%%%%%%%%%%%%%%

\section{Coinvariants on families of  abelian varieties}
\label{sec:sheafV}

Here we construct twisted $\mathscr{D}$-modules of coinvariants on $\mathcal{A}_g$ by descending the twisted $\mathscr{D}$-modules $\widehat{\mathbb{V}}(V)$ from \S\ref{sec:Vhat} along the projection $\widehat{\mathcal{A}}_g \rightarrow \mathcal{A}_g$. We proceed similarly on $\mathcal{X}_g$.

\subsection{The action of $\mathbb{S}\mathrm{p}^+\left(H' \right)$}
Let $V$ be a representation of $\mathfrak{mp} \left( H' (\mathbb{C})\right)$. 
The action of $\mathfrak{mp} \left( H' (\mathbb{C})\right)$ on $V$ and on $\widehat{\mathcal{A}}_g$ induces an action of $\mathfrak{mp} ( H' (\mathscr{O}_{\widehat{\mathcal{A}}_g}))$  on $\widehat{\mathbb{V}}(V)$ as in the proof of Theorem \ref{thm:AtiyahactingonVhatA}. 
Composing with the inclusion \eqref{eq:sp+inmp}, we deduce an action of $\mathfrak{sp}^+_b ( H'(\mathscr{O}_{\widehat{\mathcal{A}}_g}))$ 
on the sheaf $\widehat{\mathbb{V}}(V)$ on $\widehat{\mathcal{A}}_g$. 
For a point $(Z,F,L)$ in $\widehat{\mathcal{A}}_g$, recall the stabilizer $\mathbb{S}\mathrm{p}^+_F\left(H' \right)$ of $\mathbb{S}\mathrm{p}^+\left(H' \right)$ at $(Z,F,L)$ from \S\ref{sec:Sp+}.
Next, we show: 

\begin{proposition}
\label{prop:expactionSp+}
For an admissible $\mathfrak{mp} \left( H' (\mathbb{C})\right)$--representation $V$:
\begin{enumerate}[(i)]
\item The action of $\mathbb{S}\mathrm{p}^+\left(H' \right)$ on $V \otimes_\mathbb{C} \mathscr{O}_{\widehat{\mathcal{A}}_g}$ preserves 
\[
\mathfrak{sp}_{F} \left(H' \left(\mathscr{O}_{\widehat{\mathcal{A}}_g}\right)\right)  \left(V \otimes_{\mathbb{C}} \mathscr{O}_{\widehat{\mathcal{A}}_g} \right)
\]
and induces an $\mathbb{S}\mathrm{p}^+\left(H' \right)$-equivariant structure on $\widehat{\mathbb{V}}(V)$ whose differential is the $\mathfrak{sp}^+_b \left( H' \right)$-action. 

\item For $(Z,F,L)$ in $\widehat{\mathcal{A}}_g$, the subgroup $\mathbb{S}\mathrm{p}^+_F\left(H' \right)$ acts trivially on the fiber of $\widehat{\mathbb{V}}(V)$ at $(Z,F,L)$.
\end{enumerate}
\end{proposition}

\begin{proof}
For part (i), we apply Proposition \ref{prop:HClocvar2} for the Harish--Chandra pair $\left(\mathfrak{mp} \left( H' \right), \, \mathbb{S}\mathrm{p}^+\left(H' \right)\right)$ acting on $\widehat{\mathcal{A}}_g$, with 
\[
\mathrm{Ker} \,\alpha = \mathfrak{sp}_{F} \left( H' \left( \mathscr{O}_{\widehat{\mathcal{A}}_g}\right)\right).
\] 
For this, we verify that two splittings are $\mathbb{S}\mathrm{p}^+\left(H' \right)$-equivariant as required. 
First, 
\[
\mathrm{Lie}\left(\mathbb{S}\mathrm{p}^+\left(H' \right)\right) \hookrightarrow \mathfrak{mp} \left( H' \right)
\]
is \eqref{eq:sp+inmp}, which is $\mathbb{S}\mathrm{p}^+\left(H' \right)$-equivariant since $\left(\mathfrak{mp} \left( H' \right), \, \mathbb{S}\mathrm{p}^+\left(H' \right)\right)$ is a Harish--Chandra pair (\S\ref{sec:HCpairsmp}). Second, 
\[
\mathfrak{sp}_{F} \left( H' \left( \mathscr{O}_{\widehat{\mathcal{A}}_g}\right)\right) \hookrightarrow \mathfrak{mp} \left( H' \right) \,\widehat{\otimes}_\mathbb{C}\, \mathscr{O}_{\widehat{\mathcal{A}}_g}
\]
is given fiberwise by Proposition \ref{prop:splitting2}(i), and its $\mathbb{S}\mathrm{p}^+\left(H' \right)$-equivariance is seen as follows. 
By Theorem \ref{thm:unif1+}(i), the action of $\mathfrak{sp} \left( H' \right)$ on $\widehat{\mathcal{A}}_g$ is $\mathbb{S}\mathrm{p}^+\left(H' \right)$-equivariant, so $\mathrm{Ad}_\rho$ maps the stabilizer $\mathfrak{sp}_F \left( H' \right)$ onto $\mathfrak{sp}_{\rho(F)} \left( H' \right)$ for $\rho$ in $\mathbb{S}\mathrm{p}^+\left(H' \right)$. Since the splittings of Proposition \ref{prop:splitting2}(i) are the unique continuous ones, $\mathrm{Ad}_\rho$ maps the image of $\mathfrak{sp}_F \left( H' \right)$ in $\mathfrak{mp} \left( H' \right)$ onto the image of $\mathfrak{sp}_{\rho(F)} \left( H' \right)$, which is the required equivariance. Thus Proposition \ref{prop:HClocvar2}   implies the statement.

For part (ii), the fiber of $\widehat{\mathbb{V}}(V)$ at $(Z,F,L)$ is $V\, / \, \mathfrak{sp}_{F} \left(H' \right)(V)$, and $\mathbb{S}\mathrm{p}^+_F\left(H' \right)$ acts trivially on it by Definition \ref{def:Heisenberg--admissible}.
\end{proof}

\subsection{Sheaves of coinvariants on $\mathcal{A}_g$}
\label{sec:descentAg}
We are now ready for:

\begin{proof}[Proof of Theorems \ref{thm:maininitVA} and \ref{thm:mainVA}]
The statements follow from the Harish--Chandra localization in \S\ref{sec:geoHCp}.
Namely, recall from Theorem \ref{thm:unif1+}(i) that $\widehat{\mathcal{A}}_g$
carries a transitive action of $\mathfrak{sp} \left( H' \right)$ compatible with a  
transitive action of $\mathbb{S}\mathrm{p}^+\left(H' \right)$ on the fibers of  \mbox{$\widehat{\mathcal{A}}_g\rightarrow \mathcal{A}_g$}.
From Theorem \ref{thm:AtiyahactingonVhatA}(i), the sheaf $\widehat{\mathbb{V}}(V)$ on $\widehat{\mathcal{A}}_g$
carries an action of the Lie algebroid 
\[
P:= \mathfrak{mp} \left( H' \left(\mathscr{O}_{\widehat{\mathcal{A}}_g}\right)\right)/\mathfrak{sp}_F \left( H' \left(\mathscr{O}_{\widehat{\mathcal{A}}_g}\right)\right) \rightarrow \mathscr{T}_{\widehat{\mathcal{A}}_g}.
\]
From Proposition \ref{prop:expactionSp+}, $\widehat{\mathbb{V}}(V)$ is $\mathbb{S}\mathrm{p}^+\left(H' \right)$-equivariant with differential the $\mathfrak{sp}^+_b \left( H' \right)$-action.
In particular, $\widehat{\mathbb{V}}(V)$ is a $(\mathscr{U}, \mathbb{S}\mathrm{p}^+\left(H' \right))$--module for the Harish--Chandra algebra $\mathscr{U}:=\mathscr{U}(P)$. It follows from Proposition \ref{prop:HClocvar2} that $\widehat{\mathbb{V}}(V)$ is a twisted $\mathscr{D}$-module on $\widehat{\mathcal{A}}_g$ which descends to a twisted $\mathscr{D}$-module $\mathbb{V}'(V)$ on the quotient stack $[\widehat{\mathcal{A}}_g/\mathbb{S}\mathrm{p}^+\left(H' \right)]$. 
The action of $P$ on $\widehat{\mathbb{V}}(V)$ descends to an action of the Lie algebroid 
\begin{equation}
\label{eq:Pmodsp+b}
P\,/\,\mathfrak{sp}^+_b\left(H'\left(\mathscr{O}_{\widehat{\mathcal{A}}_g}\right) \right) \rightarrow \mathscr{T}_{[\widehat{\mathcal{A}}_g/\mathbb{S}\mathrm{p}^+\left(H' \right)]}
\end{equation}
on $\mathbb{V}'(V)$.

From Theorem \ref{thm:unif1+}(ii), the relative stabilizer of 
\[
\left[\widehat{\mathcal{A}}_g/\mathbb{S}\mathrm{p}^+\left(H' \right)\right] \rightarrow \mathcal{A}_g
\]
at $(Z,F,L)$ is $\mathbb{S}\mathrm{p}^+_F\left(H' \right)$, which acts trivially on the fiber of $\widehat{\mathbb{V}}(V)$ at $(Z,F,L)$ by Proposition \ref{prop:expactionSp+}(ii).
Applying \cite[Thm.~10.3]{alper2013good}, the sheaf $\mathbb{V}'(V)$ descends 
to a sheaf $\mathbb{V}(V)$ on $\mathcal{A}_g$. 

The Lie algebroid \eqref{eq:Pmodsp+b} descends as well. Its fiber at $(Z,F,L)$ is 
\[
\mathbb{C}\bm{1} \oplus \mathfrak{sp}\left(H' \right)/\left(\mathfrak{sp}_F\left(H' \right) + \mathfrak{sp}^+_b\left(H' \right)\right) = \mathbb{C}\bm{1} \oplus {T}_{(X,B)}\mathcal{A}_g
\]
by \eqref{eq:DCAg}. Here $\mathfrak{sp}^+_b\left(H' \right)$ surjects onto $\mathfrak{sp}^+\left(H' \right)/\mathfrak{sp}^+_F\left(H' \right)$, since by \eqref{eq:tangenttofiberAtildeA} the latter consists of bounded elements, hence 
\[
\mathfrak{sp}_F\left(H' \right) + \mathfrak{sp}^+_b\left(H' \right) = \mathfrak{sp}_F\left(H' \right) + \mathfrak{sp}^+\left(H' \right).
\]
For $\rho \in \mathbb{S}\mathrm{p}^+_F\left(H' \right)$, $\mathrm{Ad}_\rho$ fixes $\bm{1}$ and induces the identity on ${T}_{(X,B)}\mathcal{A}_g$, since $\widehat{\mathcal{A}}_g \rightarrow \mathcal{A}_g$ is $\mathbb{S}\mathrm{p}^+\left(H' \right)$-invariant. 
Hence, for a lift $\widetilde X\in\mathfrak{mp}(H')$ of a tangent
vector, $\mathrm{Ad}_\rho(\widetilde X)-\widetilde X=\chi\,\mathbf 1$, for
some $\chi\in\mathbb C$, modulo the images of $\mathfrak{sp}_F(H')$ and
$\mathfrak{sp}^+_b(H')$. To see that $\chi=0$, it is enough to apply this to a Fock module with vacuum
$v$ as in Lemma~\ref{lemma:exadm}. 
Since $\rho(v)=v$, one has
\[
(\mathrm{Ad}_\rho(\widetilde X)-\widetilde X)v=\rho(\widetilde Xv)-\widetilde Xv,
\]
which lies in $\mathfrak{sp}_F(H')(V)$, as $\rho$ acts trivially on the image of 
$\widetilde Xv$ in $V/\mathfrak{sp}_F(H')(V)$. Since $\mathfrak{sp}^+_b(H')$ annihilates $v$, one has $\chi\,v\in\mathfrak{sp}_F(H')(V)=F(V)$.
As $v\notin F(V)$, it follows that $\chi=0$.
Hence $\mathbb{S}\mathrm{p}^+_F\left(H' \right)$ acts trivially on the fibers of \eqref{eq:Pmodsp+b}, and \eqref{eq:Pmodsp+b} together with its action on $\mathbb{V}'(V)$ descends to 
a Lie algebroid on $\mathcal{A}_g$ with an action on $\mathbb{V}(V)$ by \cite[Thm.~10.3]{alper2013good}.
This Lie algebroid is an extension of $\mathscr{T}_{\mathcal{A}_g}$ by $\mathscr{O}_{\mathcal{A}_g}$, hence $\mathbb{V}(V)$ is a twisted $\mathscr{D}$-module on $\mathcal{A}_g$.
This proves Theorem~\ref{thm:maininitVA}.

Theorem \ref{thm:AtiyahactingonVhatA}(ii) identifies the action of $P$ with an action of the Atiyah algebra $\frac{c}{2}\,\mathscr{F}_\Lambda$
on the sheaf $\widehat{\mathbb{V}}(V)$ on $\widehat{\mathcal{A}}_g$. 
Thus from Proposition \ref{prop:HClocvar3} we deduce that the sheaf $\mathbb{V}'(V)$ on $[\widehat{\mathcal{A}}_g/\mathbb{S}\mathrm{p}^+\left(H' \right)]$ is equipped with an action of $\frac{c}{2}\,\mathscr{F}_\Lambda$.
Since $\Lambda$ is pulled back from $\mathcal{A}_g$, the descended sheaf $\mathbb{V}(V)$ on $\mathcal{A}_g$ carries an action of $\frac{c}{2}\,\mathscr{F}_\Lambda$.
This proves Theorem \ref{thm:mainVA}.
\end{proof}

\begin{example}
When $c=0$, the action of the Lie algebroid factors to an action of the tangent sheaf to $\mathcal{A}_g$ on the sheaf $\mathbb{V}(V)$ on $\mathcal{A}_g$. 
In this case, one can simply use Proposition \ref{prop:HClocvar1} in the above proof, instead of Proposition \ref{prop:HClocvar2}.
Hence for $c=0$, the sheaf $\mathbb{V}(V)$ is more simply a $\mathscr{D}$-module on $\mathcal{A}_g$. 
\end{example}

\subsection{The action of  $\mathbb{S}\mathrm{p}^+\left(H' \right)\ltimes  \mathscr{O}_1^\times$}
Next, we discuss how to construct similarly sheaves of coinvariants on $\mathcal{X}_g$. For this, we start by describing the action of $\mathbb{S}\mathrm{p}^+\left(H' \right)\ltimes  \mathscr{O}_1^\times$ on the sheaf of coinvariants on $\widehat{\mathcal{X}}_g$.

Let $V$ be a representation  of $\widetilde{\mathscr{U}}_{2}(H(\mathbb{C}))$. 
The action of $\widetilde{\mathscr{U}}_{2}(H(\mathbb{C}))$ on $V$ and on $\widehat{\mathcal{X}}_g$ induces an action of $\widetilde{\mathscr{U}}_{2}(H(\mathscr{O}_{\widehat{\mathcal{X}}_g}))$  on $\widehat{\mathbb{V}}(V)$ as in the proof of Theorem \ref{thm:AtiyahactingonVhatX}. 
Composing with the inclusion \eqref{eq:sp+H+inUleq2}, we deduce an action of $\mathfrak{sp}^+_b \left( H' \right)\ltimes H'_+$ on the sheaf $\widehat{\mathbb{V}}(V)$ on~$\widehat{\mathcal{X}}_g$.

\begin{proposition}
\label{prop:expactionSp+twisted}
For an admissible $\widetilde{\mathscr{U}}_{2}(H(\mathbb{C}))$--representation $V$:
\begin{enumerate}[(i)]
\item The action of $\mathbb{S}\mathrm{p}^+\left(H' \right)\ltimes \mathscr{O}_1^\times$ on $V \otimes_\mathbb{C} \mathscr{O}_{\widehat{\mathcal{X}}_g}$ preserves $\mathfrak{i}(V \otimes_\mathbb{C} \mathscr{O}_{\widehat{\mathcal{X}}_g})$ and induces an $(\mathbb{S}\mathrm{p}^+\left(H' \right)\ltimes \mathscr{O}_1^\times)$-equivariant structure on $\widehat{\mathbb{V}}(V)$ whose differential is the $(\mathfrak{sp}^+_b \left( H' \right) \ltimes H'_+)$-action;

\item For $x \in \widehat{\mathcal{X}}_g$, the subgroup $\left(\mathbb{S}\mathrm{p}^+\left(H' \right)\ltimes \mathscr{O}_1^\times\right)_x$ of \eqref{eq:stabxinXg} acts trivially on the fiber of $\widehat{\mathbb{V}}(V)$ at $x$.
\end{enumerate}
\end{proposition}

\begin{proof}
For part (i), we apply Proposition \ref{prop:HClocvar2} to the Harish--Chandra pair 
\[
\left( \widetilde{\mathscr{U}}_{2}(H), \, \mathbb{S}\mathrm{p}^+\left(H' \right)\ltimes  \mathscr{O}_1^\times \right)
\]
acting on $\widehat{\mathcal{X}}_g$, with $\mathrm{Ker}\, \alpha = \mathfrak{i}$. The splitting 
\[
\mathrm{Lie}\left(\mathbb{S}\mathrm{p}^+\left(H' \right)\ltimes \mathscr{O}_1^\times\right) \hookrightarrow \widetilde{\mathscr{U}}_{2}(H)
\] 
is \eqref{eq:sp+H+inUleq2}, which is equivariant by \S\ref{sec:HCpairsmp}. The splitting 
\[
\mathfrak{i} \hookrightarrow \widetilde{\mathscr{U}}_{2}(H) \,\widehat{\otimes}_\mathbb{C}\, \mathscr{O}_{\widehat{\mathcal{X}}_g}
\] 
is given fiberwise by Proposition \ref{prop:splitting2}(ii). By Theorem \ref{thm:unif2+}(i), the action of $\mathfrak{sp} \left( H' \right)\ltimes H'$ on $\widehat{\mathcal{X}}_g$ is $(\mathbb{S}\mathrm{p}^+\left(H' \right)\ltimes \mathscr{O}_1^\times)$-equivariant, so for $\gamma \in \mathbb{S}\mathrm{p}^+\left(H' \right)\ltimes \mathscr{O}_1^\times$ the automorphism $\mathrm{Ad}_\gamma$ maps the stabilizer $\mathfrak{i}(x)$ onto $\mathfrak{i}(\gamma(x))$. Since the splittings of Proposition \ref{prop:splitting2}(ii) are the unique continuous ones, $\mathrm{Ad}_\gamma$ maps the image of $\mathfrak{i}(x)$ in $\widetilde{\mathscr{U}}_{2}(H)$ onto the image of $\mathfrak{i}(\gamma(x))$, which is the required equivariance. Thus Proposition \ref{prop:HClocvar2} implies the statement. 

For part (ii), the fiber of $\widehat{\mathbb{V}}(V)$ at $x$ is $V\,/ \, \mathfrak{i}(x)(V)$, on which $\left(\mathbb{S}\mathrm{p}^+\left(H' \right)\ltimes \mathscr{O}_1^\times\right)_x$ acts trivially by Definition \ref{def:Heisenberg--admissible}.
\end{proof}

\subsection{Sheaves of coinvariants on $\mathcal{X}_g$}
\label{sec:descentXg}
Let $V$ be an admissible $\widetilde{\mathscr{U}}_{2}(H(\mathbb{C}))$--representation of central charge $c$.

\begin{proof}[Proof of Theorems \ref{thm:maininitVX} and \ref{thm:mainVX}]
The argument is similar to the proof of Theorems \ref{thm:maininitVA} and \ref{thm:mainVA}.
From Theorem \ref{thm:unif2+}(i), $\widehat{\mathcal{X}}_g$
carries a transitive action of \mbox{$\mathfrak{sp} \left( H' \right)\ltimes H'$} compatible with a  
transitive action of \sloppy \mbox{$\mathbb{S}\mathrm{p}^+\left(H' \right)\ltimes \mathscr{O}_1^\times$} on the fibers of  \mbox{$\widehat{\mathcal{X}}_g\rightarrow \mathcal{X}_g$}.
From Theorem \ref{thm:AtiyahactingonVhatX}(i), the sheaf $\widehat{\mathbb{V}}(V)$ on $\widehat{\mathcal{X}}_g$
carries an action of the Lie algebroid 
\[
P:= \widetilde{\mathscr{U}}_{2}\left( H \left(\mathscr{O}_{\widehat{\mathcal{X}}_g}\right)\right)/\mathfrak{i} \rightarrow \mathscr{T}_{\widehat{\mathcal{X}}_g}.
\]
From Proposition \ref{prop:expactionSp+twisted}, $\widehat{\mathbb{V}}(V)$ is $(\mathbb{S}\mathrm{p}^+\left(H' \right)\ltimes \mathscr{O}_1^\times)$-equivariant with differential the $(\mathfrak{sp}^+_b \left( H' \right) \ltimes H'_+)$-action.
In particular, $\widehat{\mathbb{V}}(V)$ is a $(\mathscr{U}, (\mathbb{S}\mathrm{p}^+\left(H' \right)\ltimes \mathscr{O}_1^\times))$--module for the Harish--Chandra algebra $\mathscr{U}:=\mathscr{U}(P)$. It follows from Proposition \ref{prop:HClocvar2} that $\widehat{\mathbb{V}}(V)$ is a twisted $\mathscr{D}$-module on $\widehat{\mathcal{X}}_g$ which descends to a twisted $\mathscr{D}$-module $\mathbb{V}'(V)$ on the quotient stack $[\widehat{\mathcal{X}}_g/(\mathbb{S}\mathrm{p}^+\left(H' \right)\ltimes \mathscr{O}_1^\times)]$. 
The action of $P$ on $\widehat{\mathbb{V}}(V)$ descends to an action of the Lie algebroid 
\begin{equation}
\label{eq:Pmodsp+bX}
P\,/\left(\mathfrak{sp}^+_b \left( H' \left( \mathscr{O}_{\widehat{\mathcal{X}}_g}\right)\right) \ltimes H'_+\right) \rightarrow \mathscr{T}_{[\widehat{\mathcal{X}}_g/(\mathbb{S}\mathrm{p}^+\left(H' \right)\ltimes \mathscr{O}_1^\times)]}
\end{equation}
on $\mathbb{V}'(V)$.

From Theorem \ref{thm:unif2+}(ii), the relative stabilizer of 
\[
\left[\widehat{\mathcal{X}}_g/(\mathbb{S}\mathrm{p}^+\left(H' \right)\ltimes \mathscr{O}_1^\times)\right] \rightarrow \mathcal{X}_g
\]
at $x$ is $\left(\mathbb{S}\mathrm{p}^+\left(H' \right)\ltimes \mathscr{O}_1^\times\right)_x$, which acts trivially on the fiber of $\widehat{\mathbb{V}}(V)$ at $x$ by Proposition \ref{prop:expactionSp+twisted}(ii).
Applying \cite[Thm.~10.3]{alper2013good}, the sheaf $\mathbb{V}'(V)$ descends 
to a sheaf $\mathbb{V}(V)$ on $\mathcal{X}_g$. 
The Lie algebroid \eqref{eq:Pmodsp+bX} descends as well.
Its fiber at $x$ is 
\[
\left(\mathbb{C}\bm{1} \oplus \mathfrak{sp}\left(H' \right)\ltimes H'\right)/\left(\mathfrak{i}(x) + \mathfrak{sp}^+_b\left(H' \right)\ltimes H'\right) = \mathbb{C}\bm{1} \oplus {T}_{\overline{x}}\mathcal{X}_g
\]
by \eqref{eq:DCXg}. 
Here we use that 
\[
\mathfrak{i}(x) + \mathfrak{sp}^+_b\left(H' \right)\ltimes H' = \mathfrak{i}(x) + \mathfrak{sp}^+\left(H' \right) \ltimes H'
\]
as in the proof of Theorems \ref{thm:maininitVA} and \ref{thm:mainVA}.
For $\gamma \in \left(\mathbb{S}\mathrm{p}^+\left(H' \right)\ltimes \mathscr{O}_1^\times\right)_x$, $\mathrm{Ad}_\gamma$ fixes $\bm{1}$ and induces the identity on ${T}_{\overline{x}}\mathcal{X}_g$, since $\widehat{\mathcal{X}}_g \rightarrow \mathcal{X}_g$ is $(\mathbb{S}\mathrm{p}^+\left(H' \right)\ltimes \mathscr{O}_1^\times)$-invariant. 
The same argument as in the proof of Theorems \ref{thm:maininitVA} and \ref{thm:mainVA} shows that
$\left(\mathbb{S}\mathrm{p}^+\left(H' \right)\ltimes \mathscr{O}_1^\times\right)_x$ acts trivially on the fibers of \eqref{eq:Pmodsp+bX}, and \eqref{eq:Pmodsp+bX} together with its action on $\mathbb{V}'(V)$ descends to 
a Lie algebroid on $\mathcal{X}_g$ with an action on $\mathbb{V}(V)$ by \cite[Thm.~10.3]{alper2013good}.
This Lie algebroid is an extension of $\mathscr{T}_{\mathcal{X}_g}$ by $\mathscr{O}_{\mathcal{X}_g}$, hence $\mathbb{V}(V)$ is a twisted $\mathscr{D}$-module on $\mathcal{X}_g$.
This proves Theorem~\ref{thm:maininitVX}.

Theorem \ref{thm:AtiyahactingonVhatX}(ii) identifies the action of $P$ with an action of the Atiyah algebra $\frac{c}{2}\,\mathscr{F}_\Xi$
on the sheaf $\widehat{\mathbb{V}}(V)$ on $\widehat{\mathcal{X}}_g$. 
Thus from Proposition \ref{prop:HClocvar3} we deduce that the sheaf $\mathbb{V}'(V)$ on $[\widehat{\mathcal{X}}_g/(\mathbb{S}\mathrm{p}^+\left(H' \right)\ltimes \mathscr{O}_1^\times)]$ is equipped with an action of $\frac{c}{2}\,\mathscr{F}_\Xi$.
Since $\Xi$ is pulled back from $\mathcal{X}_g$, the descended sheaf $\mathbb{V}(V)$ on $\mathcal{X}_g$ carries an action of $\frac{c}{2}\,\mathscr{F}_\Xi$.
This proves Theorem \ref{thm:mainVX}.
\end{proof}

%%%%%%%%%%%%%%%%%%%%%%%%%%%%
%%%%%%%%%%%%%%%%%%%%%%%%%%%%
%%%%%%%%%%%%%%%%%%%%%%%%%%%%

\section{Final remarks}
\label{sec:final}

\subsection{Comparison with classical coinvariants on curves}
For a vertex operator algebra $V$, the space of coinvariants at $(C,P,t)\in \widehat{\mathcal{M}}_g$  constructed in \cite{bzf} is the quotient of $V$ by the action of a Lie algebra 
$\mathcal{L}_{C\setminus P}(V)$ depending on both $V$ and the open subset $C\setminus P\subset C$. 
The kernel of the map $\mathrm{Witt}\rightarrow T_{(C,P,t)}(\widehat{\mathcal{M}}_g)$ is a Lie subalgebra of $\mathcal{L}_{C\setminus P}(V)$, but in general is not equal to it.
It is shown in \cite{bzf} how the construction globalizes over $\widehat{\mathcal{M}}_g$ and yields an $\mathbb{A}\mathrm{ut}\left(\mathscr{O}\right)$-equivariant twisted $\mathscr{D}$-module on $\widehat{\mathcal{M}}_g$. This thus descends along the principal $\mathbb{A}\mathrm{ut}\left(\mathscr{O}\right)$-bundle $\widehat{\mathcal{M}}_g \rightarrow\mathcal{M}_{g,1}$. Furthermore, since the resulting fibers at $(C,P)$ and $(C,Q)$ in $\mathcal{M}_{g,1}$ are canonically isomorphic for all $P,Q\in C$, the sheaf is constant along the fibers of $\mathcal{M}_{g,1}\rightarrow \mathcal{M}_g$ and thus descends to a twisted $\mathscr{D}$-module on $\mathcal{M}_g$.

For a representation $V$ of $\mathfrak{mp} \left( H' (\mathbb{C})\right)$, the space of coinvariants at $(Z,F,L)\in \widehat{\mathcal{A}}_g$ from \S\ref{sec:Vhat} is the quotient of $V$ by the action of the Lie algebra $\mathfrak{sp}_{F} \left(H'\right)$. Contrary to $\mathcal{L}_{C\setminus P}(V)$, 
the Lie algebra $\mathfrak{sp}_{F} \left(H'\right)$ is independent of $V$ and equals the kernel of the analogous map $\mathfrak{sp} \left(H'\right)\rightarrow T_{(Z,F,L)}(\widehat{\mathcal{A}}_g)$, hence it is the smallest Lie algebra whose coinvariants yield twisted $\mathscr{D}$-modules on $\widehat{\mathcal{A}}_g$.
For this, it is natural to determine a Lie algebra containing as Lie subalgebras both $\mathfrak{sp}_{F} \left(H'\right)$ and  $\mathcal{L}_{C\setminus P}(V)$ at points in the Jacobian locus. We plan to present such an extension in a follow-up work.

The spaces of coinvariants for $\mathcal{L}_{C\setminus P}(V)$ are known to have finite dimension when $V$ satisfies some finiteness and semisimplicity conditions \cite{an1, dgt2}, and thus give rise to vector bundles of coinvariants on $\mathcal{M}_g$ whose Chern classes are determined by their twisted $\mathscr{D}$-module structure \cite{dgt}. 
It would be interesting to determine a similar result on spaces of coinvariants on $\mathcal{A}_g$. The mentioned extension of $\mathfrak{sp}_{F} \left(H'\right)$ will provide a better context for investigating this property, which we intend to explore accordingly.

\subsection{Failing descent for arbitrary modules}
Spaces of coinvariants for $\mathcal{L}_{C\setminus P}(V)$ have been constructed also for the action of $\mathcal{L}_{C\setminus P}(V)$ on an arbitrary vertex operator algebra module \cite{bzf}. These give rise  to twisted $\mathscr{D}$-modules on $\mathcal{M}_{g,1}$ \cite[\S 8.7]{dgt2}, which in general do not descend to $\mathcal{M}_g$.
In fact, the descent along $\mathcal{M}_{g,1}\rightarrow \mathcal{M}_g$ is only possible after tensoring the sheaf of coinvariants by an appropriate power of the relative cotangent line bundle
to offset the variation of the spaces of coinvariants along the fibers of $\mathcal{M}_{g,1}\rightarrow \mathcal{M}_g$ (as in \cite[\S 8.7]{dgt2}).
Similarly, the twisted $\mathscr{D}$-modules of coinvariants from arbitrary vertex operator algebra modules constructed as in \S\ref{sec:Vhat} in general  do not descend to $\mathcal{A}_g$.
As the relative $H^2$-space for the map $\widehat{\mathcal{A}}_g\rightarrow \mathcal{A}_g$ is trivial (Proposition \ref{prop:H2AXMP}), tensoring by a line bundle will not help here.
It would be interesting to find a finite-dimensional extension of the moduli space $\mathcal{A}_g$ where the sheaves of coinvariants from arbitrary vertex operator algebra modules could descend from $\widehat{\mathcal{A}}_g$.

\section*{Acknowledgments} 
The author acknowledges partial support by NSF award DMS-2404896 and a Simons Foundation's Travel Support for Mathematicians gift during the preparation of this manuscript.
The author would like to thank the two anonymous referees for their careful reading and comments, which helped to improve the manuscript, and Filippo Viviani for pointing out the results in \cite{alper2013good}. 
During the last revision, the author used Anthropic's Claude (Fable 5 models) as a
critical reader. It pointed out errors 
in Figure \ref{fig:bigLiesquare}, 
in \eqref{eq:ix} and \eqref{eq:stabxinXg}, 
in the definition of the cocycles in \S\ref{sec:cohLie},
in Proposition~\ref{prop:splitting2}, 
in the diagrams \eqref{eq:AtiyahLambdadiagram} and \eqref{eq:AtiyahThetadiagram},
and in the proofs of \eqref{eq:H2sp}--\eqref{eq:H2spH} and of Lemma \ref{lemma:exadm}.
It also pointed out gaps in the original proof of Theorem~\ref{thm:canisoH2spAX} 
and in the descent of \S\ref{sec:sheafV},
and an inconsistency of
signs between Theorems~\ref{thm:mainVA} and~\ref{thm:mainVX}. 
It stressed the role of the boundedness condition
\eqref{eq:rhohasak} and of the torsor in \S\ref{sec:extunivPPAV}. The proof of \eqref{eq:lanu}, 
the argument closing the descent in \S\ref{sec:descentAg},
and the corrections to the proof of Proposition~\ref{prop:cohLiespH} were worked out in exchanges with it. Several
of its suggestions were incorrect and were discarded or corrected along the way. All
statements and proofs have been verified by the author, who is solely responsible for
the content of the paper.

%%%%%%%%%%%%%%%%%%%%%%%%%%%%%%%%%%%%%%
\bibliographystyle{alphanumN}
\bibliography{Biblio}
%%%%%%%%%%%%%%%%%%%%%%%%%%%%%%%%%%%%%%

\end{document}